\documentclass[production,12pt]{style}\usepackage{latexsym}\usepackage{amsmath}\usepackage{amsthm}\usepackage{xspace}\usepackage{enumerate}\usepackage{amsfonts}\usepackage{amscd}\usepackage{syntonly}\usepackage{amssymb}\usepackage{fancyhdr}\usepackage{epsfig}

\newtheorem{thm}{Theorem}
\newtheorem{prop}[thm]{Proposition}
\newtheorem{lemma}[thm]{Lemma}
\newtheorem{claim}{Claim}
\newtheorem{defi}[thm]{Definition}
\newtheorem{rmk}[thm]{Remark}

\newtheorem{cor}[thm]{Corollary}
\newtheorem{xpl}[thm]{Example}

\newenvironment{pf}[1][Proof.]{\noindent \emph{#1.}}{}
\newenvironment{enui}{\begin{enumerate}[(i)]}{\end{enumerate}}
\newenvironment{enua}{\begin{enumerate}[(a)]}{\end{enumerate}}

\newcommand\const{\equiv}
\newcommand\op{{\operatorname{op}}}
\newcommand\disj{\coprod}

\newcommand\im{{\operatorname{im}}}

\newcommand\Aut{{\operatorname{Aut}}}

\newcommand\id{{\operatorname{id}}}

\newcommand\nn{{\nonumber}}
\newcommand\wt[1]{{\widetilde{#1}}}
\newcommand{\BAR}[1]{{\overline{#1}}}
\newcommand\al{{\alpha}}
\newcommand\be{\beta}
\newcommand\ga{\gamma}

\newcommand\de{\delta}
\newcommand\eps{\varepsilon}

\newcommand\Om{\Omega}
\newcommand\om{\omega}
\newcommand\lam{\lambda}
\newcommand\Si{\Sigma}

\newcommand\ze{\zeta}
\renewcommand\phi{\varphi}
\newcommand\na{\nabla}

\newcommand\U{\mathcal{U}}
\newcommand\V{\mathcal{V}}

\newcommand{\N}{\mathbb{N}}
\newcommand{\Z}{\mathbb{Z}}

\newcommand{\R}{\mathbb{R}}

\newcommand\C{\mathbb C}

\newcommand\D{\mathbb{D}}

\newcommand\g{\mathfrak g}
\newcommand\A{\mathcal A}
\newcommand\Lie{\operatorname{Lie}}

\newcommand\sub{\subseteq}

\newcommand\x{\times}
\newcommand\wo{\setminus}
\newcommand\one{\mathbf{1}}
\newcommand\iso{\cong}
\newcommand\dd{\partial}

\newcommand\lan{\langle}
\newcommand\ran{\rangle}

\newcommand\loc{{\operatorname{loc}}}

\newcommand\M{\mathcal{M}}
\newcommand\MM{\wt{\mathcal{M}}}

\newcommand\ME{\mathcal{M}_{<\infty}}

\newcommand\barev{\BAR{\operatorname{ev}}}

\newcommand\W{{\bf W}}

\newcommand\WWW{\wt{\mathcal{W}_0}}
\newcommand\WWP{\mathcal{W}}
\newcommand\WWWP{\wt{\mathcal{W}}}

\newcommand\z{{\bf z}}

\newcommand\Isom{\operatorname{Isom}}

\newcommand\wrt{w.r.t.~}

\newcommand\Emin{\operatorname{E}_{\min}}

\newcommand\hhat[1]{\widehat{#1}}
\newcommand\noi{\noindent}

\newcommand\new{{\operatorname{new}}}

\newcommand\del{{\partial}}

\newcommand\Qk{{\operatorname{Q}\!\kappa}}

\newcommand\E{E}

\newcommand\PSL{\operatorname{PSL}(2,\C)}
\newcommand\TR{{\mathcal{T}_{\R^2}}}
\newcommand\CP{\mathbb{C}\!\operatorname{P}}

\newcommand\UU{\mathcal{U}}
\newcommand\La{\Delta}

\newcommand\ad{\operatorname{ad}}

\title{A Quantum Kirwan Map, II: Bubbling}

\author{Fabian Ziltener (Korea Institute for Advanced Study)}

\begin{document}

{\bf This article has been merged with arXiv:0905.4047. The new article is:\\

A Quantum Kirwan Map: Bubbling and Fredholm Theory for Symplectic Vortices over the Plane, arXiv:1209.5866}\\

\begin{abstract} Consider a Hamiltonian action of a compact connected Lie group $G$ on an aspherical symplectic manifold $(M,\om)$. Under suitable assumptions, counting gauge equivalence classes of (symplectic) vortices on the plane $\R^2$ conjecturally gives rise to a quantum deformation $\Qk_G$ of the Kirwan map. 

This is the second of a series of articles, whose goal is to define $\Qk_G$ rigorously. The main result is that every sequence of vortices with uniformly bounded energies has a subsequence that converges to a genus 0 stable map of vortices on $\R^2$ and holomorphic spheres in the symplectic quotient. 

Potentially, the map $\Qk_G$ can be used to compute the quantum cohomology of many symplectic quotients. Conjecturally it also gives rise to quantum generalizations of non-abelian localization and abelianization. 
\end{abstract} 

\maketitle
\tableofcontents

\section{Main result}\label{sec:main}
Let $(M,\om)$ be a symplectic manifold and $G$ a compact connected Lie group with Lie algebra $\g$. We fix a Hamiltonian action of $G$ on $M$ and an (equivariant) moment map $\mu:M\to\g^*$. Throughout this article, we make the following standing assumption:\\

\noindent {\bf Hypothesis (H):} \textit{$G$ acts freely on $\mu^{-1}(0)$ and the moment map $\mu$ is proper.}\\

Then the symplectic quotient $\BAR M:=\mu^{-1}(0)/G$ is well-defined, smooth and closed (i.e., compact and without boundary). Based on ideas by D.~A.~Salamon, in \cite{ZiFredholm} I conjectured that under suitable assumptions there exists an algebra homomorphism $\Qk_G$ from the equivariant cohomology of $M$, tensored with the equivariant Novikov ring, to the quantum cohomology of $\BAR M$.

The idea of proof of the conjecture is to define $\Qk_G$ by counting symplectic vortices over $\R^2$. Once established, this should allow to compute the quantum cohomology of many symplectic quotients (e.g.~those arising from suitable linear torus actions on a symplectic vector space). Based on the map $\Qk_G$, one can formulate quantum versions of non-abelian localization and abelianization, see \cite{WZ}. 

The present article is the second of a series of papers, whose goal is to define $\Qk_G$ rigorously. The main result is that every sequence of vortices with uniformly bounded energy has a subsequence that converges to a new kind of stable map, consisting of vortices on $\R^2$ and holomorphic spheres in the symplectic quotient. 

To explain this, we recall the symplectic vortex equations: Let $J$ be an $\om$-compatible $G$-invariant almost complex structure on $M$, $\lan\cdot,\cdot\ran_\g$ an invariant inner product on $\g$, and $(\Si,\om_\Si,j)$ a (smooth) real surface equipped with an area form and a compatible complex structure. For every principal bundle $P$ over $\Si$ we denote by $\A(P)$ the affine space of connections on $P$, and by $C^\infty_G(P,M)$ the set of smooth equivariant maps from $P$ to $M$. We denote
\begin{eqnarray}\nn\WWWP(\Si)&:=\big\{w:=(P,A,u)\,\big|&P\textrm{ principal }G\textrm{-bundle over }\Si,\\
\nn&&A\in\A(P),\,u\in C^\infty_G(P,M)\big\}.\end{eqnarray}
The \emph{symplectic vortex equations} are the equations 
\begin{eqnarray}\label{eq:BAR dd J A u}\bar\dd_{J,A}(u)&=&0,\\
\label{eq:F A mu}F_A+(\mu\circ u)\om_\Si&=&0
\end{eqnarray}
for a triple $(P,A,u)\in\WWWP(\Si)$. Here for a point $x\in M$ we denote by $L_x:\g\to T_xM$ the infinitesimal action at $x$. By $\bar\dd_{J,A}(u)$ we mean the complex anti-linear part of $d_Au:=du+L_uA$, which we think of as a one-form on $\Si$ with values in the complex vector bundle $(u^*TM)/G\to\Si$. We view the curvature $F_A$ of $A$ as a two-form on $\Si$ with values in the adjoint bundle $\g_P:=(P\x\g)/G\to\Si$. Finally, identifying $\g^*$ with $\g$ via $\lan\cdot,\cdot\ran_\g$, we view $\mu\circ u$ as a section of $\g_P$. The vortex equations (\ref{eq:BAR dd J A u},\ref{eq:F A mu}) were discovered by K.~Cieliebak, A.~R.~Gaio and D.~A.~Salamon \cite{CGS}, and independently by I.~Mundet i Riera \cite{MuPhD,MuHam}. 

Two elements $w,w'\in\WWWP(\Si)$ are called \emph{equivalent} iff there exists an isomorphism $\Phi:P'\to P$ of principal $G$-bundles which descends to the identity on $\Si$, and satisfies
\begin{equation}\nn\Phi^*(A,u):=(A\circ d\Phi,u\circ\Phi)=(A',u').\end{equation}
In this case we write $w\sim w'$. We define
\begin{equation}\label{eq:WW}\WWP(\Si):=\WWWP(\Si)/\sim.\end{equation}
The equations (\ref{eq:BAR dd J A u},\ref{eq:F A mu}) are invariant under equivalence. A \emph{(symplectic) vortex (on $\Si$)} is by definition an equivalence class $W\in\WWP(\Si)$, such that every representative of $W$ satisfies (\ref{eq:BAR dd J A u},\ref{eq:F A mu}). We define the \emph{energy density} of a class $W\in\WWP(\Si)$ to be 
\begin{equation}\label{eq:e W}e_W:=\frac12\big(|d_Au|^2+|F_A|^2+|\mu\circ u|^2\big),\end{equation}
where $w:=(P,A,u)$ is any representative of $W$. (Here the norms are induced by the Riemannian metrics $\om_\Si(\cdot,j\cdot)$ on $\Si$ and $\om(\cdot,J\cdot)$ on $M$, and by $\lan\cdot,\cdot\ran_\g$. This definition does not depend on the choice of $w$.) Vortices are absolute minimizers of the \emph{(Yang-Mills-Higgs) energy functional}
\[E:\WWP(\Si)\to[0,\infty],\quad E(W):=\int_\Si e_W\om_\Si\]
in a given second equivariant homology class. (See \cite{CGS}. Here we assume that $\Si$ is closed, and vortices in the given class exist.) Consider $\Si:=\R^2$, equipped with the standard area form $\om_{\R^2}:=\om_0$ and complex structure $j:=i$. We define 
\begin{equation}\label{eq:M}\MM:=\big\{(P,A,u)\in\WWWP(\R^2)\,\big|\,(\ref{eq:BAR dd J A u},\ref{eq:F A mu})\big\},\quad\M:=\MM/\sim.
\end{equation}
Assume that $(M,\om)$ is \emph{(symplectically) aspherical}, i.e.,
\[\int_{S^2}u^*\om=0,\quad\forall u\in C^\infty(S^2,M).\]
Then heuristically, the main result of this article provides a compactification for the space of all classes in $\M$ with fixed finite energy $E>0$. There are three sources of non-compactness of this space: Consider a sequence $W_\nu\in\M$, $\nu\in\N$, of classes of energy $E$. In the limit $\nu\to\infty$, the following scenarios (and combinations) may happen:

1.~The energy densitiy of $W_\nu$ blows up at some point in $\R^2$.

2.~There exists a number $r>0$ and a sequence of points $z_\nu\in\R^2$ that converges to $\infty$, such that the energy density of $W_\nu$ on the ball $B_r(z_\nu)$ is bounded above and below by some positive constants. 

3.~The energy densities converge to 0, i.e., the energy is spread out more and more.

In case 1, by rescaling $W_\nu$ around the bubbling point, in the limit $\nu\to\infty$, we obtain a non-constant $J$-holomorphic map from $\R^2$ to $M$. Using removal of singularity, this is excluded by the asphericity condition. In case 2, we pull $W_\nu$ back by the translation $z\mapsto z+z_\nu$, and in the limit $\nu\to\infty$, obtain a vortex on $\R^2$. Finally, in case 3, we ``zoom out'' more and more. In the limit $\nu\to\infty$ and after removing the singularity at $\infty$, we obtain a pseudo-holomorphic map from $S^2$ to the symplectic quotient $\BAR M=\mu^{-1}(0)/G$.

Hence the limit object is a stable map, consisting of vortices on $\R^2$ and pseudo-holomorphic spheres in $\BAR M$ (and marked points). This notion and convergence against a stable map are made precise in Section \ref{sec:stable}. 

Here an important difference to Gromov-convergence for pseudo-holomorphic maps is the following: Although the vortex equations are invariant under under all orientation preserving isometries of $\Si$, only translations on $\R^2$ are allowed as reparametrizations used to obtain a vortex on $\R^2$ in the limit. Hence we disregard some symmetries of the equations. The reasons are that otherwise the reparametrization group would not act with finite isotropy on the set of simple stable maps, and that there is no suitable evaluation map on the set of vortices which is invariant under rotation. (See Remarks \ref{rmk:Isom +} and \ref{rmk:conv rot} below.) 

In order to state the main result, we also need the following. We call the quadruple $(M,\om,\mu,J)$ \emph{(equivariantly) convex at $\infty$} iff there exists a proper $G$-invariant function $f\in C^\infty(M,[0,\infty))$ and a constant $C\in[0,\infty)$ such that 
\[\om(\na_v\na f(x),Jv)-\om(\na_{Jv}\na f(x),v)\geq0,\quad df(x)JL_x\mu(x)\geq0,\]
for every $x\in f^{-1}([C,\infty))$ and $0\neq v\in T_xM$. Here $\na$ denotes the Levi-Civita connection of the metric $\om(\cdot,J\cdot)$.

We define the \emph{image} of a class $W\in\WWP(\Si)$ to be the set of orbits of $u(P)$, where $(P,A,u)$ is any representative of $W$. This is a subset of $M/G$. We endow $M/G$ with the quotient topology. We are now able to formulate the main result. 
\begin{thm}[Bubbling]\label{thm:bubb}  Assume that hypothesis (H) is satisfied, $(M,\om)$ is aspherical, and $(M,\om,\mu,J)$ is convex at $\infty$. Let $k\in\N_0:=\{0,1,\ldots\}$, and for $\nu\in\N$ let $W_\nu\in\M$ be a vortex and $z_1^\nu,\ldots,z_k^\nu\in\R^2$ be points. Suppose that the closure of the image of each $W_\nu$ is compact, and
\begin{eqnarray}\nn&\E(W_\nu)>0,\,\forall\nu\in\N,\quad\sup_{\nu\in\N}\E(W_\nu)<\infty,&\\
\label{eq:limsup z i nu z j nu}&\limsup_{\nu\to\infty}|z_i^\nu-z_j^\nu|>0,\quad\textrm{if }i\neq j.&
\end{eqnarray}
Then there exists a subsequence of $\big(W_\nu,z_0^\nu:=\infty,z_1^\nu,\ldots,z_k^\nu\big)$ that converges to some genus 0 stable map of vortices on $\R^2$ and pseudo-holomorphic spheres in $\BAR M$ with $k+1$ marked points.
\end{thm}
(The reasons for the additional marked point $z_0^\nu=\infty$ are explained in Remarks \ref{rmk:al 0 map} and \ref{rmk:al 0 conv} below.) The relevance of Theorem \ref{thm:bubb} is the following. There is an evaluation map from the set 
\begin{equation}\label{eq:ME}\ME:=\big\{W\in\M\,\big|\,\BAR{\textrm{image}(W)}\textrm{ compact, }E(W)<\infty\big\}.
\end{equation}
to the product of $\BAR M$ and the Borel construction for the action of $G$ on $M$. (See the forth-coming article \cite{ZiConsEv}.) The structure constants of the quantum Kirwan map $\Qk_G$ will be defined by pulling back cohomology classes via this evaluation map and integrating them over the space of vortices representing a fixed second equivariant homology class. 

To make this rigorous, one has to pass to some finite-dimensional approximation of the Borel construction and show that the evaluation map is a pseudo-cycle. The proof of this will rely on Theorem \ref{thm:bubb}. 

Secondly, Theorem \ref{thm:bubb} will also be used to prove that $\Qk_G$ is a ring homomorphism. (See the argument outlined in \cite{ZiFredholm}.) 

The proof of the theorem combines Gromov compactness for pseudo-holomorphic maps with Uhlenbeck compactness. It relies on work \cite{CGMS,GS} by K.~Cieliebak, R.~Gaio, I.~Mundet i Riera, and D.~A.~Salamon. The idea is the following. In order to capture all the energy, we ``zoom out rapidly'', i.e., rescale the vortices so much that the energies of the rescaled vortices are concentrated near the origin in $\R^2$. Now we ``zoom back in'' in such a way that we capture the first bubble, which may either be a vortex on $\R^2$ or a sphere in $\BAR M$. In the first case we are done. In the second case we ``zoom in'' further, to obtain a finite number of vortices and spheres that are attached to the first bubble. Iterating this procedure, we construct the whole stable map. 

The proof involves generalizations of results for pseudo-holomorphic maps to vortices: a bound on the energy density of a vortex, quantization of energy, compactness with bounded derivatives, and hard and soft rescaling. The proof that the bubbles connect and no energy is lost between them, uses an isoperimetric inequality for the invariant symplectic action functional, proved in \cite{ZiA}, based on a version of the inequality by R.~Gaio and D.~A.~Salamon \cite{GS}. 

Another crucial point is that when ``zooming out'', no energy is lost locally in $\R^2$ in the limit. This relies on an upper bound of the ``moment-map component'' of a vortex, due to R.~Gaio and D.~A.~Salamon. 
\subsection*{Related work and remarks}\label{subsec:rmks relation}
Assume that $\Si$ is closed, (H) holds, and $M$ is symplectically aspherical and equivariantly convex at $\infty$. In this case, in \cite[Theorem 3.4]{CGMS}, K.~Cieliebak et al.~proved compactness of the space of vortices with energy bounded above by a fixed constant. Assume that $M$ and $\Si$ are closed. Then in \cite[Theorem 4.4.2]{MuPhD} I.~Mundet i Riera compactified the space of bounded energy vortices with fixed complex structure on $\Si$. Assuming also that $G:=S^1$, this was extended by I.~Mundet i Riera and G.~Tian in \cite[Theorem 1.4]{MT} to the situation of varying complex structure. This work is based on a version of Gromov-compactness for continuous almost complex structure, proved by S.~Ivashkovich and V.~Shevchishin in \cite{IS}. 

In \cite[Theorem 1.8]{Ott} A.~Ott compactified the space of bounded energy vortices in a different way, for a general Lie group, and closed $M$ and $\Si$, the latter with fixed complex structure. He used the approach to Gromov-compactness by D.~McDuff and D.~A.~Salamon in \cite{MS}. In the case in which $\Si$ is an infinite cylinder, equipped with the standard area form and complex structure, the compactification was carried out by U.~Frauenfelder in \cite[Theorem 4.12]{FrPhD}.

In \cite{GS} R.~Gaio and D.~A.~Salamon investigated the vortex equations with area form $C\om_\Si$ in the limit $C\to\infty$. Here $\Si$ is a closed surface equipped with a fixed area form $\om_\Si$. They proved that three types of objects may bubble off: a holomorphic sphere in $\BAR M$, vortices on $\R^2$, and holomorphic spheres in $M$. (See the proof of \cite[Theorem A]{GS}.)

In some earlier work (e.g.~\cite{CGS} and \cite{ZiPhD}), the principal $P$ was fixed and the vortex equations where seen as equations for a pair $(A,u)$ rather than a triple $(P,A,u)$. (However, in \cite{MT} I.~Mundet i Riera and G.~Tian took the viewpoint of the present article.) The motivation for making $P$ part of the data is twofold: 

When formulating convergence for a sequence of vortices on $\R^2$ against a stable map, one has to pull back the vortices by translations of $\R^2$. (See Section \ref{subsec:conv}.) If the principal bundle is fixed and vortices are defined as pairs $(A,u)$ solving (\ref{eq:BAR dd J A u},\ref{eq:F A mu}), then there is no natural such pullback. However, there \emph{is} a natural pullback if the principal is made part of the data for a vortex. (This is true for an arbitrary surface $\Si$.) 

Another motivation is the following: If the area form or the complex structure on the surface $\Si$ vary, then in the limit we may obtain a surface $\Si'$ with singularities. It does not make sense to consider $P$ as a bundle over $\Si'$. One way of solving this problem is by decomposing $\Si'$ into smooth surfaces, and constructing smooth principal bundles over these surfaces. Hence the principal should be viewed as a varying object. 

Once $P$ is made part of the data, it is natural to consider \emph{equivalence classes} of triples $(P,A,u)$ rather than the triples themselves, since all important quantities, like energy density and energy, are invariant (or equivariant) under equivalence. Viewing the equivalence classes as the fundamental objects matches the physical viewpoint that the ``gauge field'', i.e., the connection $A$, is physically relevant only ``up to gauge''. 
\subsection*{Organization}
This article is organized as follows. In Section \ref{sec:stable} we define the notion of a stable map of vortices on $\R^2$ and pseudo-holomorphic spheres in $\BAR M$ and convergence against such a stable map. 

The main result of Section \ref{sec:comp} (Proposition \ref{prop:cpt mod}) is that given a sequence of rescaled vortices with uniformly bounded energies, there exists a subsequence that converges modulo bubbling at finitely many points. The proof is based on compactness for rescaled vortices on the punctured plane with uniformly bounded energy densities (Proposition \ref{prop:cpt bdd}). It also uses the fact that at each bubbling point at least the energy $\Emin$ is lost, where $\Emin>0$ is the minimal energy of a vortex on $\R^2$ or pseudo-holomorphic sphere in $\BAR M$. This is the content of Proposition \ref{prop:quant en loss}, which is proved here by a hard rescaling argument, using Proposition \ref{prop:cpt bdd} and Hofer's lemma. We also state and prove Lemma \ref{le:conv e}, which says that the energy densities of a convergent sequence of rescaled vortices converge to the density of the limit. This is used in the proof of Proposition \ref{prop:cpt mod}.

The main result of Section \ref{sec:soft} is Proposition \ref{prop:soft}, which tells how to find the next bubble in the bubbling tree, at a bubbling point of a given sequence of rescaled vortices. A crucial ingredient in its proof is Proposition \ref{prop:en conc} (proven in the same section). This result states that the energy of a vortex on an annulus is concentrated near the ends, provided that it is small enough.

Based on Sections \ref{sec:comp} and \ref{sec:soft}, the main result, Theorem \ref{thm:bubb}, is proven in Section \ref{sec:proof:thm:bubb}. In the appendix we recollect results on vortices, the invariant symplectic action, Uhlenbeck compactness, compactness for $\bar\del_J$, pseudo-holomorphic maps into $\BAR M$ and other auxiliary results, used in the proof of Theorem \ref{thm:bubb}. 
\subsection*{Acknowledgments}
This article arose from my Ph.D.-thesis. I would like to thank my adviser, Dietmar A.~Salamon, for the inspiring topic. I highly profited from his mathematical insight. I am very much indebted to Chris Woodward for his interest in my work, for sharing his ideas with me, and for his continuous encouragement. It was he who coined the term ``quantum Kirwan map''. I would also like to thank Urs Frauenfelder, Kai Cieliebak, Eduardo Gonzalez, and Andreas Ott for stimulating discussions. 
\section{Stable maps of vortices over the plane and holomorphic spheres in the symplectic quotient}\label{sec:stable}
\subsection{Stable maps}\label{subsec:stable}
Let $M,\om,G,\g,\lan\cdot,\cdot\ran_\g,\mu,J$ be as in Section \ref{sec:main}. Our standing hypothesis (H) implies that the symplectic quotient 
\[\big(\BAR M=\mu^{-1}(0)/G,\BAR\om\big)\]
is well-defined and closed. The structure $J$ induces a $\BAR\om$-compatible almost complex structure on $\BAR M$ as follows. For every $x\in M$ we denote by $L_x:\g\to T_xM$ the infinitesimal action at $x$. We define the \emph{horizontal distribution} $H\sub T\mu^{-1}(0)$ by
\[H_x:=\ker d\mu(x)\cap\im L_x^\perp,\quad\forall x\in\mu^{-1}(0).\]
Here $\perp$ denotes the orthogonal complement with respect to the metric $\om(\cdot,J\cdot)$ on $M$. We denote by $\pi:\mu^{-1}(0)\to \BAR M:=\mu^{-1}(0)/G$ the canonical projection. We define $\bar J$ to be the unique endomorphism of $T\BAR M$ such that
\begin{equation}\label{eq:bar J}\bar J\,d\pi= d\pi J\textrm{ on }H.
\end{equation}
We identify $\R^2\cup\{\infty\}$ with $S^2$. The (Connectedness) condition in the definition of a stable map below will involve evaluation of a map $S^2\to\BAR M$ at a given point in $S^2$ and of a vortex at the point $\infty\in S^2$. In order to make sense of the latter, we need the following. We denote by $Gx$ the orbit of a point $x\in M$. Let $P$ be a smooth principal $G$-bundle over $\R^2$ and $u\in C^\infty_G(P,M)$ a map. We define
\begin{equation}\nn\bar u:\R^2\to M/G,\quad\bar u(z):=Gu(p),\end{equation}
where $p\in P$ is an arbitrary point in the fiber over $z$. For $W\in\WWP$ we define 
\begin{equation}\label{eq:BAR u W}\bar u_W:=\bar u,\end{equation}
where $w=(P,A,u)$ is any representative of $W$. This is well-defined, i.e., does not depend on the choice of $w$. Recall the definition (\ref{eq:ME}) of $\ME$. 
\begin{prop}[Continuity at $\infty$]\label{prop:bar u} If $W\in\ME$ then the map $\bar u_W:\R^2\to M/G$ extends continuously to a map $f:S^2\to M/G$, such that $f(\infty)\in\BAR M=\mu^{-1}(0)/G$. 
\end{prop}
\begin{proof}[Proof of Proposition \ref{prop:bar u}] This follows from the estimate (\ref{eq:sup z z' bar d}) with $R=\infty$ in Proposition \ref{prop:en conc} below. (Alternatively, one can use \cite[Proposition 11.1]{GS}.)
\end{proof}
\begin{defi}\label{defi:barev} We define the \emph{evaluation map}
\[\barev:\big(C^0(S^2,M/G)\x S^2\big)\disj(\ME\x\{\infty\})\to M/G\]
as follows. For $(\bar u,z)\in C^0(S^2,M/G)\x S^2$ we define
\begin{equation}\label{eq:barev}
\barev_z(\bar u):=\barev(\bar u,z):=\bar u(z). 
\end{equation}
Furthermore, for $W\in\ME$ we define
\begin{equation}\label{eq:barev w}\barev_\infty(W):=f(\infty),
\end{equation}
where $f$ is as in Proposition \ref{prop:bar u}.
\end{defi}
\begin{defi}\label{defi:st} For every $k\in\N_0=\{0,1,\ldots\}$ a \emph{(genus 0) stable map of vortices on $\R^2$ and pseudo-holomorphic spheres in $\BAR M$ with $k+1$ marked points} is a tuple
\begin{equation}\label{eq:W z}(\W,\z):=\big(V,\BAR T,E,(W_\alpha)_{\alpha\in V},(\bar u_\alpha)_{\al\in \BAR T},(z_{\alpha\beta})_{\alpha E\beta},(\alpha_i,z_i)_{i=0,\ldots,k}\big),\end{equation}
where $V$ and $\BAR T$ are finite sets, $E$ is a tree relation on $T:=V\disj \BAR T$, $W_\al\in\ME$ (for $\al\in V$), $\bar u_\al:S^2\to\BAR M=\mu^{-1}(0)/G$ is a $\bar J$-holomorphic map (for $\al\in \BAR T$), $z_{\al\beta}\in S^2$ is a point for each adjacent pair $\al E\beta$, $\al_i\in T$ is a vertex and $z_i\in S^2$ is a point, for $i=0,\ldots,k$, such that the following conditions hold.

\begin{enui} 
\item \label{defi:st dist}{\bf (Special points)} 
  \begin{itemize}
  \item If $\al_0\in V$ then $z_0=\infty$. 
\item Fix $\al\in T$. Then the points $z_{\al\beta}$ with $\beta\in T$ such that $\al E\beta$ and the points $z_i$ with $i=0,\ldots,k$ such that $\al_i=\al$, are all distinct. 
\item If $\al\in V$ and $\be\in T$ are such that $\al E\be$ then $z_{\al\be}=\infty$. 
\end{itemize}

\item\label{defi:st conn}{\bf (Connectedness)} Let $\al,\beta\in T$ be such that $\al E\beta$. Then 
\[\barev_{z_{\al\beta}}(W_\al)=\barev_{z_{\beta\al}}(W_\beta).\]
Here $\barev$ is defined as in (\ref{eq:barev}) and (\ref{eq:barev w}) and we set $W_\al:=\bar u_\al$ if $\al\in \BAR T$. 
\item\label{defi:st st}{\bf(Stability)} If $\al\in V$ is such that $\E(W_\al)=0$ then there exists $i\in\{1,\ldots,k\}$ such that $\al_i=\al$. Furthermore, if $\al\in\BAR T$ is such that $E(\bar u_\al)=0$ then 
\[\#\{\beta\in T\,|\,\al E\beta\}+\#\{i\in\{0,\ldots,k\}\,|\,\al_i=\al\}\geq3.\]
\end{enui}
\end{defi}

\begin{figure}
  \centering
\leavevmode\epsfbox{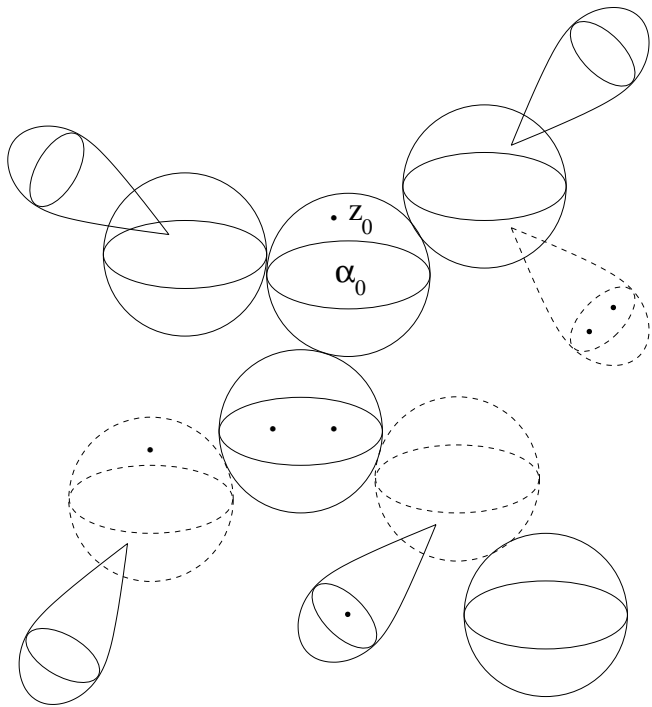}
\caption{Stable map. The ``raindrops'' correspond to vortices on $\R^2$ and the spheres to pseudo-holomorphic spheres in $\BAR M$. The seven dots are marked points. The dashed objects are ``ghosts'', i.e., they carry no energy.}
  \label{fig:stable map} 
\end{figure}
This definition is modelled on the notion of a genus 0 stable map of pseudo-holomorphic spheres, as introduced by Kontsevich in \cite{Kontsevich}. (For an exhaustive exposition of those stable maps see the book by D.~McDuff and D.~A.~Salamon \cite{MS}.)\\

\noindent{\bf Remarks.} It follows from condition (\ref{defi:st dist}) that if $\al\in V$ then there exists at most one $\beta\in T$ such that $\al E\beta$. This means that every vortex is a \emph{leaf} of the tree $T$. Furthermore, if $\al_0\in V$ then it follows that $T=V$ consists only of $\al_0$. It follows that if $T$ has at least two elements, then $\al_0\in\BAR T$, and hence $\BAR T\neq\emptyset$. Furthermore, if $\al\in V$ and $\be\in T$ are such that $\al E\be$ then $\be\in\BAR T$. This means that two vortices cannot be adjacent. $\Box$
\begin{rmk}\label{rmk:z i}\rm If $1\leq i\leq k$ is such that $\al_i\in V$ then $z_i\neq\infty$. This follows from condition (\ref{defi:st dist}). $\Box$
\end{rmk}

We fix a stable map $(\W,\z)$ as in Definition \ref{defi:st} and $\al\in T$. We define the \emph{set of nodal points at $\al$} to be 
\begin{equation}\label{eq:Z al}Z_\al:=\{z_{\al\beta}\,|\, \beta\in T,\al E\beta\}\sub S^2,\end{equation}
the set of \emph{marked points on $\al$} to be 
\[\big\{z_i\,|\,\al_i=\al,i\in\{0,\ldots,k\}\big\},\] 
and the \emph{set $Y_\al$ of special points} to be the union of $Z_\al$ and the set of marked points at $\al$. The stability condition (\ref{defi:st st}) says that if $\al\in V$ is such that $\E(W_\al)=0$ then $\al$ carries at least one marked point on $\R^2$. (It also carries a special point at $\infty$.) Furthermore, if $\al\in\BAR T$ is such that $\bar u_\al$ is a constant map, then $\al$ carries at least three special points.

The stability condition ensures that the action of a natural reparametrization group on the set of \emph{simple} stable maps of a given type is free. (See Proposition \ref{prop:simple} below.) This will be needed in order to show that the evaluation map on the set of non-trivial vortices (with marked points) is a pseudo-cycle. \\

\noindent{\bf Examples.} 
The easiest example of a stable map consists of the tree with one vertex $T=V=\{\al_0\}$, a vortex $W\in \ME$, the marked point $z_0:=\infty$ and a finite number of distinct points $z_i\in \R^2$, $i=1,\ldots,k$, where $k\geq1$ if $\E(W)=0$. 

As another example we set $V:=\emptyset$. Then a stable map in the new sense is a genus 0 stable map of $\bar J$-holomorphic spheres in $\BAR M$. $\Box$ 

\begin{xpl}\rm\label{xpl:stable} We set $k:=0$, choose an integer $\ell\in\N_0$, and define 
\[V:=\{1,\ldots,\ell\},\quad\BAR T:=\{0\},\quad E:=\big\{(0,1),\ldots,(0,\ell),(1,0),\ldots,(\ell,0)\big\},\]
\[\al_0:=0,\quad z_{i0}:=\infty,\,\forall i=1,\ldots,\ell.\] 
Let $z_0,z_{0i}\in S^2$, $i=1,\ldots,\ell$ be distinct points, $W_i\in\ME$ be such that $\E(W_i)>0$, for $i=1,\ldots,\ell$, and $\bar u_0$ a $\bar J$-holomorphic sphere. If $\ell\leq 1$ then assume that $\bar u_0$ is nonconstant. Then the tuple
\[(\W,\z):=\big(V,\BAR T,E,(W_i)_{i\in\{1,\ldots,\ell\}},\bar u_0,(z_{ij})_{i Ej},(0,z_0)\big)\]
is a stable map. (See Figure \ref{fig:example stable}.) $\Box$
\end{xpl}
\begin{figure}
  \centering
\leavevmode\epsfbox{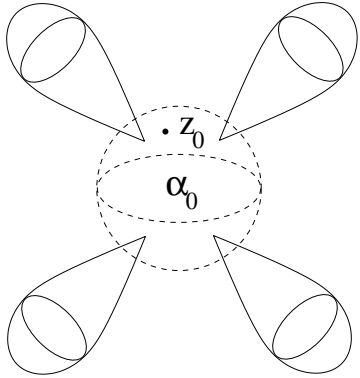}
\caption{This is the stable map described in Example \ref{xpl:stable} with $\ell:=4$.}
  \label{fig:example stable} 
\end{figure}
\begin{rmk}\rm\label{rmk:al 0 map} In the previous example with $\ell=2$ stability of the component $\al:=0\in\BAR T$ uses the ``additional'' marked point $z_0$. 
This example (with $\ell=2$) will be used in the argument showing that the quantum Kirwan map is a ring homomorphism. This is one reason for having the extra marked point. (Another one is explained in Remark \ref{rmk:al 0 conv} below.) $\Box$
\end{rmk}
\begin{xpl}\label{xpl:stable map} \rm Let $(M,\om,J,G):=(\R^2,\om_0,i,S^1)$, equip $\g:=\Lie(S^1)=i\R$ with the standard inner product, and consider the action of $S^1\sub\C$ on $\R^2=\C$ by multiplication of complex numbers. We define a moment map $\mu:\R^2\to\g$ for this action by $\mu(z):=\frac i2(1-|z|^2)$. In this setting, stable maps are classified in terms of their combinatorial structure $(V,\BAR T,E)$, the location of the special points, and for each $\al\in V$, a point in some symmetric product of $\R^2$. Each such point corresponds to a vortex on $\R^2$. (See the forth-coming article \cite{ZiExStable}.) $\Box$
\end{xpl}
For the definition of the quantum Kirwan map one needs to show that a certain natural evaluation map on the space of vortices on $\R^2$ (see \cite{ZiConsEv}) is a pseudo-cycle. This will rely on the fact that its omega limit set has codimension at least two. In order to show this, one needs to cut down the dimensions of the ``boundary strata'' by dividing by the actions of suitable ``reparametrization groups''. We define these groups as follows.

We fix two finite sets $\BAR T,V$ and a tree relation $E$ on the disjoint union $T:=\BAR T\disj V$ such that every element of $V$ is a leaf. We define the \emph{reparametrization group} $G_T$ as follows. We define $\Aut(T):=\Aut\big(\BAR T,V,E\big)$ to be the subgroup of all automorphisms $f$ of the tree $(T,E)$, satisfying $f(\BAR T)=\BAR T$ and $f(V)=V$. 

We denote by $\PSL$ the group of M\"obius transformations, i.e., biholomorphic maps on $S^2\iso\CP^1$, and by $\TR$ the group of translations of the plane $\R^2$. We define $\Aut_\al:=\TR$ if $\al\in V$, and $\Aut_\al:=\PSL$ if $\al\in\BAR T$. We denote by $\Aut_T$ the set of collections $(\phi_\al)_{\al\in T}$, such that $\phi_\al\in\Aut_\al$, for every $\al\in T$. The group $\Aut(T)$ acts on $\Aut_T$ by 
\[f\cdot(\phi_\al)_{\al\in T}:=(\phi_{f^{-1}(\al)})_{\al\in T}.\]
\begin{defi}\label{defi:G T} We define $G_T:=G_{\BAR T,V,E}$ to be the semi-direct product of $\Aut(T)$ and $\Aut_T$ induced by this action. 
\end{defi}
The group $\PSL$ acts on the set of $\bar J$-holomorphic maps $S^2\to\BAR M$ by 
\[\phi^*f:=f\circ\phi.\] 
Furthermore, the group $\TR$ acts on the set $\ME$ by 
\begin{equation}\label{eq:phi []}\phi^*[P,A,u]:=\big[\phi^*P,\Phi^*(A,u)\big],\end{equation}
where $\Phi:\phi^*P\to P$ is defined by $\Phi(z,p):=p$, and $[P,A,u]$ denotes the equivalence class of $(P,A,u)$. By the \emph{combinatorial type} of a stable map $(\W,\z)$ as in (\ref{eq:W z}) we mean the tuple $T:=(V,\BAR T,E)$. We denote by
\[\M(T):=\M(\BAR T,V,E)\]
the set of all \emph{stable maps of (combinatorial) type $T$}. $G_T$ acts on $\M(T)$ as follows. For every $(f,(\phi_\al))\in G_T$ and $(\W,\z)\in\M(T)$ we define
\begin{eqnarray}\nn&W'_\al:=\phi_{f(\al)}^*W_{f(\al)},\,\forall\al\in T,\quad z'_{\al\be}:=\phi_{f(\al)}^{-1}(z_{f(\al)f(\be)}),\,\forall\al E\be,&\\
\nn&\al'_i:=f(\al_i),\,z'_i:=\phi_{\al'_i}^{-1}(z_{\al'_i}),\,i=0,\ldots,k.&\end{eqnarray}
(Here we set $W_\al:=\bar u_\al$ if $\al\in\BAR T$. Furthermore, for $\phi\in\TR$ we set $\phi(\infty):=\infty$.) 
\begin{defi}\label{defi:action G T} We define 
\[(f,(\phi_\al))^*(\W,\z):=\big(V,\BAR T,E,(W'_\al)_{\al\in T},(z'_{\al\be})_{\al E\be},(\al'_i,z'_i)_{i=0,\ldots,k}\big).\]
\end{defi}
This defines an action of $G_T$ on $\M(T)$. Let now $(M,J)$ be an almost complex manifold. Recall that a $J$-holomorphic map $u:S^2\to M$ is called \emph{multiply covered} iff there exists a holomorphic map $\phi:S^2\to S^2$ of degree at least two, and a $J$-holomorphic map $v:S^2\to M$, such that $u=v\circ\phi$. Otherwise, $u$ is called \emph{simple}. 

Returning to the setting of the current section, let $\bar u\in C^\infty(S^2,\BAR M)$ be a $\bar J$-holomorphic map. We call a stable map $(\W,\z)$ \emph{simple} iff the following conditions hold: For every $\al\in\BAR T$ the $\bar J$-holomorphic map $\bar u_\al$ is constant or simple. Furthermore, if $\al,\be\in V$ are such that $\al\neq\be$ and $\E(W_\al)\neq0$, and $\phi\in\TR$, then $\phi^*W_\al\neq W_\be$. Moreover, if $\al,\be\in\BAR T$ are such that $\al\neq\be$ and $\bar u_\al$ is nonconstant, and $\phi\in\PSL$, then $\phi^*\bar u_\al=\bar u_\al\circ\phi\neq\bar u_\be$. We denote by
\[\M^*(T):=\M^*(\BAR T,V,E)\sub\M(T)\]
the subset of all \emph{simple stable maps}. The action of $G_T$ on $\M(T)$ leaves $\M^*(T)$ invariant. 
\begin{prop}\label{prop:simple} The action of $G_T$ on $\M^*(T)$ is free. 
\end{prop}
\begin{proof}[Proof of Proposition \ref{prop:simple}] This follows from an elementary argument, using the stability condition (\ref{defi:st st}), the freeness of the action of $\TR$ on $\ME$ (see Lemma \ref{le:trans} in the appendix), and the fact that every simple holomorphic sphere is somewhere injective (see \cite[Proposition 2.5.1]{MS}).
\end{proof}
Heuristically, this result implies that the quotient 
\[\M^*(T)/G_T\]
is canonically a smooth finite dimensional manifold. This will be important for the pseudo-cycle property of the evaluation map defined on the set of vortices on $\R^2$.

\begin{rmk}\label{rmk:Isom +}\rm The action of $\TR$ on $\ME$ extends to an action of the group $\Isom^+(\R^2)$ of orientation preserving isometries of $\R^2$. Hence one may be tempted to adjust the definition of the reparametrization group $G_T$ and its action on $\M^*(T)$ accordingly. However, for the purpose of defining the quantum Kirwan map, this is not possible. The reason is that in general there is no evaluation map on $\ME$ that is invariant under the action of $\Isom^+(\R^2)$. This is a crucial difference between vortices and pseudo-holomorphic curves. Note also that the action of $\Isom^+(\R^2)$ on the set of vortices of positive energy is not always free. (For an example see \cite{ZiExStable}.) See also the Remark \ref{rmk:conv rot}. $\Box$
\end{rmk}
\subsection{Convergence against a stable map}\label{subsec:conv}
In order to define convergence, we need the following notation. Let $\al\in T$ and $i=0,\ldots,k$. We define $z_{\al,i}\in S^2$ as follows. If $\al=\al_i$ then we set 
\begin{equation}
  \label{eq:z al i}z_{\al,i}:=z_i.
\end{equation}
Otherwise let $\be\in \BAR T$ be the unique vertex such that the chain of vertices of $T$ running from $\al$ to $\al_i$ is given by $(\al,\be,\ldots,\al_i)$. ($\be=\al_i$ is also allowed.) We define 
\begin{equation}
\label{eq:z al i }z_{\al,i}:=z_{\al\be}.
\end{equation}
We define 
\begin{equation}\label{eq:M *}M^*:=\big\{x\in M\,|\,\textrm{if }g\in G:\,gx=x\Rightarrow g=\one\big\}.\end{equation}
Note that $\mu^{-1}(0)\sub M^*$ by our standing hypothesis (H). Recall the definitions (\ref{eq:BAR u W},\ref{eq:Z al},\ref{eq:phi []}) of $\bar u_W,Z_\al$ and the action of $\TR$ on $\ME$. Let $k\geq0$, for $\nu\in\N$ let $W_\nu\in\ME$ be a vortex and $z^\nu_1,\ldots,z_k^\nu\in \R^2$ be points, and let 
\[(\W,\z):=\big(V,\BAR T,E,(W_\alpha)_{\alpha\in T},(z_{\alpha\beta})_{\alpha E\beta},(\alpha_i,z_i)_{i=0,\ldots,k}\big)\]
be a stable map. Here we use the notation $W_\al:=\bar u_\al$ if $\al\in\BAR T$. For a $\bar J$-holomorphic map $f:S^2\to\BAR M$ we denote its energy by 
\[E(f)=\int_{S^2}f^*\BAR\om.\]
Let $\Si$ be a compact smooth surface (possibly with boundary). Recall the definition (\ref{eq:WW}) of $\WWP(\Si)$. We define the \emph{$C^\infty$-topology $\tau_\Si$} on this set as follows: We fix a smooth principal $G$-bundle $P$ over $\Si$ and a $C^\infty$-open subset $\UU\sub\A(P)\x C^\infty_G(P,M)$. (This means that $\UU$ is $C^k$-open for some $k\in\N_0$.) We define
\[\BAR\UU:=\big\{[P,A,u]\,\big|\,(A,u)\in\UU\big\}.\]
We define 
\begin{equation}\label{eq:tau Si}\tau_\Si:=\big\{\BAR\UU\,\big|\,P,\,\UU\textrm{ as above}\big\}.
\end{equation}
Let $\Si$ be a smooth surface, $W=[P,A,u]\in\WWP(\Si)$, and $\Om\sub\Si$ an open subset with compact closure and smooth boundary. We define the \emph{restriction $W|\BAR\Om$} to be the equivalence class of the pullback of $(P,A,u)$ under the inclusion map $\BAR\Om\to\Si$. 
\begin{defi}[Convergence]\label{defi:conv} The sequence $(W_\nu,z_0^\nu:=\infty,z_1^\nu,\ldots,z_k^\nu)$ is said to converge to $(\W,\z)$ as $\nu\to\infty$ iff the limit $E:=\lim_{\nu\to\infty} E(W_\nu)$ exists,
\begin{equation}
  \label{eq:E V bar T}E=\sum_{\al\in T}E(W_\al),
\end{equation}
and there exist M\"obius transformations $\phi_\al^\nu:S^2\to S^2$, for $\al\in T:=V\disj\BAR T$, $\nu\in\N$, such that the following conditions hold.

\begin{enui}\item \label{defi:conv phi z} 
\begin{itemize}
\item If $\al\in V$ then $\phi_\al^\nu$ is a translation on $\R^2$. 
\item For every $\al\in \BAR T$ we have $\phi_\al^\nu(z_{\al,0})=\infty$, where $z_{\al,0}$ is defined as in (\ref{eq:z al i}), (\ref{eq:z al i }). 
\item Let $\al\in \BAR T$ and $\psi_\al$ be a M\"obius transformation such that $\psi_\al(\infty)=z_{\al,0}$. Then the derivatives $(\phi_\al^\nu\circ\psi_\al)'(z)$ converge to $\infty$, for every $z\in\R^2=\C$. 
\end{itemize}
\item \label{defi:conv al be} If $\al,\be\in T$ are such that $\al E\beta$ then $(\phi_\al^\nu)^{-1}\circ\phi_\beta^\nu\to z_{\al\beta}$, uniformly on compact subsets of $S^2\wo\{z_{\beta\al}\}$.
\item \label{defi:conv w} 
\begin{itemize}
\item Let $\al\in V$ and $\Om\sub R^2$ be an open subset with compact closure and smooth boundary. Then the restriction of $(\phi_\al^\nu)^*W_\nu$ to $\BAR\Om$ converges to $W_\al$ with respect to the topology $\tau_{\BAR\Om}$ (as defined in (\ref{eq:tau Si})).
\item Fix $\al\in \BAR T$. Let $Q$ be a compact subset of $S^2\wo (Z_\al\cup\{z_{\al,0}\})$. For $\nu$ large enough, we have
\[\bar u_\al^\nu:=\bar u_{W_\nu}\circ\phi_\al^\nu(Q)\sub M^*/G,\]
and $\BAR{u_\al^\nu}$ converges to  $\bar u_\al$ in $C^1$ on $Q$. (Here $\bar u_{W_\nu}$ is defined as in (\ref{eq:BAR u W}).) 
\end{itemize}
\item\label{defi:conv z} We have $(\phi_{\al_i}^\nu)^{-1}(z_i^\nu)\to z_i$ for every $i=1,\ldots,k$. 
\end{enui}
\end{defi}
(See Figure \ref{fig:convergence}.)
\begin{figure} 
  \centering
\leavevmode\epsfbox{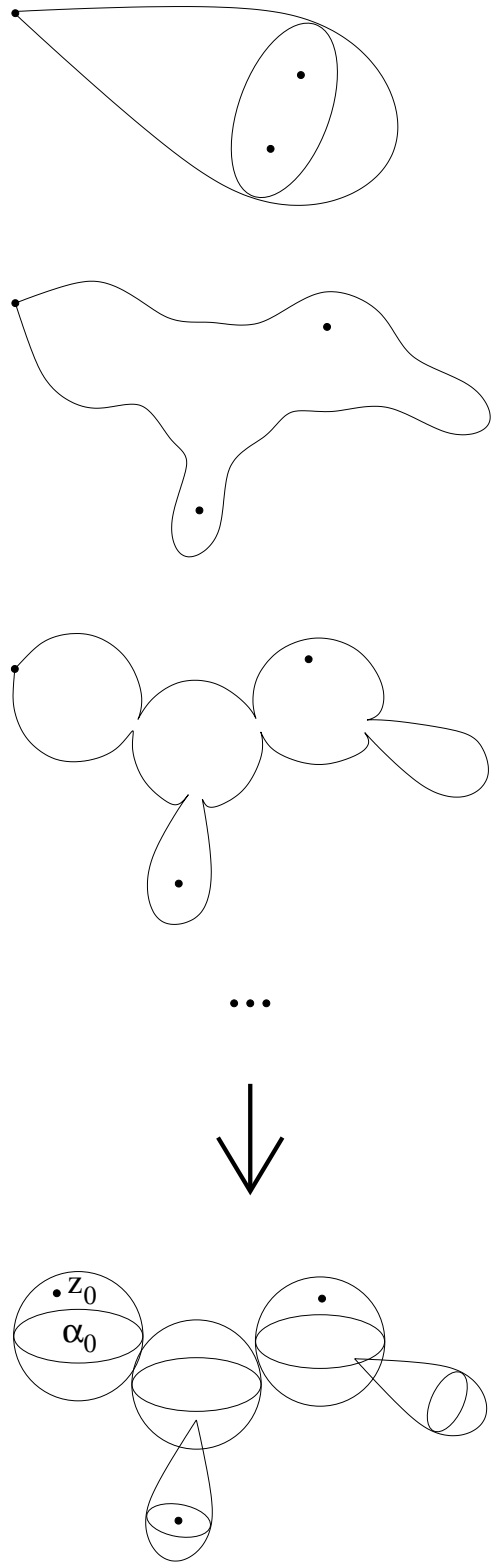}
\caption{Convergence of a sequence of vortices on $\R^2$ against a stable map.}
  \label{fig:convergence} 
\end{figure}
This definition is based on the notion of convergence of a sequence of pseudo-holomorphic spheres to a genus 0 stable map of pseudo-holomorphic spheres. (For that notion see for example \cite{MS}).\\

\noindent{\bf Remark.} The last part of condition (\ref{defi:conv phi z}) and the second part of condition (\ref{defi:conv w}) capture the idea of catching a pseudo-holomorphic sphere in $\BAR M$ by ``zooming out'': Fix $\al\in\BAR T$, and consider the case $z_{\al,0}=\infty$. Then there exist $\lam_\al^\nu\in\C\wo \{0\}$ and $z_\al^\nu\in\C$ such that $\phi_\al^\nu(z)=\lam_\al^\nu z+ z_\al^\nu$. 

It follows from a direct calculation that $(\phi_\al^\nu)^*W_\nu$ is a vortex with respect to the area form $\om_\Si=|\lam_\al^\nu|^2\om_0$, where $\om_0$ denotes the standard area form on $\R^2$. 

The last part of condition (\ref{defi:conv phi z}) means that $\lam_\al^\nu\to\infty$, for $\nu\to\infty$. Hence in the limit $\nu\to\infty$ we obtain the equations 
\[\bar\del_{J,A}(u)=0,\quad\mu\circ u=0.\]
These correspond to the $\bar J$-Cauchy-Riemann equations for a map from $\R^2=\C$ to $\BAR M$. (See Proposition \ref{prop:bar del J}.) The second part of (\ref{defi:conv w}) imposes that the sequence of rescaled vortices converges (in a suitable sense) to the $\bar J$-holomorpic sphere $\bar u_\al$. $\Box$\\

\noindent{\bf Remark.} The ``energy-conservation'' condition (\ref{eq:E V bar T}) has the important consequence that the stable map $(\W,\z)$ represents the same equivariant homology class as the vortex $W_\nu$, for $\nu$ large enough. (See \cite{ZiConsEv}.) $\Box$ 
\begin{rmk}\label{rmk:al 0 conv}\rm One purpose of the additional marked point $(\al_0,z_0)$ is to be able to formulate the second part of condition (\ref{defi:conv w}). (Another one is explained in Remark \ref{rmk:al 0 map} above.) For $\al\in\BAR T$ and $\nu\in\N$ the map $G u_\nu\circ\phi_\al^\nu$ is only defined on the subsets $(\phi_\al^\nu)^{-1}(\R^2)\sub S^2$. Since by condition (\ref{defi:conv phi z}) we have $\phi_\al^\nu(z_{\al,0})=\infty$, the composition $\bar u_{W_\nu}\circ\phi_\al^\nu:Q\to M/G$ is well-defined for each compact subset $Q\sub S^2\wo (Z_\al\cup\{z_{\al,0}\})$. Hence the the second part of condition (\ref{defi:conv w}) makes sense. $\Box$
\end{rmk}
\noindent{\bf Example.} Let $M,\om$ etc.~be as in Example \ref{xpl:stable map}. Then a sequence $W_\nu\in\ME$ converges to a stable map if and only if the total degree of $W_\nu$ equals the sum of the degrees of the vortex components of the stable map, and for each $\al\in V$, up to translations, the point in the symmetric product of $\R^2$ corresponding to $W_\nu$, converges to the point corresponding to the vortex $W_\al$. (See \cite{ZiExStable}.) $\Box$ 
\begin{rmk}\label{rmk:conv rot}\rm One conceptual difficulty in defining the notion of convergence is the following. (Compare also to Remark \ref{rmk:Isom +}.) Consider the group $\Isom^+(\Si)$ of orientation preserving isometries of $\Si$ (with respect to the metric $\om_\Si(\cdot,j\cdot)$). (This coincides with the group of diffeomorphisms of $\Si$ that preserve the pair $(\om_\Si,j)$.) This group acts on $\WWP(\Si)$ (defined as in (\ref{eq:WW})), as in (\ref{eq:phi []}). The set $\ME$ of finite energy vortices is invariant under this action. 

Hence naively, in the definition of convergence one would allow $\phi_\al^\nu$ to be an orientation preserving isometry of $\R^2$, rather than just a translation. The problem is that with this less restrictive condition, there is no evaluation map on the set of stable maps, that is continuous with respect to convergence. (Such a map is needed for the definition of the quantum Kirwan map.)

Note here that we cannot define evaluation of a vortex $W$ at some point $z\in\Si$ by choosing a representative of $W$ and evaluating it at some point in the fiber over $z$, since this depends on the choices. Instead, evaluation of $W$ at $z$ yields a point in the Borel construction for the action of $G$ on $M$. (See \cite{ZiConsEv}.) $\Box$
\end{rmk}
\section{Compactness modulo bubbling for rescaled vortices}\label{sec:comp}
In this section we consider a sequence of rescaled vortices on $\R^2$ with image in a fixed compact subset of $M/G$ and uniformly bounded energies. We assume that $(M,\om)$ is aspherical. The main result, Proposition \ref{prop:cpt mod} below, is that there exists a subsequence that away from finitely many bubbling points, converges to either a rescaled vortex on $\R^2$ or a $\bar J$-holomorphic sphere in $\BAR M$. This is a crucial ingredient of the proof of Theorem \ref{thm:bubb}.

In order to explain the result, let $M,\om,G,\g,\lan\cdot,\cdot\ran_\g,\mu,J,\Si,\om_\Si,j$ be as in Section \ref{sec:main}. Recall the definition (\ref{eq:e W}) of the energy density $e^{\om_\Si,j}_W:=e_W$ of a class $W\in\WWP(\Si)$.
\begin{rmk}\rm\label{rmk:trafo e vort} This density has the following transformation property: Let $\Si'$ be another real surface, and $\phi:\Si'\to\Si$ a smooth immersion. We define the pullback $\phi^*W$ as in (\ref{eq:phi []}). Then a straight-forward calculation shows that 
\begin{equation}\label{eq:e phi *}e^{\phi^*(\om_\Si,j)}_{\phi^*W}=e^{\om_\Si,j}_W\circ\phi,
\end{equation} 
and $W$ is a vortex with respect to $(\om_\Si,j)$ if and only if $\phi^*W$ is a vortex with respect to $\phi^*(\om_\Si,j)$. $\Box$
\end{rmk}
\begin{rmk}\rm\label{rmk:e vort} If $W$ is a vortex (with respect to $(\om_\Si,j)$) then 
\begin{equation}\label{eq:e vort}e^{\om_\Si,j}_W=|\dd_{J,A}u|^2+|\mu\circ u|^2,
\end{equation}
where $\dd_{J,A}u$ is the complex linear part of $d_Au$, viewed as a one-form on $\Si$ with values in $(u^*TM)/G\to\Si$. This follows from the vortex equations 
(\ref{eq:BAR dd J A u},\ref{eq:F A mu}). $\Box$ 
\end{rmk}
Let $R\in[0,\infty]$ and $W\in\WWP(\Si)$. Consider first the case $0<R<\infty$. Then we define the \emph{$R$-energy density of $W$} to be
\begin{equation}\label{eq:e R W R}e^R_W:=R^2e^{R^2\om_\Si,j}_W.\end{equation}
This means that 
\begin{equation}\label{eq:e R W 1 2}e^R_W=\frac12\big(|d_Au|_{\om_\Si}^2+R^{-2}|F_A|_{\om_\Si}^2+R^2|\mu\circ u|^2\big),
\end{equation}
where the subscript ``$\om_\Si$'' means that the norms are taken with respect to the metric $\om_\Si(\cdot,j\cdot)$. 

If $R=0$ or $\infty$ then we define 
\begin{equation}\nn e^R_W:=\frac12|d_Au|_{\om_\Si}^2.
\end{equation}

We define the \emph{$R$-energy of $W$} on a measurable subset $X\sub\Si$ to be
\begin{equation}\nn E^R(W,X):=\int_Xe^R_W\om_\Si\in [0,\infty].\end{equation}
The density and the energy have the following rescaling property: Consider the case $(\Si,\om_\Si,j)=(\R^2,\om_0,i)$, where $\om_0$ denotes the standard area form on $\R^2$. Assume that $0<R<\infty$, and consider the map $\phi:\R^2\to\R^2$ defined by $\phi(z):=Rz$. Then equality (\ref{eq:e phi *}) implies that
\[e^R_{\phi^*W}=R^2e^{\om_0,i}_W\circ\phi.\]
\noindent{\bf Remark.} The factor $R^2$ in the definition (\ref{eq:e R W R}) is important for the subsequent analysis (bubbling, convergence with bounded energy density etc.). However, the density $e^{R^2\om_\Si}_W$ is more intrinsic. (Compare to (\ref{eq:e phi *}).) $\Box$ 

The \emph{(symplectic) $R$-vortex equations} are the equations (\ref{eq:BAR dd J A u},\ref{eq:F A mu}) with $\om_\Si$ replaced by $R^2\om_\Si$, i.e., the equations 
\begin{equation}\label{eq:vort P R}\bar \dd_{J,A}(u)=0,\quad F_A+R^2(\mu\circ u)\om_\Si=0.\end{equation}
In the case $R=\infty$ we interpret the second equation in (\ref{eq:vort P R}) as
\[\mu\circ u=0.\]
\noindent{\bf Remark.} Consider the case $(\Si,\om_\Si,j)=(\R^2,\om_0,i)$ and $0<R<\infty$, and the map $\phi:\R^2\to\R^2$ given by $\phi(z):=Rz$. It follows from Remark \ref{rmk:trafo e vort} that a class $W\in\WWP(\R^2)$ is a vortex if and only if $\phi^*W$ is an $R$-vortex. $\Box$\\

\noindent{\bf Remark.} The rescaled energy density has the following important property. Let $R_\nu\in(0,\infty)$ be a sequence that converges to some $R_0\in[0,\infty]$, and for $\nu\in\N_0$ let $W_\nu$ be an $R_\nu$-vortex. If $W_\nu$ converges to $W_0$ in a suitable sense then
\begin{equation}\nn e^{R_\nu}_{W_\nu}\to e^{R^0}_{W_0}.\end{equation} 
(See Lemma \ref{le:conv e} below.) 
In the proof of Theorem \ref{thm:bubb}, this will be used in order to show that locally on $\R^2$ no energy is lost in the limit $\nu\to\infty$. $\Box$ 

We define the \emph{minimal energy $\Emin$} as follows. Recall the definition (\ref{eq:M}) of $\M$, and that we denote the energy of a $\bar J$-holomorphic map $f:S^2\to\BAR M$ by $E(f)=\int_{S^2}f^*\BAR\om$. We define
\begin{equation}\label{eq:E V}E_V:=\inf\big(\big\{E(W)\,\big|\,W\in\M:\,\BAR{\textrm{image}(W)}\textrm{ compact}\big\}\cap(0,\infty)\big),
\end{equation}
\[\BAR E:=\inf\big(\big\{E(f)\,\big|\,f\in C^\infty(S^2,\BAR M):\,\bar\dd_{\bar J}(f)=0\big\}\cap(0,\infty)\big),\]
\begin{equation}\label{eq:Emin}\Emin:=\min\{E_V,\BAR E\}.\end{equation}
Here we used the convention that $\inf\emptyset=\infty$. Assume that $M$ is equivariantly convex at $\infty$. Then Corollary \ref{cor:quant} below implies that $E_V>0$. Furthermore, our standing assumption (H) implies that $\BAR M$ is closed. It follows that $\BAR E>0$ (see for example \cite[Proposition 4.1.4]{MS}). Hence the number $\Emin$ is positive. 

The results of this and the next section are formulated for connections and maps of Sobolev regularity. This is a natural setup for the relevant analysis. Furthermore, we restrict our attention to the trivial bundle $\Si\x G$. (Since every smooth bundle over $\R^2$ is trivial, this suffices for the proof of the main result.) 

We fix $p>2$ and naturally identify the affine space of connections on $\Si\x G$ of local Sobolev class $W^{1,p}_\loc$ with the space of one-forms on $\Si$ with values in $\g$, of class $W^{1,p}_\loc$. Furthermore, we identify the space of $G$-equivariant maps from $\Si\x G$ to $M$ of class $W^{1,p}_\loc$ with $W^{1,p}_\loc(\Si,M)$. Finally, we identify the gauge group (i.e.,~group of gauge transformations) on $\Si\x G$ of class $W^{2,p}_\loc$ with $W^{2,p}_\loc(\Si,G)$. We denote
\[\WWW(\Si):=\Om^1(\Si,\g)\x C^\infty(\Si,M),\]
\[\WWW^p(\Si):=\big\{W^{1,p}_\loc\textrm{-one-form on }\Si\textrm{ with values in }\g\big\}\x W^{1,p}_\loc(\Si,M).\]
We call a solution $(A,u)\in\WWW^p(\Si)$ of the equations (\ref{eq:vort P R}) an \emph{$R$-vortex} over $\Si$. (It will be clear from the notation whether the term ``$R$-vortex'' refers to such a pair $(A,u)$ or to an equivalence class $W$ of triples $(P,A,u)$.) The gauge group $W^{2,p}_\loc(\Si,G)$ acts on $\WWW^p(\Si)$ by
\[g^*(A,u):=\big(\ad_{g^{-1}}A+g^{-1}dg,g^{-1}u\big),\]
where $\ad_{g_0}:\g\to\g$ denotes the adjoint action of an element $g_0\in G$. Let $w\in\WWW^p(\Si),R\in[0,\infty]$, and $X\sub\Si$ be a measurable subset. We denote by $[w]$ the gauge equivalence class of $w$, and denote
\[e^R_w:=e^R_{[w]},\quad E^R(w,X):=E^R([w],X)\textrm{ etc.}\]
For $r>0$ we denote by $B_r\sub\R^2$ the open ball of radius $r$, around 0.
\begin{prop}[Compactness modulo bubbling]\label{prop:cpt mod} Assume that $(M,\om)$ is aspherical. Let $R_\nu\in(0,\infty)$ be a sequence that converges to some $R_0\in (0,\infty]$, $r_\nu\in(0,\infty)$ a sequence that converges to $\infty$, and for every $\nu\in\N$ let $w_\nu=(A_\nu,u_\nu)\in\WWW^p(B_{r_\nu})$ be an $R_\nu$-vortex (with respect to $(\om_0,i)$). Assume that there exists a compact subset $K\sub M$ such that $u_\nu(B_{r_\nu})\sub K$, for every $\nu$. Suppose also that 
\[\sup_\nu E^{R_\nu}(w_\nu,B_{r_\nu})<\infty.\] 
Then there exist a finite subset $Z\sub\R^2$ and an $R_0$-vortex $w_0:=(A_0,u_0)\in\WWW(\R^2\wo Z)$, and passing to some subsequence, there exist gauge transformations $g_\nu\in W^{2,p}_\loc(\R^2\wo Z,G)$, such that the following conditions hold. 
\begin{enui}\item\label{prop:cpt mod:<} If $R_0<\infty$ then $Z=\emptyset$ and the sequence $g_\nu^*(A_\nu,u_\nu)$ converges to $w_0$ in $C^\infty$ on every compact subset of $\R^2$. 
\item\label{prop:cpt mod:=} If $R_0=\infty$ then on every compact subset of $\R^2\wo Z$, the sequence $g_\nu^*A_\nu$ converges to $A_0$ in $C^0$, and the sequence $g_\nu^{-1}u_\nu$ converges to $u_0$ in $C^1$.
\item\label{prop:cpt mod lim nu E eps} Fix a point $z\in Z$ and a number $\eps_0>0$ so small that $B_{\eps_0}(z)\cap Z=\{z\}$. Then for every $0<\eps<\eps_0$ the limit 
\[E_z(\eps):=\lim_{\nu\to\infty}E^{R_\nu}(w_\nu,B_\eps(z))\] 
exists and 
\[E_z(\eps)\geq\Emin.\] 
Furthermore, the function $(0,\eps_0)\ni\eps\mapsto E_z(\eps)\in[\Emin,\infty)$ is continuous.
\end{enui}
\end{prop}
\noindent{\bf Remark.} Convergence in conditions (\ref{prop:cpt mod:<},\ref{prop:cpt mod:=}) should be understood as convergence of the subsequence labelled by those indices $\nu$ for which $B_{r_\nu}$ contains the given compact set. $\Box$ 

This proposition will be proved on page \pageref{prop:cpt mod proof}. The strategy of the proof is the following. Assume that the energy densities $e^{R_\nu}_{w_\nu}$ are uniformly bounded on every compact subset of $\R^2$. Then the statement of Proposition \ref{prop:cpt mod} with $Z=\emptyset$ follows from an argument involving Uhlenbeck compactness, an estimate for $\bar\del_J$, elliptic bootstrapping (for statement (\ref{prop:cpt mod:<})), and a patching argument. 

If the densities are not uniformly bounded then we rescale the maps $w_\nu$ by zooming in near a bubbling point $z_0$ in a ``hard way'', to obtain a positive energy $\wt R_0$-vortex in the limit, with $\wt R_0\in\{0,1,\infty\}$. If $R_0<\infty$ then $\wt R_0=0$, and we obtain a $J$-holomorphic sphere in $M$. This contradicts symplectic asphericity, and thus this case is impossible. 

If $R_0=\infty$ then either $\wt R_0=1$ or $\wt R_0=\infty$, and hence either a vortex on $\R^2$ or a pseudo-holomorphic sphere in $\BAR M$ bubbles off. Therefore, at least the energy $\Emin$ is lost at $z_0$. Our assumption that the energies of $w_\nu$ are uniformly bounded implies that there can only be finitely many bubbling points. On the complement of these points a subsequence of $w_\nu$ converges modulo gauge. 

The bubbling part of this argument is captured by Proposition \ref{prop:quant en loss} below, whereas the convergence part is the content of the following result. 
\begin{prop}[Compactness with bounded energy densities]\label{prop:cpt bdd} Let $Z\sub\R^2$ be a finite subset, $R_\nu\geq 0$ be a sequence of numbers that converges to some $R_0\in [0,\infty]$, $\Om_1\sub\Om_2\sub\ldots\sub \R^2\wo Z$ open subsets such that $\bigcup_\nu\Om_\nu=\R^2\wo Z$, and for $\nu\in\N$ let $w_\nu=(u_\nu,A_\nu)\in\WWW^p(\Om_\nu)$ be an $R_\nu$-vortex. Assume that there exists a compact subset $K\sub M$ such that for $\nu$ large enough
\begin{equation}
  \label{eq:u nu Om K}u_\nu(\Om_\nu)\sub K. 
\end{equation}
Suppose also that for every compact subset $Q\sub \R^2\wo Z$, we have 
\begin{equation}
    \label{eq:sup e}
\sup\big\{\Vert e_{w_\nu}^{R_\nu}\Vert_{L^{\infty}(Q)}\,\big|\,\nu\in\N:\,Q\sub\Om_\nu\big\}<\infty.
  \end{equation}
Then there exists an $R_0$-vortex $w_0:=(A_0,u_0)\in\WWW(\R^2\wo Z)$, and passing to some subsequence, there exist gauge transformations $g_\nu\in W^{2,p}_\loc(\R^2\wo Z,G)$, such that the following conditions are satisfied.
\begin{enui}
\item\label{prop:cpt bdd < infty} If $R_0<\infty$ then $g_\nu^*w_\nu$ converges to $w_0$ in $C^\infty$ on every compact subset of $\R^2\wo Z$. 
\item\label{prop:cpt bdd = infty} If $R_0=\infty$ then on every compact subset of $\R^2\wo Z$, $g_\nu^*A_\nu$ converges to $A_0$ in $C^0$, and $g_\nu^{-1}u_\nu$ converges to $u_0$ in $C^1$. 
\end{enui}
\end{prop}
The proof of this result is an adaption of the argument of Step 5 in the proof of Theorem A the paper by R.~Gaio and D.~A.~Salamon \cite{GS}. The proof of statement (\ref{prop:cpt bdd < infty}) is based on a compactness result in the case of a compact surface $\Si$ (possibly with boundary). (See Theorem \ref{thm:cpt cpt} below. That result follows from an argument by K.~Cieliebak et al.~in \cite{CGMS}.) The proof also involves a patching argument for gauge transformations, which are defined on an exhausting sequence of subsets of $\R^2\wo Z$. 

To prove statement (\ref{prop:cpt bdd = infty}), we will show that curvatures of the connections $A_\nu$ are uniformly bounded in $W^{1,p}$. This uses the second rescaled vortex equations and a uniform upper bound on $\mu\circ u_\nu$ (Lemma \ref{le:mu u}), due to R.~Gaio and D.~A.~Salamon. The statement then follows from Uhlenbeck compactness with compact base, compactness for $\bar\del_J$, and a patching argument.
\begin{proof}[Proof of Proposition \ref{prop:cpt bdd}]\setcounter{claim}{0} We choose $i_0\in\N$ so big that the balls $\bar B_{1/i_0}(z)$, $z\in Z$, are disjoint and contained in $B_{i_0}$. We fix $i\in\N_0$ and define 
\[X^i:=\bar B_{i+i_0}\wo\bigcup_{z\in Z}B_{\frac1{i+i_0}}(z)\sub\R^2.\]

We prove {\bf statement (\ref{prop:cpt bdd < infty})}. Assume that $R_0<\infty$. Using the hypotheses (\ref{eq:u nu Om K},\ref{eq:sup e}), it follows from Theorem \ref{thm:cpt cpt} below that there exist an infinite subset $I^1\sub\N$ and gauge transformations $g^1_\nu\in W^{2,p}(X^1,G)$ ($\nu\in I^1$), such that $X^1\sub\Om_\nu$ and $w^1_\nu:=(A^1_\nu,u^1_\nu):=(g^1_\nu)^*(w_\nu|X^1)$ is smooth, for every $\nu\in I^1$, and the sequence $(w^1_\nu)_{\nu\in I^1}$ converges to some $R_0$-vortex $w^1\in\WWW(X^1)$, in $C^\infty$ on $X^1$. 

Iterating this argument, for every $i\geq2$ there exists an infinite subset $I^i\sub I^{i-1}$ and gauge transformations $g^i_\nu\in W^{2,p}(X^i,G)$ ($\nu\in I^i$), such that $X^i\sub\Om_\nu$ and $w^i_\nu:=(A^i_\nu,u^i_\nu):=(g^i_\nu)^*(w_\nu|X^i)$ is smooth, for every $\nu\in I^i$, and the sequence $(w^i_\nu)_{\nu\in I^i}$ converges to some $R_0$-vortex $w^i\in\WWW(X^1)$, in $C^\infty$ on $X^i$. 

Let $i\in\N$. For $\nu\in I^i$ we define $h^i_\nu:=(g^{i+1}_\nu|X^i)^{-1}g^i_\nu$. We have $(h^i_\nu)^*(A^{i+1}_\nu|X^i)=A^i_\nu$. Furthermore, $(A^{i+1}_\nu)_{\nu\in I^{i+1}}$ and $(A^i_\nu)_{\nu\in I^{i+1}}$ are bounded in $W^{k,p}$ on $X^i$, for every $k\in\N$. Hence it follows from Lemma \ref{le:g smooth} below that the sequence $(h^i_\nu)_{\nu\in I^{i+1}}$ is bounded in $W^{k,p}$ on $X^i$, for every $k\in\N$. Hence, using the Kondrachov compactness theorem, it has a subsequence that converges to some gauge transformation $h^i\in C^\infty(X^i,G)$, in $C^\infty$ on $X^i$. Note that 
\begin{equation}\label{eq:h i w i}(h^i)^*(w^{i+1}|X^i)=w^i.
\end{equation}
We choose a map $\rho^i:X^{i+1}\to X^i$ such that $\rho^i=\id$ on $X^{i-1}$. We define\footnote{This patching construction follows the lines of the proofs of \cite[Theorem 3.6 and Theorem A.3]{FrPhD}.} $k^1:=h^1$, and recursively,
\begin{equation}\label{k i h i}k^i:=h^i(k^{i-1}\circ\rho^{i-1})\in C^\infty(X^i,G),\quad\forall i\geq2.
\end{equation}
Using (\ref{eq:h i w i}) and the fact $\rho^{i-1}=\id$ on $X^{i-2}$, we have, for every $i\geq2$, 
\[(k^i)^*w^{i+1}=(k^{i-1}\circ\rho^{i-1})^*w^i=(k^{i-1})^*w^i,\quad\textrm{on }X^{i-2}.\]
It follows that there exists a unique $w\in\WWW(\R^2\wo Z)$ that restricts to $(k^{i+1})^*w^{i+2}$ on $X^i$, for every $i\in\N$. Let $i\in\N$. We choose $\nu_i\in I^{i+1}$ such that $\nu_i\geq i$ and a map $\tau^i:\R^2\wo Z\to X^i$ that is the identity on $X^{i-1}$. We define $g_i:=(g^{i+1}_{\nu_i}k^i)\circ\tau^i\in C^\infty(\R^2\wo Z,G)$. The sequence $g_i^*w_{\nu_i}$ converges to $w$, in $C^\infty$ on every compact subset of $\R^2\wo Z$. (Here we use the $C^\infty$-convergence on $X^i$ of $(w^i_\nu)_{\nu\in I^i}$ against $w^i$ and the facts $X_1\sub X_2\sub\cdots$ and $\bigcup_{i\in\N}X_i=\R^2\wo Z$.) Statement (\ref{prop:cpt bdd < infty}) follows. 

We prove {\bf statement (\ref{prop:cpt bdd = infty})}. Assume that $R_0=\infty$. 
\begin{claim}\label{claim:kappa nu p} For every compact subset $Q\sub \R^2\wo Z$ we have
  \begin{equation}
    \label{eq:sup kappa nu}\sup_\nu\big\{\Vert F_{A_\nu}\Vert_{L^p(Q)}\,\big|\,\nu\in\N:\,Q\sub\Om_\nu\big\}<\infty.
  \end{equation}
\end{claim}
\begin{proof}[Proof of Claim \ref{claim:kappa nu p}] Let $\Om\sub\R^2$ be an open subset containing $Q$ such that $\BAR\Om$ is compact and contained in $\R^2\wo Z$. Hypothesis (\ref{eq:sup e}) implies that
\begin{equation}\label{eq:sup nu Vert}\sup_\nu\Vert d_{A_\nu}u_\nu\Vert_{L^\infty(\BAR\Om)}<\infty. 
\end{equation}
It follows from our standing hypothesis (H) that there exists $\delta>0$ such that $G$ acts freely on
\[K:=\{x\in M\,|\,|\mu(x)|\leq\delta\}.\]
Since $\mu$ is proper the set $K$ is compact. It follows that 
\begin{equation}\label{eq:sup xi}\sup\left\{\frac{|\xi|}{|L_x\xi|}\,\Big|\,x\in K,\,0\neq\xi\in\g\right\}<\infty.\end{equation}
Using the second vortex equation, we have $|\mu\circ u_\nu|\leq\sqrt{e^{R_\nu}_{w_\nu}}/R_\nu$. Hence by hypothesis (\ref{eq:sup e}) and the assumption $R_\nu\to\infty$, we have $\Vert\mu\circ u_\nu\Vert_{L^\infty(\Om)}<\delta$, for $\nu$ large enough. Using (\ref{eq:sup nu Vert},\ref{eq:sup xi}), Lemma \ref{le:mu u} implies that 
\[\sup_\nu R_\nu^2\Vert\mu\circ u_\nu\Vert_{L^p(Q)}<\infty.\]
Estimate (\ref{eq:sup kappa nu}) follows from this and the second vortex equation. This proves Claim \ref{claim:kappa nu p}.
\end{proof}
Using Claim \ref{claim:kappa nu p}, Theorem \ref{thm:Uhlenbeck compact} (Uhlenbeck compactness) below implies that there exist an infinite subset $I^1\sub\N$ and gauge transformations $g^1_\nu\in W^{2,p}(X^1,G)$, for $\nu\in I^1$, such that $X^1\sub\Om_\nu$, for every $\nu\in I^1$, and the sequence $A^1_\nu:=(g^1_\nu)^*(A_\nu|X^1)$ converges to some $W^{1,p}$-connection $A^1$ over $X^1$, weakly in $W^{1,p}$ on $X^1$. By the Kondrachov compactness theorem, shrinking $I^1$, we may assume that $A^1_\nu$ converges (strongly) in $C^0$ on $X^1$. 

Iterating this argument, for every $i\geq2$ there exist an infinite subset $I^i\sub I^{i-1}$ and gauge transformations $g^i_\nu\in W^{2,p}(X^1,G)$, for $\nu\in I^i$, such that $X^i\sub\Om_\nu$, for every $\nu\in I^i$, and the sequence $A^i_\nu:=(g^i_\nu)^*(A_\nu|X^i)$ converges to some $W^{1,p}$-connection $A^i$ over $X^i$, weakly in $W^{1,p}$ and in $C^0$ on $X^i$. 

Let $i\in\N$. For $\nu\in I^i$ we define $h^i_\nu:=(g^{i+1}_\nu|X^i)^{-1}g^i_\nu$. An argument as in the proof of statement (\ref{prop:cpt bdd < infty}), using Lemma \ref{le:g smooth}, implies that the sequence $(h^i_\nu)_{\nu\in I^i}$ has a subsequence that converges to some gauge transformation $h^i\in W^{2,p}(X^i,G)$, weakly in $W^{2,p}$ on $X^i$. 

Repeating the construction in the proof of statement (\ref{prop:cpt bdd < infty}) and using the weak $W^{1,p}$- and strong $C^0$-convergence of $A^i_\nu$ on $X^i$, we obtain $\nu_i\geq i+1$ and $g_i\in W^{2,p}(\R^2\wo Z,G)$, for $i\in\N$, such that $\nu_i\in I^{i+1}$, and $g_i^*A_{\nu_i}$ converges to some $W^{1,p}$-connection $A$ over $\R^2\wo Z$, weakly in $W^{1,p}$ and in $C^0$ on every compact subset of $\R^2\wo Z$.

Replacing the set $K$ by the compact set $GK$, we may assume w.l.o.g.~(without loss of generality) that $K$ is $G$-invariant. Hence passing to the subsequence $(\nu_i)_i$, we may assume w.l.o.g.~that $A_\nu$ converges to $A$, weakly in $W^{1,p}$ and in $C^0$ on every compact subset of $\R^2\wo Z$.
\begin{claim}\label{claim:hyp:prop:compactness delbar} The hypotheses of Proposition \ref{prop:compactness delbar} with $k=1$ are satisfied.  
\end{claim}
\begin{proof}[Proof of Claim \ref{claim:hyp:prop:compactness delbar}] Let $\Om\sub \R^2\wo Z$ be an open subset with compact closure, and $\nu_0\in\N$ be such that $\Om\sub\Om_{\nu_0}$. Since the sequence $(A_\nu)$ converges to $A$, weakly in $W^{1,p}(\Om)$, we have 
\begin{equation}
  \label{eq:A nu W 1 p}
\sup_{\nu\geq\nu_0}\Vert A_\nu\Vert_{W^{1,p}(\Om)}<\infty.
\end{equation}
{\bf Condition (\ref{eq:u nu K})} is satisfied by the assumption (\ref{eq:u nu Om K}). We check {\bf condition (\ref{eq:du nu})}: We denote by $|\Om|$ the area of $\Om$ and choose a constant $C>0$ such that $X_\xi(x)\leq C|\xi|$, for every $x\in K$ and $\xi\in\g$. For $\nu\geq\nu_0$, we have
\begin{eqnarray}\nn\Vert du_\nu\Vert_{L^p(\Om)}&\leq&\Vert d_{A_\nu}u_\nu\Vert_{L^p(\Om)}+\Vert LA_\nu\Vert_{L^p(\Om)}\\
&\leq&|\Om|^{\frac1p}\Vert d_{A_\nu}u_\nu\Vert_{L^\infty(\Om)}+C\Vert A_\nu\Vert_{L^p(\Om)}.
\end{eqnarray}
Here the second inequality uses the hypothesis (\ref{eq:u nu Om K}). Combining this with (\ref{eq:sup e}) and (\ref{eq:A nu W 1 p}), condition (\ref{eq:du nu}) follows. 

{\bf Condition (\ref{eq:sup k p})} follows from the first vortex equation, (\ref{eq:A nu W 1 p}), (\ref{eq:du nu}), and hypothesis (\ref{eq:u nu Om K}). This proves Claim \ref{claim:hyp:prop:compactness delbar}.
\end{proof}
By Claim \ref{claim:hyp:prop:compactness delbar}, we may apply Proposition \ref{prop:compactness delbar}, to conclude that, passing to some subsequence, $u_\nu$ converges to some map $u\in W^{2,p}(\R^2\wo Z)$, weakly in $W^{2,p}$ and in $C^1$ on every compact subset of $\R^2\wo Z$. The pair $w:=(A,u)$ solves the first vortex equation. Furthermore, multiplying the second $R_\nu$-vortex equation with $R_\nu^{-2}$, it follows that $\mu\circ u=0$. This means that $w$ is an $\infty$-vortex. By Proposition \ref{prop:bar del J} below the map $Gu:\R^2\wo Z\to\BAR M$ is $\bar J$-holomorphic. Hence it is smooth. It follows that there exists a gauge transformation $g\in W^{2,p}(\R^2\wo Z,G)$ such that $g^*(A,u)$ is smooth. (We obtain such a $g$ from a smooth lift of the map $Gu$ to a map $\R^2\wo Z\to\mu^{-1}(0)$. Such a lift exists, since by hypothesis, $G$ is connected.) Regauging $A_\nu$ by $g$, statement (\ref{prop:cpt bdd = infty}) follows. This completes the proof of Proposition \ref{prop:cpt bdd}.  
\end{proof}
\noindent{\bf Remark.} One can try to circumvent the patching argument for the gauge transformations in this proof by choosing an extension $\wt g^i_\nu$ of $g^i_\nu$ to $\R^2\wo Z$, and defining $g_\nu:=\wt g^\nu_\nu$. However, the sequence $(g_\nu)$ does not have the required properties, since $g_\nu^*w_\nu$ does not necessarily converge on compact subsets of $\R^2\wo Z$. The reason is that for $j>i$ the transformation $g^j_\nu$ does in general not restrict to $g^i_\nu$ on $X^i$. $\Box$\\

\noindent{\bf Remark.} It is not clear if in the case $R_0=\infty$ the $g_\nu$'s can be chosen in such a way that $g_\nu^*w_\nu$ converges in $C^\infty$ on every compact subset of $\R^2\wo Z$. To prove this, a possible approach is to fix an open subset of $\R^2$ with smooth boundary and compact closure, which is contained in $\R^2\wo Z$. We can now try mimic the proof of \cite[Theorem 3.2]{CGMS}. In Step 3 of that proof the first and second vortex equations (and relative Coulomb gauge) are used iteratively in an alternating way. This iteration fails in our setting, because of the factor $R_\nu^2$ in the second vortex equations, which converges to $\infty$ by assumption. $\Box$\\

The next ingredient of the proof of Proposition \ref{prop:cpt mod} is the following. Recall the definition (\ref{eq:Emin}) of $\Emin$. The next result shows that if the energy densities of a sequence of rescaled vortices are not uniformly bounded on some compact subset $Q$, then at least the energy $\Emin$ is lost at some point in $Q$. 
\begin{prop}[Quantization of energy loss]\label{prop:quant en loss} Assume that $(M,\om)$ is aspherical. Let $\Om\sub\R^2$ be an open subset, $0<R_\nu<\infty$ a sequence such that $\inf_\nu R_\nu>0$, and $w_\nu\in\WWW^p(\Om)$ an $R_\nu$-vortex, for $\nu\in\N$. Assume that there exists a compact subset $K\sub M$ such that $u_\nu(\Om)\sub K$ for every $\nu$ and that $\sup_\nu E^{R_\nu}(w_\nu)<\infty$. Then the following conditions hold.
\begin{enui}
\item \label{prop:hard R e} For every compact subset $Q\sub\Om$ we have
\[\sup_\nu R_\nu^{-2}\Vert e_{w_\nu}^{R_\nu}\Vert_{C^0(Q)}<\infty.\]
\item \label{prop:hard limsup} If there exists a compact subset $Q\sub \Om$ such that $\sup_\nu||e_{w_\nu}^{R_\nu}||_{C^0(Q)}=\infty$ then there exists $z_0\in Q$ with the following property. For every $\eps>0$ so small that $B_\eps(z_0)\sub \Om$ we have
\begin{equation}
\label{eq:limsup Emin}\limsup_{\nu\to\infty} E^{R_\nu}(w_\nu,B_\eps(z_0))\geq\Emin.
\end{equation}
\end{enui}
\end{prop}
The proof of Proposition \ref{prop:quant en loss} is built on a bubbling argument, as in step 5 in the proof of Theorem A in \cite{GS}. The idea is that under the assumption of (\ref{prop:hard limsup}) we may construct either a $\bar J$-holomorphic sphere in $\BAR M$ or a vortex over $\R^2$, by rescaling the sequence $w_\nu$ in a ``hard way''. This means that after rescaling the energy densities are bounded. We need the following two lemmata. 
\begin{lemma}[Hofer]\label{lemma:Hofer} Let $(X,d)$ be a metric space, $f:X\to[0,\infty)$ a continuous function, $x\in X$, and $\delta>0$. Assume that the closed ball $\bar B_{2\delta}(x)$ is complete. Then there exists $\xi\in X$ and a number $0<\eps\leq\delta$ such that
\[d(x,\xi)<2\delta,\qquad\sup_{B_\eps(\xi)}f\leq2f(\xi),\qquad\eps f(\xi)\geq\delta f(x).\]
\end{lemma}

\begin{proof}See \cite[Lemma 4.6.4]{MS}.
\end{proof}
The next lemma ensures that for a suitably convergent sequence of rescaled vortices in the limit $\nu\to\infty$ no energy gets lost on any compact set. Apart from Proposition \ref{prop:quant en loss}, it will also be used in the proof of Propositions \ref{prop:cpt mod} and \ref{prop:soft}, and Theorem \ref{thm:bubb}. 
\begin{lemma}[Convergence of energy densities]\label{le:conv e} Let $(\Si,\om_\Si,j)$ be a surface without boundary, equipped with an area form and a compatible complex structure, $R_\nu\in[0,\infty)$, $\nu\in\N$, a sequence of numbers that converges to some $R_0\in[0,\infty]$, and for $\nu\in\N_0$ let $w_\nu:=(A_\nu,u_\nu)\in\WWW^p(\Si)$ an $R_\nu$-vortex. Assume that on every compact subset of $\Si$, $A_\nu$ converges to $A_0$ in $C^0$ and $u_\nu$ converges to $u_0$ in $C^1$. Then we have
\begin{equation}\label{eq:e R nu W nu e f} e^{R_\nu}_{w_\nu}\to e^{R_0}_{w_0}
\end{equation}
in $C^0$ on every compact subset of $\Si$. 
\end{lemma}
\begin{proof}[Proof of Lemma \ref{le:conv e}]\setcounter{claim}{0} In the {\bf case $R_0<\infty$} the statement of the lemma is a consequence of equality (\ref{eq:e R W 1 2}).

Consider the {\bf case $R_0=\infty$}. It follows from our standing hypothesis (H) that there exists a constant $\delta>0$ such that $G$ acts freely on 
\[K:=\{x\in M\,|\,|\mu(x)|\leq \delta\}.\] 
Properness of $\mu$ implies that $K$ is compact. 

Let $Q\sub\Si$ be a compact subset. The convergence of $u_\nu$ and the fact $\mu\circ u_0=0$ imply that for $\nu$ large enough, we have $u_\nu(Q)\sub K$. Furthermore, our hypotheses about the convergence of $A_\nu$ and $u_\nu$ imply that $\sup_\nu\Vert d_{A_\nu}u_\nu\Vert_{C^0(Q)}<\infty$. Finally, since $K$ is compact and $G$ acts freely on it, we have
\[\sup\left\{\frac{|\xi|}{|L_x\xi|}\,\Big|\,x\in K,\,0\neq\xi\in\g\right\}<\infty.\]
Therefore, we may apply Lemma \ref{le:mu u} below, to conclude that 
\[\sup_QR_\nu^{2-2/p}|\mu\circ u_\nu|<\infty.\]
Since $p>2,R_\nu\to\infty$, and $e^\infty_{w_0}=\frac12|d_{A_0}u_0|^2$, the convergence (\ref{eq:e R nu W nu e f}) follows. This completes the proof of Lemma \ref{le:conv e}.
\end{proof}

In the proof of Proposition \ref{prop:quant en loss} we will also use the following.
\begin{rmk}\label{rmk:bar J}\rm Let $(A,u)\in\WWW^p(\R^2)$ be an $\infty$-vortex, i.e., a solution of the equations $\bar\dd_{J,A}(u)=0$ and $\mu\circ u=0$. By Proposition \ref{prop:bar del J} below the map $Gu:\R^2\to\BAR M=\mu^{-1}(0)/G$ is $\bar J$-holomorphic, and $E^\infty(A,u)=E(\bar u)$. If this energy is finite, then by removal of singularities the map $\bar u$ extends to a $\bar J$-holomorphic map $\bar u:S^2\to\BAR M$. (See for example \cite[Theorem 4.1.2]{MS}.) It follows that $E^\infty(w)\geq\Emin$, provided that $E^\infty(w)>0$. $\Box$
\end{rmk}

\begin{proof}[Proof of Proposition \ref{prop:quant en loss}]\setcounter{claim}{0} We write $(A_\nu,u_\nu):=w_\nu$. Consider the function 
\[f_\nu:=|d_{A_\nu}u_\nu|+R_\nu|\mu\circ u_\nu|:\Om\to\R.\]
\begin{claim}\label{claim:w 0} Suppose that the hypotheses of Proposition \ref{prop:quant en loss} are satisfied and that there exists a sequence $z_\nu\in\Om$ that converges to some $z_0\in\Om$, such that $f_\nu(z_\nu)\to\infty$. Then there exists
\begin{equation}
  \label{eq:r 0 limsup}0<r_0\leq\limsup_{\nu\to\infty} \frac{R_\nu}{f_\nu(z_\nu)}\,(\leq\infty)
\end{equation}
and an $r_0$-vortex $w_0\in\WWW(\R^2)$, such that
\begin{equation}
  \label{eq:limsup E w 0}0<E^{r_0}(w_0)\leq\limsup_{\nu\to\infty} E^{R_\nu}(w_\nu,B_\eps(z_0)),
\end{equation}
for every $\eps>0$ so small that $B_\eps(z_0)\sub \Om$.
\end{claim}
\begin{proof}[Proof of Claim \ref{claim:w 0}] {\bf Construction of $r_0$:} We define $\de_\nu:=f_\nu(z_\nu)^{-\frac12}$. For $\nu$ large enough we have $\bar B_{2\de_\nu}(z_\nu)\sub\Om$. We pass to some subsequence such that this holds for every $\nu$. By Lemma \ref{lemma:Hofer}, applied with $(f,x,\de):=(f_\nu,z_\nu,\de_\nu)$, there exist $\ze_\nu\in B_{2\de_\nu}(z_0)$ and $\eps_\nu\leq\de_\nu$, such that 
\begin{eqnarray}
\label{eq:ze z nu}|\ze_\nu-z_\nu|&<&2\de_\nu,\\
\label{eq:sup B}\sup_{B_{\eps_\nu}(\ze_\nu)}f_\nu&\leq&2f_\nu(\ze_\nu),\\
\label{eq:eps nu}\eps_\nu f_\nu(\ze_\nu)&\geq&f_\nu(z_\nu)^{\frac12}.
\end{eqnarray}
Since by assumption $f_\nu(z_\nu)\to\infty$, it follows from (\ref{eq:ze z nu}) that the sequence $\ze_\nu$ converges to $z_0$. We define 
\begin{eqnarray*}\nn&c_\nu:=f_\nu(\ze_\nu),\quad\wt\Om_\nu:=\big\{c_\nu(z-\ze_\nu)\,\big|\,z\in\Om\big\},&\\
\nn&\phi_\nu:\wt\Om_\nu\to\Om,\quad\phi_\nu(\wt z):=c_\nu^{-1}\wt z+\ze_\nu,&\\
\nn&\wt w_\nu:=\phi_\nu^*w_\nu=(\phi_\nu^*A_\nu,u_\nu\circ\phi_\nu),\quad\wt R_\nu:=c_\nu^{-1}R_\nu.&
\end{eqnarray*}
Note that $\wt w_\nu$ is an $\wt R_\nu$-vortex. Passing to some subsequence we may assume that $\wt R_\nu$ converges to some $r_0\in[0,\infty]$. Since $\eps_\nu\leq\de_\nu=f_\nu(z_\nu)^{-\frac12}$ it follows from (\ref{eq:eps nu}) that $f_\nu(z_\nu)\leq f_\nu(\ze_\nu)$. It follows that the {\bf second inequality in (\ref{eq:r 0 limsup})} holds for the original sequence. 

{\bf Construction of $w_0$:} We check the conditions of Proposition \ref{prop:cpt bdd} with $(Z,\Om_\nu):=\big(\emptyset,\bigcup_{\nu'=1,\ldots,\nu}\wt\Om_\nu\big)$ and $R_\nu,w_\nu$ replaced by $\wt R_\nu,\wt w_\nu$: {\bf Condition (\ref{eq:u nu Om K})} is satisfied by hypothesis. 

We check {\bf condition (\ref{eq:sup e})}: A direct calculation involving (\ref{eq:sup B}) shows that 
\begin{equation}\label{eq:d wt A nu wt u nu}|d_{\wt A_\nu}\wt u_\nu|+\wt R_\nu|\mu\circ\wt u_\nu|=c_\nu^{-1}f_\nu\circ\phi_\nu\leq2,\quad\textrm{on }B_{\eps_\nu c_\nu}(0).\end{equation}
It follows from (\ref{eq:eps nu}) and the fact $f_\nu(z_\nu)\to\infty$, that $\eps_\nu c_\nu\to\infty$. Combining this with (\ref{eq:d wt A nu wt u nu}), condition (\ref{eq:sup e}) follows, for every compact subset $Q\sub\R^2$. 

Therefore, applying Proposition \ref{prop:cpt bdd}, there exists an $r_0$-vortex $w_0=(A_0,u_0)\in\WWW(\R^2)$ and, passing to some subsequence, there exist gauge transformations $g_\nu\in W^{2,p}(\R^2,G)$, with the following property. For every compact subset $Q\sub\R^2$, $g_\nu^*\wt A_\nu$ converges to $A_0$ in $C^0$ on $Q$, and $g_\nu^{-1}\wt u_\nu$ converges to $u_0$ in $C^1$ on $Q$. 

We prove the {\bf first inequality in (\ref{eq:limsup E w 0})}: By Lemma \ref{le:conv e} we have 
\begin{equation}
\label{eq:e R0 w0}e^{\wt R_\nu}_{\wt w_\nu}=e^{\wt R_\nu}_{g_\nu^*\wt w_\nu}\to e^{r_0}_{w_0},
\end{equation}
in $C^0(Q)$ for every compact subset $Q\sub\R^2$. Since $e^{\wt R_\nu}_{\wt w_\nu}(0)=c_\nu^{-2}e^{R_\nu}_{w_\nu}(\ze_\nu)\geq1/2$, it follows that $e^{r_0}_{w_0}(0)\geq1/2$. This implies that $E^{r_0}(w_0)>0$. This proves the first inequality in (\ref{eq:limsup E w 0}).

We prove the {\bf second inequality in (\ref{eq:limsup E w 0})}: Let $\eps>0$ be so small that $B_\eps(z_0)\sub \Om$, and $\de>0$. It follows from (\ref{eq:e R0 w0}) that $E^{r_0}(w_0)\leq\sup_\nu E^{R_\nu}(w_\nu)$. By hypothesis this supremum is finite. Hence there exists $R>0$ such that $E^{r_0}(w_0,\R^2\wo B_R)<\de.$ Since $E^{R_\nu}(w_\nu,B_{c_\nu^{-1}R}(\ze_\nu))=E^{\wt R_\nu}(\wt w_\nu,B_R)$, the convergence (\ref{eq:e R0 w0}) implies that
\begin{equation}\label{eq:E R nu}\lim_{\nu\to\infty} E^{R_\nu}(w_\nu,B_{c_\nu^{-1}R}(\ze_\nu))=E^{r_0}(w_0,B_R)>E^{r_0}(w_0)-\delta.
\end{equation}
On the other hand, since $c_\nu\to\infty$ and $\ze_\nu \to z_0$, for $\nu$ large enough the ball $B_{c_\nu^{-1}R}(\ze_\nu)$ is contained in $B_\eps(z_0)$. Combining this with (\ref{eq:E R nu}), we obtain 
\[\limsup_{\nu\to\infty} E^{R_\nu}(w_\nu,B_\eps(z_0))\geq E^{r_0}(w_0)-\delta.\]
Since this holds for every $\delta>0$, the second inequality in (\ref{eq:limsup E w 0}) (for the original sequence) follows.

It remains to prove the {\bf first inequality in (\ref{eq:r 0 limsup})}, i.e., that $r_0>0$. Assume by contradiction that $r_0=0$. For a map $u\in C^\infty(\R^2,M)$ we denote by 
\[E(u):=\frac12\int_{\R^2}|du|^2\]
its (Dirichlet-)energy. (Here the norm is taken with respect to the metric $\om(\cdot,J\cdot)$ on $M$.) By the second $R$-vortex equation with $R:=0$ we have $F_{A_0}=0$. Therefore, by Proposition \ref{prop:ka 0} there exists $h\in C^\infty(\R^2,G)$ such that $h^*A_0=0$. By the first vortex equation the map $u'_0:=h^{-1}u_0:\R^2=\C\to M$ is $J$-holomorphic. Let $\eps>0$ be such that $B_\eps(z_0)\sub\Om$. Using the second inequality in (\ref{eq:limsup E w 0}), we have
\[E(u'_0)=E^0(w_0)\leq\limsup_{\nu\to\infty} E^{R_\nu}(w_\nu,B_\eps(z_0)).\]
Combining this with the hypothesis $\sup_\nu E^{R_\nu}(w_\nu,\Om)<\infty$, it follows that $E(u'_0)<\infty$. Hence by removal of singularities (see e.g.~\cite[Theorem 4.1.2]{MS}), it follows that $u'_0$ extends to a smooth $J$-holomorphic map $v:S^2\to M$. By the first inequality in (\ref{eq:limsup E w 0}) we have $\int_{S^2}v^*\om=E(v)=E^0(w_0)>0$. This contradicts asphericity of $(M,\om)$. Hence $r_0$ must be positive. This concludes the proof of Claim \ref{claim:w 0}.  
\end{proof}
{\bf Statement (\ref{prop:hard R e})} of Proposition \ref{prop:quant en loss} follows from Claim \ref{claim:w 0}, considering a sequence $z_\nu\in Q$, such that $f_\nu(z_\nu)=\Vert f_\nu\Vert_{C^0(Q)}$, and using (\ref{eq:r 0 limsup}).

We prove {\bf statement (\ref{prop:hard limsup})}. Assume that there exists a compact subset $Q\sub \Om$ such that $\sup_\nu||e^{R_\nu}_{w_\nu}||_{C^0(Q)}=\infty$. Let $z_\nu\in Q$ be such that $f_\nu(z_\nu)\to\infty$. We choose a pair $(r_0,w_0)$ as in Claim \ref{claim:w 0}. Using the first inequality in (\ref{eq:limsup E w 0}) and Remark \ref{rmk:bar J} (in the case $r_0=\infty$), we have $E^{r_0}(w_0)\geq\Emin$. Combining this with the second inequality in (\ref{eq:limsup E w 0}), inequality (\ref{eq:limsup Emin}) follows. This proves (\ref{prop:hard limsup}) and concludes the proof of Proposition \ref{prop:quant en loss}.
\end{proof}
We are now ready for the proof of Proposition \ref{prop:cpt mod}.
\begin{proof}[Proof of Proposition \ref{prop:cpt mod}]\setcounter{claim}{0}\label{prop:cpt mod proof} We abbreviate $e_\nu:=e_{w_\nu}^{R_\nu}$. 

\begin{claim}\label{claim:Z ell} For every $\ell\in\N\cup\{0\}$ there exists a finite subset $Z_\ell\sub\R^2$ such that the following holds. If $R_0<\infty$ then we have $Z_\ell=\emptyset$. Furthermore, if $|Z_\ell|<\ell$ then we have 
\begin{equation}
\label{eq:sup Vert e nu}\sup_{\nu\in\N}\big\{\Vert e_\nu\Vert_{C^0(Q)}\,\big|\,Q\sub B_{r_\nu}\big\}<\infty,
\end{equation}
for every compact subset $Q\sub \R^2\wo Z_\ell$. Moreover, for every $z_0\in Z_\ell$ and every $\eps>0$ the inequality (\ref{eq:limsup Emin}) holds. 
\end{claim}
\begin{proof}[Proof of Claim \ref{claim:Z ell}] For $\ell=0$ the assertion holds with $Z_0:= \emptyset$. We prove by induction that it holds for every $\ell\geq1$. Fix $\ell\geq1$. By induction hypothesis there exists a finite subset $Z_{\ell-1}\sub\R^2$ such that the assertion with $\ell$ replaced by $\ell-1$ holds. If (\ref{eq:sup Vert e nu}) is satisfied for every compact subset $Q\sub \R^2\wo Z_{\ell-1}$, then the statement for $\ell$ holds with $Z_\ell:=Z_{\ell-1}$. 

Hence assume that there exists a compact subset $Q\sub \R^2\wo Z_{\ell-1}$, such that (\ref{eq:sup Vert e nu}) does not hold. It follows from the induction hypothesis that 
\begin{equation}
  \label{eq:Z ell 1}
|Z_{\ell-1}|\geq\ell-1.  
\end{equation}
Applying Proposition \ref{prop:quant en loss}, by statement (\ref{prop:hard limsup}) of that proposition there exists a point $z_0\in Q$ such that inequality (\ref{eq:limsup Emin}) holds, for every $\eps>0$. We set $Z_\ell:=Z_{\ell-1}\cup \{z_0\}$. 

It follows from the fact that (\ref{eq:sup Vert e nu}) does not hold and condition (\ref{prop:hard R e}) of Proposition \ref{prop:quant en loss} that $R_0=\lim_{\nu\to\infty}R_\nu=\infty$. Furthermore, since $z_0\in Q\sub \R^2\wo Z_{\ell-1}$, (\ref{eq:Z ell 1}) implies that $|Z_\ell|\geq\ell$. It follows that the statement of Claim \ref{claim:Z ell} for $\ell$ is satisfied. By induction, Claim \ref{claim:Z ell} follows.
\end{proof}
We fix an integer $\ell>\sup_\nu E^{R_\nu}(w_\nu,B_{r_\nu})/\Emin$ and a finite subset $Z:=Z_\ell\sub\R^2$ that satisfies the conditions of Claim \ref{claim:Z ell}. It follows from the inequality (\ref{eq:limsup Emin}) that $\ell>|Z|$. Hence by the statement of Claim \ref{claim:Z ell}, the hypothesis (\ref{eq:sup e}) of Proposition \ref{prop:cpt bdd} is satisfied with $\Om_\nu:=B_{r_\nu}\wo Z$. Applying that result and passing to some subsequence, there exist an $R_0$-vortex $w_0\in\WWW(\R^2\wo Z)$ and gauge transformations $g_\nu\in W^{2,p}_\loc(\R^2\wo Z,G)$, such that the {\bf statements (\ref{prop:cpt mod:<},\ref{prop:cpt mod:=})} of Proposition \ref{prop:cpt mod} are satisfied. (Here we use that $Z=\emptyset$ if $R_0<\infty$.)

We prove {\bf statement (\ref{prop:cpt mod lim nu E eps})}. Passing to some ``diagonal'' subsequence, the limit $\lim_{\nu\to\infty}E^{R_\nu}(w_\nu,B_{1/i}(z))$ exists, for every $i\in\N$ and $z\in Z$. Let now $z\in Z$ and $\eps>0$. We choose $i\in\N$ bigger than $\eps^{-1}$. For $0<r<R$ we denote 
\[A(z,r,R):=\bar B_R(z)\wo B_r(z).\] 
By Lemma \ref{le:conv e} the limit $\lim_{\nu\to\infty}E^{R_\nu}\big(w_\nu,A(z,1/i,\eps)\big)$ exists and equals $E^{R_0}(w_0,A(z,1/i,\eps))$. It follows that the limit $E_z(\eps):=\lim_{\nu\to\infty}E^{R_\nu}(w_\nu,B_\eps(z))$ exists. Inequality (\ref{eq:limsup Emin}) implies that $E_z(\eps)\geq\Emin$. Since $E^{R_0}(w_0,A(z,1/i,\eps))$ depends continuously on $\eps$, the same holds for $E_z(\eps)$. This proves statement (\ref{prop:cpt mod lim nu E eps}) and completes the proof of Proposition \ref{prop:cpt mod}.
\end{proof}
\noindent{\bf Remark.} In the above proof the set of bubbling points $Z$ is constructed by ``terminating induction''. Intuitively, this is induction over the number of bubbling points. The ``auxiliary index'' $\ell$ in Claim \ref{claim:Z ell} is needed to make this idea precise. Inequality (\ref{eq:limsup Emin}) ensures that the ``induction stops''.
\section{Soft rescaling}\label{sec:soft}
The next proposition will be used inductively in the proof of the main result to find the next bubble in the bubbling tree, at a bubbling point of a given sequence of rescaled vortices. It is an adaption of \cite[Proposition 4.7.1.]{MS} to vortices.
\begin{prop}[Soft rescaling]\label{prop:soft} Assume that $(M,\om)$ is aspherical. Let $r>0$, $z_0\in\R^2$, $R_\nu>0$ a sequence that converges to $\infty$, $p>2$, and for every $\nu\in\N$ let $w_\nu:=(A_\nu,u_\nu)\in\WWW^p(B_r(z_0))$ be an $R_\nu$-vortex, such that the following conditions are satisfied.
\begin{enua}\item\label{prop:soft K} There exists a compact subset $K\sub M$ such that $u_\nu(B_r(z_0))\sub K$ for every $\nu$.
\item \label{prop:soft E} For every $0<\eps\leq r$ the limit $E(\eps):=\lim_{\nu\to\infty} E^{R_\nu}(w_\nu,B_\eps(z_0))$ exists and $\Emin\leq E(\eps)<\infty$. Furthermore, the function 
\begin{equation}\label{eq:eps E eps}(0,r]\ni\eps\mapsto E(\eps)\in\R
\end{equation} 
is continuous.
\end{enua}
Then there exist $R_0\in\{1,\infty\}$, a finite subset $Z\sub \R^2$, and an $R_0$-vortex $w_0:=(A_0,u_0)\in\WWW(\R^2\wo Z)$, and passing to some subsequence, there exist sequences $\eps_\nu>0$, $z_\nu\in \R^2$, and $g_\nu\in W^{2,p}_\loc(\R^2\wo Z,G)$, such that, defining
\[\phi_\nu:\R^2\to\R^2,\quad\phi_\nu(\wt z):=\eps_\nu\wt z+z_\nu,\]
the following conditions hold.
\begin{enui}
\item\label{prop:soft Z} If $R_0=1$ then $Z=\emptyset$ and $E(w_0)>0$. If $R_0=\infty$ and $E^\infty(w_0)=0$ then $|Z|\geq2$. 
\item\label{prop:soft eps} The sequence $z_\nu$ converges to $z_0$. Furthermore, if $R_0=1$ then $\eps_\nu= R_\nu^{-1}$ for every $\nu$, and if $R_0=\infty$ then $\eps_\nu$ converges to 0 and $\eps_\nu R_\nu$ converges to $\infty$. 
\item\label{prop:soft conv} If $R_0=1$ then the sequence $g_\nu^*\phi_\nu^*w_\nu$ converges to $w_0$ in $C^\infty$ on every compact subset of $\R^2\wo Z$. Furthermore, if $R_0=\infty$ then on every compact subset of $\R^2\wo Z$, the sequence $g_\nu^*\phi_\nu^*A_\nu$ converges to $A_0$ in $C^0$, and the sequence $g_\nu^{-1}(u_\nu\circ\phi_\nu)$ converges to $u_0$ in $C^1$. 
\item\label{prop:soft lim nu E eps} Fix $z\in Z$ and a number $\eps_0>0$ such that $B_{\eps_0}(z)\cap Z=\{z\}$. Then for every $0<\eps<\eps_0$ the limit 
\[E_z(\eps):=\lim_{\nu\to\infty}E^{\eps_\nu R_\nu}\big(\phi_\nu^*w_\nu,B_\eps(z)\big)\]
exists and $\Emin\leq E_z(\eps)<\infty$. Furthermore, the function $(0,\eps_0)\ni\eps\mapsto E_z(\eps)\in\R$ is continuous.
\item \label{prop:soft en} We have
\begin{equation}
\label{eq:en cons}\lim_{R\to\infty}\limsup_{\nu\to\infty} E^{R_\nu}\big(w_\nu,B_{R^{-1}}(z_0)\wo B_{R\eps_\nu}(z_\nu)\big)=0.
\end{equation}
\end{enui}
\end{prop}
\noindent{\bf Remarks.} In the proof of Theorem \ref{thm:bubb}, condition (\ref{prop:soft Z}) will guarantee that the new bubble is stable. Condition (\ref{prop:soft lim nu E eps}) will be used to prove that the construction of the bubbling tree terminates after finitely many steps. Finally, condition (\ref{prop:soft en}) will ensure that no energy is lost between the old and new bubble. 

Note that in condition (\ref{prop:soft conv}) the pullback $\phi_\nu^*w_\nu$ is defined over the set $\phi_\nu^{-1}(B_r(z_0))$. $\Box$\\

The proof of Proposition \ref{prop:soft} is given on page \pageref{proof prop:soft}. It is based on the following result, which states that the energy of a vortex on an annulus is concentrated near the ends, provided that it is small enough. For $0\leq r,R\leq \infty$ we denote the open annulus around 0 with radii $r,R$ by 
\[A(r,R):=B_R\wo \bar B_r.\]
Note that $A(r,\infty)=\R^2\wo\bar B_r$, and $A(r,R)=\emptyset$ in the case $r\geq R$. We define
\[d:\bigcup_MM\x M\to [0,\infty]\]
to be the distance function induced by the Riemannian metric $\om(\cdot,J\cdot)$. (If $M$ is disconnected then $d$ attains the value $\infty$.) We define 
\begin{equation}
  \label{eq:bar d}\bar d:\bigcup_MM/G\x M/G\to [0,\infty],\qquad \bar d(\BAR x,\BAR y):=\min_{x\in\BAR x,\,y\in\BAR y}d(x,y).
\end{equation}
By Lemma \ref{le:metr} below this is a distance function on $M/G$ which induces the quotient topology. 
\begin{prop}[Energy concentration near ends] \label{prop:en conc} There exists a constant $r_0>0$ such that for every compact subset $K\sub M$ and every $\eps>0$ there exists a constant $E_0$, such that the following holds. Assume that $r_0\leq r,R\leq\infty$, $p>2$, and $w:=(u,A)\in\WWW^p(A(r,R))$ is a vortex (with respect to $(\om_0,i)$), such that 
\begin{eqnarray}\nn&u(A(r,R))\sub K,&\\
\label{eq:E w A E 0}&E(w)=E\big(w,A(r,R)\big)\leq E_0.&
\end{eqnarray} 
Then we have
\begin{eqnarray}
\label{eq:E w A ar}&E\big(w,A(ar,a^{-1}R)\big)\leq4a^{-2+\eps}E(w),\quad\forall a\geq2,&\\
\label{eq:sup z z' bar d}&\sup_{z,z'\in A(ar,a^{-1}R)}\bar d(Gu(z),Gu(z'))\leq100a^{-1+\eps}\sqrt{E(w)},\,\,\forall a\geq4.&
\end{eqnarray}
(Here $Gx\in M/G$ denotes the orbit of a point $x\in M$.)
\end{prop}
Note that in the case $a>\sqrt{R/r}$ we have $A(ar,a^{-1}R)=\emptyset$, and hence the statement of the proposition is void. The proof of this is modelled on the proof of \cite[Theorem 1.3]{ZiA}, which in turn is based on the proof of \cite[Proposition 11.1]{GS}. It is based on an isoperimetric inequality for the invariant symplectic action functional (Theorem \ref{thm:isoperi} in Appendix \ref{sec:add}). It also relies on an identity relating the energy of a vortex over a compact cylinder with the actions of its end-loops (Proposition \ref{prop:en act} below). The proof of (\ref{eq:sup z z' bar d}) also uses the following remark.
\begin{rmk}\rm\label{rmk:BAR d BAR x 0 BAR x 1} Let $\big(M,\lan\cdot,\cdot\ran_M\big)$ be a Riemannian manifold, $G$ a compact Lie group that acts on $M$ by isometries, $P$ a principal $G$-bundle over $[0,1]$, $A\in\A(P)$ a connection, and $u\in C^\infty_G(P,M)$ a map. We define 
\[\ell(A,u):=\int_0^1|d_Au|dt,\]
where the norm is taken with respect to the standard metric on $[0,1]$ and $\lan\cdot,\cdot\ran_M$. Furthermore, we define $\bar u:[0,1]\to M/G$ by $\bar u(t):=Gu(p)$, where $p\in P$ is any point over $t$. We denote by $d$ the distance function induced by $\lan\cdot,\cdot\ran_M$, and define $\bar d$ as in (\ref{eq:bar d}). Then for every pair of points $\BAR x_0,\BAR x_1\in M/G$, we have
\[\bar d(\BAR x_0,\BAR x_1)\leq\inf\big\{\ell(A,u)\,\big|\,(P,A,u)\textrm{ as above: }\bar u(i)=\BAR x_i,\,i=0,1\big\}.\]
This follows from a straight-forward argument. $\Box$
\end{rmk}
\begin{proof}[Proof of Proposition \ref{prop:en conc}]\setcounter{claim}{0} For every subset $X\sub M$ we define 
\[m_X:=\inf\big\{|L_x\xi|\,\big|\,x\in X,\,\xi\in\g:\,|\xi|=1\big\},\]
where the norms are with respect to $\om(\cdot,J\cdot)$ and $\lan\cdot,\cdot\ran_\g$. We set
\begin{equation}\label{eq:R 0 2 pi m}r_0:=m_{\mu^{-1}(0)}^{-1}.
\end{equation}
Let $K\sub M$ be a compact subset and $\eps>0$. Replacing $K$ be $GK$, we may assume w.l.o.g.~that $K$ is $G$-invariant. An elementary argument using our standing hypothesis (H) shows that there exists a number $\de_0>0$ such that $G$ acts freely on $K':=\mu^{-1}(\bar B_{\de_0})$, and 
\begin{equation}\label{eq:m K 1 eps}m_{K'}\geq\sqrt{1-\eps/2}m_{\mu^{-1}(0)}.
\end{equation}
We choose a constant $\de$ as in Theorem \ref{thm:isoperi}, corresponding to $\lan\cdot,\cdot\ran_M:=\om(\cdot,J\cdot),K',c:=\frac1{2-\eps}$. Shrinking $\de$ we may assume that it satisfies the condition of Proposition \ref{prop:en act} (Energy action identity) for $K'$. We choose a constant $\wt E_0>0$ as in Lemma \ref{le:a priori} below (called $E_0$ there), corresponding to $K$. We define 
\begin{equation}\label{eq:E 0 de 0 de}E_0:=\min\Big\{\wt E_0,\frac\pi{32}r_0^2\de_0^2,\frac{\de^2}{128\pi}\Big\}.
\end{equation}
Assume that $r,R,p,w$ are as in the hypothesis. Without loss of generality, we may assume that $r<R$. 

Consider first the {\bf case $R<\infty$}, and assume that $w$ extends to a smooth vortex on the compact annulus of radii $r$ and $R$. We show that {\bf inequality (\ref{eq:E w A ar})} holds. We define the function 
\begin{equation}\label{eq:E s}E:[0,\infty),\quad E(s):=E\big(w,A(re^s,Re^{-s})\big).
\end{equation}
\begin{claim}\label{claim:d ds E}For every $s\in[\log2,\log(R/r)/2)$ we have
\begin{equation}\label{eq:d ds E}\frac d{ds}E(s)\leq-(2-\eps)E(s).
\end{equation}
\end{claim}
\begin{proof}[Proof of Claim \ref{claim:d ds E}] Using the fact $r\geq r_0$ and (\ref{eq:E w A E 0},\ref{eq:E 0 de 0 de}), it follows from Lemma \ref{le:a priori} below (with ``$r$''$:=|z|/2$) that
\begin{equation}\label{eq:e w z 0 8}e_w(z)\leq\min\Big\{\de_0^2,\frac{\de^2}{4\pi^2|z|^2}\Big\},\quad\forall z\in A(2r,R/2).\end{equation}
We define
\begin{eqnarray*}&\Si_s:=\big(s+\log r,-s+\log R\big)\x S^1,\,\forall s\in\R,&\\
&\phi:\Si_0\to\R^2=\C,\,\phi(z):=e^z,\quad\wt w:=(\wt A,\wt u):=\phi^*w.&
\end{eqnarray*}
(Here we identify $\Si_0\iso\C/\sim$, where $z\sim z+2\pi in$, for every $n\in\Z$.) Let $s_0\in\big[\log(2r),\log(R/2)\big]$. Combining (\ref{eq:e w z 0 8}) with the fact $|\mu\circ u|\leq\sqrt{e_w}$ and Remark \ref{rmk:BAR d BAR x 0 BAR x 1}, it follows that
\begin{equation}\label{eq:wt u bar ell}\wt u(s_0,t)\in K'=\mu^{-1}(\bar B_{\de_0}),\,\forall t\in S^1,\quad\bar\ell(G\wt u(s_0,\cdot))\leq\de.
\end{equation}
Hence the hypotheses of Theorem \ref{thm:isoperi} are satisfied with $K$ replaced by $K'$ and $c:=1/(2-\eps)$. By the statement of that result the loop $\wt u(s_0,\cdot)$ is admissible, and defining $\iota_{s_0}:S^1\to\Si_0$ by $\iota_{s_0}(t):=(s_0,t)$, we have
\begin{equation}\label{eq:A iota s 0}\big|\A\big(\iota_{s_0}^*\wt w\big)\big|\leq\frac1{2-\eps}\Vert \iota_{s_0}^*d_{\wt A}\wt u\Vert_2^2+\frac1{2m_{K'}^2}\big\Vert\mu\circ\wt u\circ\iota_{s_0}\big\Vert_2^2.
\end{equation}
Here $\A$ denotes the invariant symplectic action, as defined in appendix \ref{sec:add}. Furthermore, the $L^2$-norms are with respect to the standard metric on $S^1\iso\R/(2\pi\Z)$, the metric $\om(\cdot,J\cdot)$ on $M$, and the operator norm $|\cdot|_\op:\g^*\to\R$, induced by $\lan\cdot,\cdot\ran_\g$. 

By (\ref{eq:R 0 2 pi m},\ref{eq:m K 1 eps}) and the fact $2r\leq e^{s_0}$, we have
\begin{equation}\label{eq:frac 1 2 2 eps}\frac1{2-\eps}|\iota_{s_0}^*d_{\wt A}\wt u|_0^2+\frac1{2m_{K'}^2}\big|\mu\circ\wt u\circ\iota_{s_0}\big|^2\leq \frac1{2-\eps}e^{2s_0}e_w(e^{s_0+i\cdot}),\,\textrm{on }S^1.\end{equation}
Here the norm $|\cdot|_0$ is with respect to the standard metric on $S^1\iso\R/(2\pi\Z)$, and we used the fact $|\phi|_\op\leq|\phi|$ for $\phi\in\g^*$, where $|\cdot|$ denotes the norm induced by $\lan\cdot,\cdot\ran_\g$. We fix $s\in\big[\log2,\log(R/r)/2\big)$. Recalling (\ref{eq:E s}), we have $E(s)=\int_{\Si_s}e^{2s_0}e_w(e^{s_0+it})dt\,ds_0$. Combining this with (\ref{eq:A iota s 0},\ref{eq:frac 1 2 2 eps}), it follows that
\begin{equation}\label{eq:A iota}-\A\big(\iota_{-s+\log R}^*\wt w\big)+\A\big(\iota_{s+\log r}^*\wt w\big)\leq-\frac1{2-\eps}\frac d{ds}E(s).\end{equation}
Using (\ref{eq:wt u bar ell}), the hypotheses of Proposition \ref{prop:en act} are satisfied with $K$ replaced by $K'$. Applying that result, we have $E(s)=-\A\big(\iota_{-s+\log R}^*\wt w\big)+\A\big(\iota_{s+\log r}^*\wt w\big)$. Combining this with (\ref{eq:A iota}), inequality (\ref{eq:d ds E}) follows. This proves Claim \ref{claim:d ds E}.
\end{proof}
By Claim \ref{claim:d ds E} the derivative of the function $\big[\log2,\log(R/r)/2\big)\ni s\mapsto E(s)e^{(2-\eps)s}$ is non-positive, and hence this function is non-increasing. Inequality (\ref{eq:E w A ar}) follows.

{\bf We prove (\ref{eq:sup z z' bar d})}. Let $z\in A(4r,\sqrt{rR})$. Using 
(\ref{eq:E w A E 0}) and the fact $E_0\leq\wt E_0$, it follows from Lemma \ref{le:a priori} (with ``$r$''$:=|z|/2$) that 
\begin{equation}\label{eq:e w z}e_w(z)\leq\frac{32}{\pi|z|^2}E\big(w,B_{|z|/2}(z)\big).
\end{equation} 
We define $a:=|z|/(2r)$. Then $a\geq2$ and $B_{|z|/2}(z)$ is contained in $A(ar,a^{-1}R)$. Therefore, by (\ref{eq:E w A ar}) we have 
\[E\big(w,B_{|z|/2}(z)\big)\leq16r^{2-\eps}|z|^{-2+\eps}E(w).\]
Combining this with (\ref{eq:e w z}), the fact $|d_Au|(z)\leq\sqrt{2e_w(z)}$, and the first vortex equation, it follows that 
\begin{equation}\label{eq:d A u C 2r}|d_Au(z)v|\leq Cr^{1-\eps/2}|z|^{-2+\eps/2}\sqrt{E(w)}|v|,\quad\forall z\in A(4r,\sqrt{rR}),\,v\in\R^2.
\end{equation}
where $C:=2^{9/2}\pi^{-1/2}$. A similar argument shows that 
\begin{equation}\label{eq:d A u C R}|d_Au(z)v|\leq CR^{-1+\eps/2}|z|^{-\eps/2}\sqrt{E(w)}|v|,\quad\forall z\in A(\sqrt{rR},R/4).
\end{equation}
Let now $a\geq4$ and $z,z'\in A(ar,a^{-1}R)$. Assume that $\eps\leq1$. (This is no real restriction.) We define $\ga:[0,1]\to\R^2$ to be the radial path of constant speed, such that $\ga(0)=z$ and $|\ga(1)|=|z'|$. Furthermore, we choose an angular path $\ga':[0,1]\to\R^2$ of constant speed, such that $\ga'(0)=\ga(1)$, $\ga'(1)=z'$, and $\ga'$ has minimal length among such paths. (See Figure \ref{fig:path}.)
\begin{figure}
\centering
\leavevmode\epsfbox{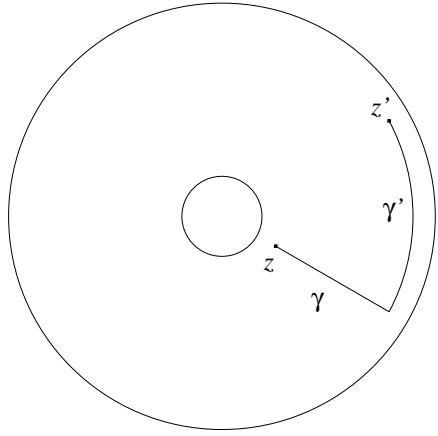}
\caption{The paths $\ga$ and $\ga'$ described in the text.}
\label{fig:path} 
\end{figure}

Consider the ``twisted length'' of $\ga^*(A,u)$, given by $\int_0^1\big|d_Au\,\dot\ga(t)\big|dt$. It follows from (\ref{eq:d A u C 2r},\ref{eq:d A u C R}) and the fact $\eps\leq1$, that this length is bounded above by $4C\sqrt{E(w)}a^{-1+\eps/2}$. Similarly, it follows that the ``twisted length'' of ${\ga'}^*(A,u)$ is bounded above by $C\pi\sqrt{E(w)}a^{-1+\eps/2}$. Therefore, using Remark \ref{rmk:BAR d BAR x 0 BAR x 1}, inequality (\ref{eq:sup z z' bar d}) with $\eps$ replaced by $\eps/2$ follows. 

Assume now that $w$ is not smooth. By Theorem \ref{thm:reg gauge bdd} below the restriction of $w$ to any compact cylinder contained in $A(r,R)$ is gauge equivalent to a smooth vortex. Hence the inequalities (\ref{eq:E w A ar},\ref{eq:sup z z' bar d}) follow from what we just proved, using the $G$-invariance of $K$. 

Similarly, the {\bf case $R=\infty$} can be reduced to the case $R<\infty$. This completes the proof of Proposition \ref{prop:en conc}.
\end{proof}
\begin{proof}[Proof of Proposition \ref{prop:soft}]\label{proof prop:soft}\setcounter{claim}{0} By hypothesis (\ref{prop:soft E}) the function $E$ as in (\ref{eq:eps E eps}) is well-defined. Since it is increasing and bounded below by $\Emin$, the limit 
\begin{equation}
  \label{eq:m 0 lim}m_0:=\lim_{\eps\to0}E(\eps)
\end{equation}
exists and is bounded below by $\Emin$. We fix a compact subset $K\sub M$ as in hypothesis (\ref{prop:soft K}). We choose a constant $E_0>0$ as in Lemma \ref{le:a priori}, depending on $K$. We may assume w.l.o.g.~that $z_0=0$.
\begin{claim}\label{claim:soft wlog} We may assume w.l.o.g.~that 
\begin{equation}
  \label{eq:e C 0}\Vert e_{w_\nu}^{R_\nu}\Vert_{C^0(\bar B_r)}=e_{w_\nu}^{R_\nu}(0).
\end{equation}
\end{claim}
\begin{proof} [Proof of Claim \ref{claim:soft wlog}] Suppose that we have already proved the proposition under this additional assumption, and let $r,z_0=0,R_\nu,w_\nu$ be as in the hypotheses of the proposition. We choose $0<\hhat r\leq r/4$ so small that 
 \begin{equation}
   \label{eq:E hhat r}E(4\hhat r)=\lim_{\nu\to\infty}E^{R_\nu}(w_\nu,B_{4\hhat r})<m_0+E_0.
 \end{equation}
For $\nu\in\N$ we choose $\wt z_\nu\in\bar B_{2\hhat r}$ such that 
\begin{equation}
    \label{eq:e R ze}e^{R_\nu}_{w_\nu}(\wt z_\nu)=\Vert e^{R_\nu}_{w_\nu}\Vert_{C^0(\bar B_{2\hhat r})}.
\end{equation}
\begin{claim}\label{claim:z nu} The sequence $\wt z_\nu$ converges to $0$.
\end{claim}
\begin{proof}[Proof of Claim \ref{claim:z nu}] Recall that $A(r,R)$ denotes the open annulus of radii $r$ and $R$. Let $0<\eps\leq 2\hhat r$. Inequality (\ref{eq:E hhat r}) implies that there exists $\nu(\eps)\in\N$ such that 
\[E^{R_\nu}\big(w_\nu,A(\eps/2,4\hhat r)\big)<E_0,\]
for every $\nu\geq\nu(\eps)$. Hence it follows from Lemma \ref{le:a priori} (Bound on energy density, using $\eps\leq2\hhat r$) that
\begin{equation}\label{eq:e 4 C}e^{R_\nu}_{w_\nu}(z)<\frac{32E_0}{\pi\eps^2},\quad\forall\nu\geq\nu(\eps),\,\forall z\in A(\eps,2\hhat r).
\end{equation}
We define $\de_0:=\min\big\{2\hhat r,\eps\sqrt{m_0/(64E_0)}\big\}$. Increasing $\nu(\eps)$, we may assume that for every $\nu\geq\nu(\eps)$, we have $E^{R_\nu}(w_\nu, B_{\de_0})>m_0/2$, and therefore
\[\Vert e^{R_\nu}_{w_\nu}\Vert_{C^0(\bar B_{\de_0})}>\frac{32E_0}{\pi\eps^2}.\]
Combining this with (\ref{eq:e R ze},\ref{eq:e 4 C}) and the fact $\de\leq2\hhat r$, it follows that $\wt z_\nu\in B_\eps$, for every $\nu\geq\nu(\eps)$. This proves Claim \ref{claim:z nu}. 
\end{proof}
By Claim \ref{claim:z nu} we may pass to some subsequence such that $|\wt z_\nu|<\hhat r$ for every $\nu$. We define 
\[\psi_\nu:B_{\hhat r}\to\R^2,\quad\psi_\nu(z):=z+z_\nu,\quad\wt w_\nu:=(\wt A,\wt u):=\psi_\nu^*w_\nu.\]
Then (\ref{eq:e C 0}) with $w_\nu,r$ replaced by $\wt w_\nu,\hhat r$ is satisfied. By elementary arguments the hypotheses of Proposition \ref{prop:soft} are satisfied with $(w_\nu, r, z_0)$ replaced by $(\wt w_\nu,\hhat r,0)$. Assuming that we have already proved the statement of the proposition for $\wt w_\nu$, a straight-forward argument using Claim \ref{claim:z nu} shows that it also holds for $w_\nu$. This proves Claim \ref{claim:soft wlog}.
\end{proof}
{\bf So we assume w.l.o.g.~that (\ref{eq:e C 0}) holds.}\\

{\bf Construction of $R_0,Z,$ and $w_0$:} Recall that we have chosen $E_0>0$ as in Lemma \ref{le:a priori}. We choose a constants $r_0$ and $E_1$ as in Proposition \ref{prop:en conc}, the latter (called $E_0$ there) corresponding to the compact set $K$ and $\eps:=1$. We fix a constant 
\begin{equation}\label{eq:de min}0<\de<\min\{m_0,E_0/2,E_1/2\}.
\end{equation} 
We pass to some subsequence such that 
\begin{equation}\label{eq:E R nu w nu >}E^{R_\nu}(w_\nu,B_r(z_0))> m_0-\de,\quad\forall\nu\in\N.
\end{equation}
For every $\nu\in\N$, there exists $0<\hhat\eps_\nu<r$, such that
\begin{equation}
  \label{eq:E R =}E^{R_\nu}(w_\nu,B_{\hhat\eps_\nu})=m_0-\de.  
\end{equation}
It follows from the definition of $m_0$ that 
\begin{equation}
  \label{eq:eps nu 0}\hhat\eps_\nu\to0.  
\end{equation}
\begin{claim}\label{claim:inf nu hhat eps nu} We have
\begin{equation}
  \label{eq:hhat eps nu R nu}\inf_\nu \hhat\eps_\nu R_\nu>0.
\end{equation}
\end{claim}
\begin{proof}[Proof of Claim \ref{claim:inf nu hhat eps nu}] 
Equality (\ref{eq:e C 0}) implies that
\begin{equation}
  \label{eq:E eps nu C}E^{R_\nu}(w_\nu,B_{\hhat\eps_\nu})\leq \pi\hhat\eps_\nu^2e_{w_\nu}^{R_\nu}(0).
\end{equation}
The hypotheses $R_\nu\to\infty$, (\ref{prop:soft K}), and (\ref{prop:soft E}) imply that the hypotheses of Proposition \ref{prop:quant en loss} (Quantization of energy loss) are satisfied with $\Om:=B_r$. Thus by assertion (\ref{prop:hard R e}) of that proposition with $Q:=\{0\}$, we have 
\begin{equation}\nn\inf_\nu \frac{R_\nu^2}{e_{w_\nu}^{R_\nu}(0)}>0.
\end{equation}
Combining this with (\ref{eq:E eps nu C},\ref{eq:E R =}) and the fact $\de<m_0$, inequality (\ref{eq:hhat eps nu R nu}) follows. This proves Claim \ref{claim:inf nu hhat eps nu}.
\end{proof}
Passing to some subsequence, we may assume that the limit 
\begin{equation}
  \label{eq:hhat eps R}\hhat R_0:=\lim_{\nu\to\infty}\hhat \eps_\nu R_\nu\in[0,\infty] 
\end{equation}
exists. By Claim \ref{claim:inf nu hhat eps nu} we have $\hhat R_0>0$. We define
\begin{equation}\label{eq:R 0 eps nu}(R_0,\eps_\nu):=\left\{
  \begin{array}{ll}(\infty,\hhat\eps_\nu),&\textrm{if }\hhat R_0=\infty,\\
(1,R_\nu^{-1}),&\textrm{otherwise,}
  \end{array}
\right.  
\end{equation}
\[\wt R_\nu:=\eps_\nu R_\nu,\quad\phi_\nu:B_{\eps_\nu^{-1}r}\to B_r,\,\phi_\nu(z):=\eps_\nu z,\quad\wt w_\nu:=(\wt A_\nu,\wt u_\nu):=\phi_\nu^*w_\nu.\]
By Proposition \ref{prop:cpt mod} with $R_\nu,$ $w_\nu$ replaced by $\wt R_\nu$, $\wt w_\nu$ and $r_\nu:=r/\eps_\nu$ there exist a finite subset $Z\sub\R^2$ and an $R_0$-vortex $w_0=(A_0,u_0)\in\WWW(\R^2\wo Z)$, and passing to some subsequence, there exist gauge transformations $g_\nu\in W^{2,p}_\loc(\R^2\wo Z,G)$, such that the conditions of that proposition are satisfied. 

We check the {\bf conditions of Proposition \ref{prop:soft}} with $z_\nu:=z_0:=0$: {\bf Condition \ref{prop:soft}(\ref{prop:soft eps})} holds by (\ref{eq:eps nu 0},\ref{eq:hhat eps R},\ref{eq:R 0 eps nu}). {\bf Condition \ref{prop:soft}(\ref{prop:soft conv})} follows from \ref{prop:cpt mod}(\ref{prop:cpt mod:<},\ref{prop:cpt mod:=}), and {\bf condition \ref{prop:soft}(\ref{prop:soft lim nu E eps})} follows from \ref{prop:cpt mod}(\ref{prop:cpt mod lim nu E eps}).

We prove {\bf condition \ref{prop:soft}(\ref{prop:soft en})}: We define
\[\psi_\nu:B_{\hhat\eps_\nu^{-1}r}\to B_r,\,\psi_\nu(z):=\hhat\eps_\nu z,\quad\hhat w_\nu:=\psi_\nu^*w_\nu.\]
We choose $0<\eps\leq r$ so small that $\lim_{\nu\to\infty}E^{R_\nu}(w_\nu,B_\eps)<m_0+E_1/2$. Furthermore, we choose an integer $\nu_0$ so large that for $\nu\geq\nu_0$, we have $E^{R_\nu}(w_\nu,B_\eps)<m_0+E_1/2$. We fix $\nu\geq\nu_0$. Using (\ref{eq:E R =},\ref{eq:de min}), it follows that $E\big(\hhat w_\nu,A(\hhat\eps_\nu R_\nu,\eps R_\nu)\big)<E_1$. It follows that the requirements of Proposition \ref{prop:en conc} are satisfied with $r$, $R$, $w_\nu$ replaced by $\max\{r_0,\hhat\eps_\nu R_\nu\}$, $\eps R_\nu$, $\hhat w_\nu$. Therefore, we may apply that result (with ``$\eps$'' equal to $1$), obtaining
\[E^{R_\nu}\Big(w_\nu,A\big(a\max\{R_\nu^{-1}r_0,\hhat\eps_\nu\},a^{-1}\eps\big)\Big)\leq4a^{-1}E_1,\quad\forall a\geq2.\]
Using (\ref{eq:R 0 eps nu}) and the fact $z_\nu=z_0=0$, the inequality (\ref{eq:en cons}) follows. This proves \ref{prop:soft}(\ref{prop:soft en}).

To see that {\bf condition \ref{prop:soft}(\ref{prop:soft Z})} holds, assume first that $R_0=1$. Then $Z=\emptyset$ by statement (\ref{prop:cpt mod:<}) of Proposition \ref{prop:cpt mod}. Condition \ref{prop:cpt mod}(\ref{prop:cpt mod:<}) and Lemma \ref{le:conv e} imply that $E(w_0,B_{2\hhat R_0})=\lim_{\nu\to\infty}E(\wt w_\nu,B_{2\hhat R_0})$. It follows from convergence $\hhat \eps_\nu R_\nu\to \hhat R_0<\infty$ and (\ref{eq:E R =},\ref{eq:de min}) that this limit is positive. This proves condition \ref{prop:soft}(\ref{prop:soft Z}) in the case $R_0=1$. 

{\bf Assume now that $R_0=\infty$ and $E^\infty(w_0)=0$.} Then condition \ref{prop:soft}(\ref{prop:soft Z}) is a consequence of the following two claims.
\begin{claim}\label{claim:Z B 1} The set $Z$ is not contained in the open ball $B_1$.
\end{claim}
\begin{proof}[Proof of Claim \ref{claim:Z B 1}] By \ref{prop:soft}(\ref{prop:soft en}) there exists $R>0$ so that
\begin{equation}
  \label{eq:limsup de}\limsup_{\nu\to\infty}E^{R_\nu}\big(w_\nu,A(R\eps_\nu,R^{-1})\big)<\de.
\end{equation}
(Here we used that $z_0=z_\nu=0$.) Since $R_0=\infty$, we have $\hhat \eps_\nu=\eps_\nu$. Hence it follows from (\ref{eq:E R =}) and the definition (\ref{eq:m 0 lim}) of $m_0$, that
\[\lim_{\nu\to\infty}E^{R_\nu}\big(w_\nu,A(\eps_\nu,R^{-1})\big)\geq\de.\]
Combining this with (\ref{eq:limsup de}), it follows that 
\begin{equation}\label{eq:liminf E A}\liminf_{\nu\to\infty}E^{R_\nu}(w_\nu,A(\eps_\nu,\eps_\nu R))>0.\end{equation}
Suppose by contradiction that $Z\sub B_1$. Then by \ref{prop:cpt mod}(\ref{prop:cpt mod:=}), the connection $g_\nu^*\wt A_\nu$ converges to $A_0$ in $C^0$ on $\bar A(1,R):=\BAR B_R\wo B_1$, and the map $g_\nu^{-1}\wt u_\nu$ converges to $u_0$ in $C^1$ on $\bar A(1,R)$. Hence Lemma \ref{le:conv e} implies that
\[E^\infty\big(w_0,A(1,R)\big)=\lim_{\nu\to\infty}E^{\wt R_\nu}\big(\wt w_\nu,A(1,R)\big).\]
Combining this with (\ref{eq:liminf E A}), we arrive at a contradiction to our assumption $E^\infty(w_0)=0$. This proves Claim \ref{claim:Z B 1}.
\end{proof}
\begin{claim}\label{claim:0 Z} The set $Z$ contains $0$.
\end{claim}
\begin{pf}[Proof of Claim \ref{claim:0 Z}] By Claim \ref{claim:Z B 1} the set $Z\wo B_1$ is nonempty. We choose a point $z\in Z\wo B_1$ and a number $\eps_0>0$ so small that $B_{\eps_0}(z)\cap Z=\{z\}$. We fix $0<\eps<\eps_0$. Since $\eps_\nu\to0$ (as $\nu\to\infty$), (\ref{eq:e C 0}) implies that $e^{\wt R_\nu}_{\wt w_\nu}(0)=\Vert e^{\wt R_\nu}_{\wt w_\nu}\Vert_{C^0(\bar B_\eps(z))}$, for $\nu$ large enough. Combining this with condition \ref{prop:soft}(\ref{prop:soft lim nu E eps}), it follows that $\liminf_{\nu\to\infty}e^{\wt R_\nu}_{\wt w_\nu}(0)\geq\Emin/(\pi\eps^2)$. Since $\eps\in(0,\eps_0)$ is arbitrary, it follows that 
\begin{equation}\label{eq:e wt R nu wt w nu 0}e^{\wt R_\nu}_{\wt w_\nu}(0)\to\infty,\quad\textrm{as }\nu\to\infty.
\end{equation}
If $0$ did not belong to $Z$, then by \ref{prop:cpt mod}(\ref{prop:cpt mod:=}) and Lemma \ref{le:conv e} we would have $e^{\wt R_\nu}_{\wt w_\nu}(0)\to e^\infty_{w_0}(0)$, a contradiction to (\ref{eq:e wt R nu wt w nu 0}). This proves Claim \ref{claim:0 Z}, and completes the proof of \ref{prop:soft}(\ref{prop:soft Z}) and therefore of Proposition \ref{prop:soft}.
\end{pf}
\end{proof}
\begin{rmk}\label{rmk:soft proof} \rm Assume that $R_0,Z,w_0$ are constructed as in the proof of condition (\ref{prop:soft Z}) of Proposition \ref{prop:soft}, and that $R_0=\infty$ and $E^\infty(w_0)=0$. Then $Z\sub\bar B_1$ (and hence $Z\cap S^1\neq\emptyset$ by Claim \ref{claim:Z B 1}). This follows from the inequalities
\[\lim_{\nu\to\infty}E^{\wt R_\nu}(\wt w_\nu,A(1,R))\leq\de<\Emin,\quad\forall R>1.\]
Here the first inequality is a consequence of condition (\ref{eq:E R =}). $\Box$
\end{rmk}
\section{Proof of Theorem \ref{thm:bubb} (Bubbling)}\label{sec:proof:thm:bubb}
Based on the results of the previous sections, we are now ready to prove the main result of this article. The proof is an adaption of the proof of \cite[Theorem 5.3.1]{MS} to the present setting. The strategy is the following: Consider first the case $k=0$, i.e., the only marked point is $z_0^\nu=\infty$. We rescale the sequence $W_\nu$ so rapidly that all the energy is concentrated at the origin in $\R^2$. Then we ``zoom back in'' in a soft way, to capture the bubbles (spheres in $\BAR M$ and vortices on $\R^2$) in an inductive way. (See Claim \ref{claim:tree} below.) 

Next we show that at each stage of this construction, the total energy of the components of the tree plus the energy loss at the unresolved bubbling points equals the limit of the energies $E(W^\nu)$. (See Claim \ref{claim:f E}.) Furthermore, we prove that the number of vertices of the tree is uniformly bounded above. (See inequality (\ref{eq:N E}).) This implies that the inductive construction terminates at some point. 

We also show that the components of the tree have the required properties. (See Claim \ref{claim:MMM bar u i}.) Finally, we prove that the data fits together to a stable map, which is the limit of a subsequence of $W^\nu$. (See Claim \ref{claim:st conv}.) 

For $k\geq1$ we then prove the statement of the theorem inductively, using the statement for $k=0$. At each induction step we need to handle one additional marked point in the sequence of vortices and marked points. In the limit there are three possibilities for the location of this point: (I) It may lie on a vertex where it does not coincide with any special point. (II) It may coincide with the marked point $z_i$ (lying on the $\al_i$-th vertex), for some $i$. (III) It may lie between two already constructed bubbles. 

In case (I) we can just include the new marked point into the bubble tree. In case (II) we introduce a ``ghost bubble'', which carries the two marked points and is connected to $\al_i$. In case (III) we introduce a ``ghost bubble'' between the two bubbles, which carries the new marked point. 
\begin{proof}[Proof of Theorem \ref{thm:bubb}]\label{thm:bubb proof}\setcounter{claim}{0}We consider first the {\bf case $k=0$.} Let $W_\nu$ be a sequence of vortices as in the hypothesis. For each $\nu\in\N$ we choose a representative $w_\nu:=(P_\nu,A_\nu,u_\nu)$ of $W_\nu$, such that $P_\nu=\R^2\x G$. Passing to some subsequence we may assume that $E(w_\nu)$ converges to some constant $E$. The hypothesis $E(W_\nu)>0$ (for every $\nu$) implies that $E\geq\Emin$. We choose a sequence $R_\nu\geq1$ such that 
\begin{equation}
  \label{eq:E B E}E(W_\nu,B_{R_\nu})\to E.  
\end{equation}
We define 
\begin{eqnarray*}&R_0^\nu:=\nu R_\nu,\quad \phi_\nu:\R^2\to\R^2,\quad w^\nu_0:=\phi_\nu^*w_\nu,&\\
&j_1:=0,\quad z_1:=0,\quad Z_0:=\{0\},\quad z_0^\nu:=0.&
\end{eqnarray*}
The next claim provides an inductive construction of the bubble tree. (Some explanations are given below. See also Figure \ref{fig:claim}.)
\begin{figure}
  \centering
\leavevmode\epsfbox{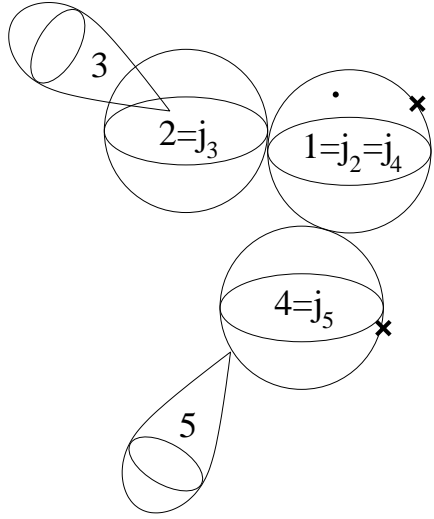}
\caption{This is a ``partial stable map'' as in Claim \ref{claim:tree}. It is a possible step in the construction of the stable map of Figure \ref{fig:stable map}. The crosses are bubbling points that have not yet been resolved. When adding marked points the components 4 and 5 will be separated by a ghost bubble which carries one marked point.}
  \label{fig:claim} 
\end{figure}
\begin{claim}\label{claim:tree} For every number $\ell\in\N$, passing to some subsequence, there exist an integer $N:=N(\ell)\in\N$ and tuples
\[(R_i,Z_i,w_i)_{i\in\{1,\ldots,N\}},\quad(R_i^\nu,z_i^\nu)_{i\in\{1,\ldots,N\},\,\nu\in\N},\quad(j_i,z_i)_{i\in\{2,\ldots,N\}},\]
where $R_i\in\{1,\infty\}$, $Z_i\sub \R^2$ is a finite subset, $w_i=(A_i,u_i)\in\WWW(\R^2\wo Z_i)$ is an $R_i$-vortex, $R_i^\nu>0$, $z_i^\nu\in\R^2$, $j_i\in \{1,\ldots,i-1\}$, and $z_i\in\R^2$, such that the following conditions hold.
\begin{enui}
\item\label{claim:tree Z dist} For every $i=2,\ldots,N$ we have $z_i\in Z_{j_i}$. Moreover, if $i,i'\in \{2,\ldots,N\}$ are such that $i\neq i'$ and $j_i=j_{i'}$ then $z_i\neq z_{i'}$.
\item\label{claim:tree stab} Let $i=1,\ldots,N$. If $R_i=1$ then $Z_i=\emptyset$ and $E(w_i)>0$. If $R_i=\infty$ and $E^\infty(w_i)=0$ then $|Z_i|\geq2$.
\item\label{claim:tree R} Fix $i=1,\ldots,N$. If $R_i=1$ then $R_i^\nu=1$ for every $\nu$, and if $R_i=\infty$ then $R_i^\nu\to\infty$. Furthermore, 
\begin{equation}
\label{eq:R i nu}\frac{R_i^\nu}{R_{j_i}^\nu}\to 0,\qquad \frac{z_i^\nu-z_{j_i}^\nu}{R_{j_i}^\nu}\to z_i.
\end{equation}

In the following we set $\phi_i^\nu(z):=R_i^\nu z +z_i^\nu$, for $i=0,\ldots,N$ and $\nu\in\N$. 
\item\label{claim:tree conv} For every $i=1,\ldots,N$ there exist gauge transformations $g_i^\nu\in W^{2,p}_\loc(\R^2\wo Z_i,G)$ such that the following holds. If $R_i=1$ then $(g_i^\nu)^*(\phi_i^\nu)^*w_\nu$ converges to $w_i$ in $C^\infty$ on every compact subset of $\R^2$. Furthermore, if $R_i=\infty$ then on every compact subset of $\R^2\wo Z_i$ the sequence $(g_i^\nu)^*(\phi_i^\nu)^*A_\nu$ converges to $A_i$ in $C^0$, and the sequence $(g_i^\nu)^*(\phi_i^\nu)^*u_\nu$ converges to $u_i$ in $C^1$. 
\item\label{claim:tree lim nu E eps} Let $i=1,\ldots,N$, $z\in Z_i$ and $\eps_0>0$ be such that $B_{\eps_0}(z)\cap Z_i=\{z\}$. Then for every $0<\eps<\eps_0$ the limit 
\[E_z(\eps):=\lim_{\nu\to\infty} E^{R_i^\nu}\big((\phi_i^\nu)^*w_\nu,B_\eps(z)\big)\] 
exists, and $\Emin\leq E_z(\eps)<\infty$. Furthermore, the function $(0,\eps_0)\ni\eps\mapsto E_z(\eps)\in[\Emin,\infty)$ is continuous.
\item\label{claim:tree B B}For every $i=1,\ldots,N$, we have
\[\lim_{R\to\infty}\limsup_{\nu\to\infty} E\big(w_\nu,B_{R_{j_i}^\nu/R}(z_{j_i}^\nu+R_{j_i}^\nu z_i)\wo B_{RR_i^\nu}(z_i^\nu)\big)=0.\]
\item\label{claim:tree compl} If $\ell>N$ then for every $j=1,\ldots,N$ we have 
\begin{equation}\label{eq:Z j}Z_j=\big\{z_i\,|\,j<i\leq N,j_i=j\big\}.
\end{equation}
\end{enui}
\end{claim}
To understand this claim, note that the collection $(j_i)_{i\in\{2,\ldots,N\}}$ describes a tree with vertices the numbers $1,\ldots,N$ and unordered edges $\big\{(i,j_i),(j_i,i)\big\}$. Attached to the vertices of this tree are vortices and $\infty$-vortices. (The latter will give rise to holomorphic spheres in $\BAR M$.) Each pair $(R_i^\nu,z_i^\nu)$ defines a rescaling $\phi_i^\nu$, which is used to obtain the $i$-th limit vortex or $\infty$-vortex. (See condition (\ref{claim:tree conv}).)

The point $z_i$ is the nodal point on the $j_i$-th vertex, at which the $i$-th vertex is attached. The corresponding nodal point on the $i$-th vertex is $\infty$. The set $Z_i$ consists of the nodal points except $\infty$ (if $i\geq2$) on the $i$-th vertex together with the bubbling points that have not yet been resolved. 

Condition (\ref{claim:tree Z dist}) implies that the nodal points at a given vertex are distinct. Condition (\ref{claim:tree stab}) guarantees that once all bubbling points have been resolved, the $i$-th component will be stable. (Note that in the case $i\geq2$ there is another nodal point at $\infty$, and for $i=1$ there will be a marked point at $\infty$, which comes from sequence $z_0^\nu$.) 

Condition (\ref{claim:tree R}) implies that the rescalings $\phi_i^\nu$ ``zoom out'' less than the rescalings $\phi_{j_i}^\nu$. A consequence of condition (\ref{claim:tree lim nu E eps}) is that at every nodal or unresolved bubbling point at least the energy $\Emin$ concentrates in the limit. Condition (\ref{claim:tree B B}) means that no energy is lost between each pair of adjacent bubbles. Finally, condition (\ref{claim:tree compl}) means that in the case $\ell>N$ all bubbling points have been resolved. 
\begin{proof}[Proof of Claim \ref{claim:tree}]We show that the statement holds for $\ell:=1$. We check the conditions of Proposition \ref{prop:soft} (Soft rescaling) with $z_0:=0$, $r:=1$ and $R_\nu$, $w_\nu$ replaced by $R_0^\nu$, $w_0^\nu$. Condition \ref{prop:soft}(\ref{prop:soft K}) follows from Proposition \ref{prop:bounded} below, using the hypothesis that $M$ is equivariantly convex at $\infty$. Condition \ref{prop:soft}(\ref{prop:soft E}) follows from the facts
\[\lim_{\nu\to\infty} E^{R_0^\nu}(w_0^\nu,B_\eps)=E,\,\forall\eps>0,\quad E\geq\Emin.\]
The first condition is a consequence of the facts $R_0^\nu=\nu R_\nu$, $E(w_\nu)\to E$, and (\ref{eq:E B E}).

Thus by Proposition \ref{prop:soft}, there exist $R_0\in \{1,\infty\}$, a finite subset $Z\sub\R^2$, and an $R_0$-vortex $w_0\in\WWW^p(\R^2\wo Z_1)$, and passing to some subsequence, there exist sequences $\eps_\nu>0$, $z_\nu$, and $g_\nu$, such that the conclusions of Proposition \ref{prop:soft} with $R_\nu,w_\nu$ replaced by $R_0^\nu,w_0^\nu$ hold. We define $N:=N(1):=1$, $R_1:=R_0$, $Z_1:=Z$, $w_1:=w_0$, $R_1^\nu:=\eps_\nu R_0^\nu$, and $z_1^\nu:=R_0^\nu z_\nu$. 

We check {\bf conditions (\ref{claim:tree Z dist})-(\ref{claim:tree compl})} of Claim \ref{claim:tree} with $\ell=1$: Conditions (\ref{claim:tree Z dist},\ref{claim:tree compl}) are void. Furthermore, conditions (\ref{claim:tree stab})-(\ref{claim:tree B B}) follow from \ref{prop:soft}(\ref{prop:soft Z})-(\ref{prop:soft en}). This proves the statement of the Claim for $\ell=1$. 

Let $\ell\in\N$ and assume, by induction, that we have already proved the statement of Claim \ref{claim:tree} for $\ell$. We show that it holds for $\ell+1$. By assumption there exists a number $N:=N(\ell)$ and there exist collections $(R_i,Z_i,w_i)_{i\in\{1,\ldots,N\}}$, $(R_i^\nu,z_i^\nu)_{i\in\{1,\ldots,N\},\,\nu\in\N}$, $(j_i,z_i)_{i\in\{2,\ldots,N\}}$, such that conditions (\ref{claim:tree Z dist})-(\ref{claim:tree compl}) hold. If for every $j=1,\ldots,N$ we have $Z_j=\big\{z_i\,|\,j< i\leq N,\,j_i=j\big\}$ then conditions (\ref{claim:tree Z dist})-(\ref{claim:tree compl}) hold with $N(\ell+1):=N$, and we are done. Hence assume that there exists a $j_0\in\{1,\ldots,N\}$ such that 
\begin{equation}
  \label{eq:Z j 0}Z_{j_0}\neq\big\{z_i\,|\,j_0<i\leq N,\,j_i=j_0\big\}.  
\end{equation}
We set $N(\ell+1):=N+1$ and choose an element 
\begin{equation}\label{eq:z N 1}z_{N+1}\in Z_{j_0}\wo\big\{z_i\,|\,j<i\leq N,\,j_i=j_0\big\}.
\end{equation}
We fix a number $r>0$ so small that $B_r(z_{N+1})\cap Z_{j_0}=\{z_{N+1}\}$. We apply Proposition \ref{prop:soft} with $z_0:=z_{N+1}$ and $R_\nu$, $w_\nu$ replaced by $R_{j_0}^\nu$, $(\phi_{j_0}^\nu)^*w_\nu.$ Condition \ref{prop:soft}(\ref{prop:soft K}) holds by hypothesis. Furthermore, by condition (\ref{claim:tree lim nu E eps}) for $\ell$, condition \ref{prop:soft}(\ref{prop:soft E}) is satisfied. Hence passing to some subsequence, there exist $R_0\in\{1,\infty\}$, a finite subset $Z\sub \R^2$, an $R_0$-vortex $w_0\in\WWW^p(\R^2\wo Z)$, and sequences $\eps_\nu>0$, $z_\nu$, such that the conclusion of Proposition \ref{prop:soft} holds. We define  $R_{N+1}:=R_0$, $Z_{N+1}:=Z$, $w_{N+1}:=w_0$, $R_{N+1}^\nu:=\eps_\nu R_{j_0}^\nu$, $z_{N+1}^\nu:=R_{j_0}^\nu z_\nu+z_{j_0}^\nu$ and $j_{N+1}:=j_0$. 

We check {\bf conditions (\ref{claim:tree Z dist})-(\ref{claim:tree compl})} of Claim \ref{claim:tree} with $\ell$ replaced by $\ell+1$, i.e., $N$ replaced by $N+1$. Condition (\ref{claim:tree Z dist}) follows from the induction hypothesis and (\ref{eq:z N 1}). Conditions (\ref{claim:tree stab})-(\ref{claim:tree B B}) follow from \ref{prop:soft}(\ref{prop:soft Z})-(\ref{prop:soft en}).

We show that (\ref{claim:tree compl}) holds with $N$ replaced by $N+1$: By the induction hypothesis, it holds for $N$. Hence (\ref{eq:Z j 0}) implies that $N\geq \ell$, i.e., $N+1\geq\ell+1$. So there is nothing to check. This completes the induction and the proof of Claim \ref{claim:tree}. \end{proof}

Let $\ell\in\N$ be an integer and $N:=N(\ell)$, $(R_i,Z_i,w_i)$, $(R_i^\nu,z_i^\nu)$, $(j_i,z_i)$ be as in Claim \ref{claim:tree}. Recall that $Z_0=\{0\}$ and $z_0^\nu:=0$. We fix $i=0,\ldots,N$. We define $\phi_i^\nu(z):=R_i^\nu z+z_i^\nu$, for every measurable subset $X\sub\R^2$ we denote
\[E_i(X):=E^{R_i}(w_i,X),\quad E_i:=E_i(\R^2\wo Z_i),\quad E_i^\nu(X):=E^{R_i^\nu}((\phi_i^\nu)^*w_\nu,X).\]
Furthermore, for $z\in Z_i$ we define
\begin{equation}
  \label{eq:m i z}m_i(z):=\lim_{\eps\to 0}\lim_{\nu\to\infty}E_i^\nu(B_\eps(z)).
\end{equation}
For $i=0$ it follows from (\ref{eq:E B E}) and $R_0^\nu=\nu R_\nu$ that the limit $m_0(0)$ exists and equals $E$. For $i=1,\ldots,N$ it follows from condition (\ref{claim:tree lim nu E eps}) that the limit (\ref{eq:m i z}) exists and that $m_i(z)\geq \Emin$. 
For $j,k=0,\ldots, N$ we define 
\[Z_{j,k}:=Z_j\wo\{z_i\,|\,j<i\leq k,\,j_i=j\}\]
(This is the set of points on the $j$-th sphere that have not been resolved after the construction of the $k$-th bubble.) We define the function $f:\{1,\ldots,N\}\to [0,\infty)$ by 
\begin{equation}
  \label{eq:f i}f(i):=E_i+\sum_{z\in Z_{i,N}}m_i(z).  
\end{equation}
\begin{claim}\label{claim:f E} 
\begin{equation}\nn\sum_{i=1}^Nf(i)=E.
\end{equation}
\end{claim}
\begin{proof}[Proof of Claim \ref{claim:f E}]We show by induction that 
\begin{equation}
\label{eq:k f E}\sum_{i=1}^k\Big(E_i+\sum_{z\in Z_{i,k}}m_i(z)\Big)=E,
\end{equation}
for every $k=1,\ldots,N$. Claim \ref{claim:f E} is a consequence of this with $k=N$. For the proof of equality (\ref{eq:k f E}) we need the following.
\begin{claim}\label{claim:E Z} For every $i=1,\ldots,N$ we have
\begin{equation}
    \label{eq:m j i z i}m_{j_i}(z_i)=E_i+\sum_{z\in Z_i}m_i(z).
\end{equation}
\end{claim}
\begin{proof}[Proof of Claim \ref{claim:E Z}] Let $i=1,\ldots,N$. We choose a number $\eps>0$ so small that 
\[\bar B_{\eps}(z_i)\cap Z_{j_i}=\{z_i\},\qquad Z_i\sub B_{\eps^{-1}-\eps},\]
and if $z\neq z'$ are points in $Z_i$ then $|z-z'|>2\eps$. By condition (\ref{claim:tree lim nu E eps}) of Claim \ref{claim:tree}, for each $z\in Z_i$ the limit $\lim_{\nu\to\infty}E_i^\nu(B_\eps(z))$ exists. Lemma \ref{le:conv e} implies that
\[\lim_{\nu\to\infty}E_i^\nu(B_{\eps^{-1}})=E_i\Big(B_{\eps^{-1}}\wo \bigcup_{z\in Z_i}B_\eps(z)\Big)+\sum_{z\in Z_i}\lim_{\nu\to\infty}E_i^\nu(B_\eps(z)).\]
Combining this with condition (\ref{claim:tree B B}) of Claim \ref{claim:tree}, equality (\ref{eq:m j i z i}) follows from a straight-forward argument. This proves Claim \ref{claim:E Z}. 
\end{proof}
Since $Z_{1,1}=Z_1$, equality (\ref{eq:k f E}) for $k=1$ follows from Claim \ref{claim:E Z} and the fact $m_0(0)=E$. Let now $k=1,\ldots,N-1$ and assume that we have proved (\ref{eq:k f E}) for $k$. An elementary argument using Claim \ref{claim:E Z} with $i:=k+1$ shows (\ref{eq:k f E}) with $k$ replaced by $k+1$. By induction, Claim \ref{claim:f E} follows. 
\end{proof}
Consider the tree relation $E$ on $T:=\{1,\ldots,N\}$ defined by $i E i'$ iff $i=j_{i'}$ or $i'=j_i$. Lemma \ref{le:weight} below with this pair $(T,E)$, $f$ as in (\ref{eq:f i}), $k:=1$, $\al_1:=1\in T$, and $E_0:=\Emin$, implies that 
\begin{equation} \label{eq:N E} N\leq \frac{2E}\Emin+1.
\end{equation}
(Hypothesis (\ref{eq:f al E}) follows from conditions (\ref{claim:tree stab},\ref{claim:tree lim nu E eps}) of Claim \ref{claim:tree}.) Assume now that we have chosen $\ell>2E/\Emin+1$. By (\ref{eq:N E}) we have $\ell>N$, and therefore by condition (\ref{claim:tree compl}) of Claim \ref{claim:tree}, equality (\ref{eq:Z j}) holds, for every $j=1,\ldots,N$. We define
\[T:=\{1,\ldots,N\},\quad V:=\{i\in T\,|\,R_i=1\},\quad \BAR T:=T\wo V,\] 
and the tree relation $E$ on $T$ by 
\[i Ei'\iff i=j_{i'}\textrm{ or }i'=j_i.\] 
Furthermore, for $i,i'\in T$ such that $iEi'$ we define the nodal points 
\[z_{ii'}:=\left\{
\begin{array}{ll}\infty,&\textrm{if }i'=j_i,\\
z_{i'},&\textrm{if }i=j_{i'}.
  \end{array}
\right.\]
Moreover, we define the marked point 
\[(\al_0,z_0):=(1,\infty)\in T\x S^2.\]
\begin{claim}\label{claim:MMM bar u i} Let $i\in T$. If $i\in V$ then $E(w_i)<\infty$ and $u_i(\R^2\x G)$ has compact closure. Furthermore, if $i\in\BAR T$ then the map $G u_i:\R^2\wo Z_i\to \BAR M=\mu^{-1}(0)/G$ extends to a smooth $\bar J$-holomorphic map 
\[\bar u_i:S^2\iso\R^2\cup\{\infty\}\to \BAR M.\]
\end{claim}
\begin{proof} We choose gauge transformations $g_i^\nu\in W^{2,p}_\loc(\R^2\wo Z_i,G)$ as in condition (\ref{claim:tree conv}) of Claim \ref{claim:tree}, and define $w_i^\nu:=(g_i^\nu)^*(\phi_i^\nu)^*w_\nu$. 

{\bf Assume that $i\in V$.} It follows from Fatou's lemma that $E(w_i)\leq\liminf_{\nu\to\infty}E(w_i^\nu)=E<\infty$. Furthermore, since by hypothesis $M$ is equivariantly convex at $\infty$, by Proposition \ref{prop:bounded} below there exists a $G$-invariant compact subset $K_0\sub M$ such that $u_i^\nu(\R^2)\sub K_0$, for every $\nu\in\N$. Since $u_i^\nu$ converges to $u_i$ pointwise, it follows that $u_i(\R^2)\sub K_0$. Hence $w_i$ has the required properties.

{\bf Assume now that $i\in\BAR T$.} By Proposition \ref{prop:bar del J} below the map
\[G u_i:\C\wo Z_i\to \BAR M=\mu^{-1}(0)/G\]
is $\bar J$-holomorphic, and $e_{G u_i}=e^\infty_{w_i}$. It follows from Fatou's lemma that $E^\infty(w_i,\R^2\wo Z_i)\leq\liminf_{\nu\to\infty}E^{R_i^\nu}(w_i^\nu)=E<\infty$. Therefore, by removal of singularities, it follows that $G u_i$ extends to a smooth $\bar J$-holomorphic map $\bar u_i:S^2\to\BAR M$. (See e.g.~\cite[Theorem 4.1.2]{MS}.) This proves Claim \ref{claim:MMM bar u i}. \end{proof}
\begin{claim}\label{claim:st conv} The tuple
\[(\W,\z):=\big(V,\BAR T,E,([w_i])_{i\in V},(\bar u_i)_{i\in \BAR T},(z_{ii'})_{iEi'},(\al_0:=1,z_0:=\infty)\big)\]
is a stable map in the sense of Definition \ref{defi:st}, and the sequence $([w_\nu],z_0^\nu:=\infty)$ converges to $(\W,\z)$ in the sense of Definition \ref{defi:conv}. (Here $[w_i]$ denotes the gauge equivalence class of $w_i$.)
\end{claim}
\begin{proof}[Proof of Claim \ref{claim:st conv}] We check the conditions of Definition \ref{defi:st}. {\bf Condition (\ref{defi:st dist})} follows from condition (\ref{claim:tree Z dist}) of Claim \ref{claim:tree} and the fact $Z_i=\emptyset$, for $i\in V$. (This follows from condition (\ref{claim:tree stab}) of Claim \ref{claim:tree}.)

{\bf Condition (\ref{defi:st conn})} follows from an elementary argument using Claim \ref{claim:tree}(\ref{claim:tree R},\ref{claim:tree conv},\ref{claim:tree B B}) and Proposition \ref{prop:en conc}. {\bf Condition (\ref{defi:st st})} follows from Claim \ref{claim:tree}(\ref{claim:tree stab}). Hence all conditions of Definition \ref{defi:st} are satisfied.

We check the {\bf conditions of Definition \ref{defi:conv}}. {\bf Condition (\ref{eq:E V bar T})} follows from Claim \ref{claim:f E}, using condition (\ref{claim:tree compl}) of Claim \ref{claim:tree}. {\bf Condition \ref{defi:conv}(\ref{defi:conv phi z})} follows from a straight-forward argument, using Claim \ref{claim:tree}(\ref{claim:tree R}).

{\bf Condition \ref{defi:conv}(\ref{defi:conv al be})} follows from Claim \ref{claim:tree}(\ref{claim:tree R}) by an elementary argument. {\bf Condition \ref{defi:conv}(\ref{defi:conv w})} follows from Claim \ref{claim:tree}(\ref{claim:tree conv}). Finally, {\bf condition \ref{defi:conv}(\ref{defi:conv z})} is void, since $k=0$. This proves Claim \ref{claim:st conv}.
\end{proof}
Thus we have proved Theorem \ref{thm:bubb} in the case $k=0$.\\

{\bf We prove now by induction that the Theorem holds for every $k\geq1$:} Let $k\geq\N_0$ be an integer, $\big(W_\nu,z_1^\nu,\ldots,z_{k-1}^\nu\big)$ as in the hypotheses of Theorem \ref{thm:bubb}, and assume that there exists a stable map $(\W,\z)$ (as in (\ref{eq:W z})) and a collection $(\phi_\al^\nu)$ of M\"obius transformations such that $\big(W_\nu,z_0^\infty:=\infty,z_1^\nu,\ldots,z_{k-1}^\nu\big)$ converges to $(\W,\z)$ via $(\phi_\al^\nu)$, 
\begin{eqnarray}\nn&\phi_\al^\nu(\infty)=\infty,&\\ 
\label{eq:lim de limsup nu}&\lim_{R\to\infty}\limsup_{\nu\to\infty}E\big(W_\nu,\R^2\wo\phi_{\al_0}^\nu(B_R)\big)=0,&
\end{eqnarray}
and for every edge $\al E\be$ such that $\be$ lies in the chain of vertices from $\al$ to $\al_0$, we have
\begin{equation}\label{eq:lim limsup be al}\lim_{R\to\infty}\limsup_{\nu\to\infty}E\big(W_\nu,\phi_\be^\nu(B_{R^{-1}}(z_{\be\al}))\wo \phi_\al^\nu(B_R)\big)=0. 
\end{equation}
(For $k=0$ we proved this above. In this case condition (\ref{eq:lim de limsup nu}) follows from (\ref{claim:tree B B}) of Claim \ref{claim:tree} with $i=1$, and the facts $E(W_\nu,B_{R_\nu})\to E$ and $R_0^\nu=\nu R_\nu$. Furthermore, condition (\ref{eq:lim limsup be al}) follows from condition (\ref{claim:tree B B}) of Claim \ref{claim:tree}.) By hypothesis (\ref{eq:limsup z i nu z j nu}), passing to some subsequence, we may assume that for every $i=1,\ldots,k-1$, the limit 
\begin{equation}\label{eq:z k i}z_{ki}:=\lim_{\nu\to\infty}(z_k^\nu-z_i^\nu)\in\R^2\cup\{\infty\}\end{equation}
exists, and $z_{ki}\neq0$. We set 
\[z_{k0}:=\infty.\] 
Passing to a further subsequence, we may assume that the limit 
\[z_{\al k}:=\lim_{\nu\to\infty}(\phi_\al^\nu)^{-1}(z_k^\nu)\in S^2\] 
exists, for every $\al\in T$. There are three cases.\\

\noi{\bf Case (I)} There exists a vertex $\al\in T$, such that $z_{\al k}$ is not a special point of $(\W,\z)$ at $\al$.\\

\noi{\bf Case (II)} There exists an index $i\in\{0,\ldots,k-1\}$ such that $z_{\al_ik}=z_i$.\\

\noi{\bf Case (III)} There exists an edge $\al E\be$ such that $z_{\al k}=z_{\al\be}$ and $z_{\be k}=z_{\be\al}$.\\

These three cases exclude each other. For the combination of the cases (II) and (III) this follows from condition (\ref{defi:st dist}) (distinctness of the special points) of Definition \ref{defi:st}. 

\begin{claim}\label{claim:I III} One of the three cases always applies. 
\end{claim}

\begin{proof}[Proof of Claim \ref{claim:I III}] This follows from an elementary argument, using that $T$ is finite and does not contain cycles. 
\end{proof}

{\bf Assume that Case (I) holds.} We fix a vertex $\al\in T$ such that $z_{\al k}$ is not a special point. (This vertex is unique, but we do not need this.) We define $\al_k:=\al$ and introduce a new marked point 
\[z_k^\new:=z_{\al_kk}\]
on the $\al_k$-sphere. Then $(\W,\z)$ augmented by $z_k^\new$ is again a stable map and the sequence $(W_\nu,z_0^\nu,\ldots,z_k^\nu)$ converges to this new stable map via $(\phi_\al^\nu)_{\al\in T}$.

{\bf Assume that Case (II) holds.} We fix an index $0\leq i\leq k-1$ such that $z_{\al_ik}=z_i$. (It is unique.) The hypothesis (\ref{eq:limsup z i nu z j nu}) implies that $\al_i\in\BAR T$. We extend the tree $T$ by introducing an additional vertex $\ga$ which is adjacent to $\al_i$. If $z_{ki}=\infty$ (defined as in (\ref{eq:z k i})) then the new vertex corresponds to a bubble in $\BAR M$, otherwise it corresponds to a vortex. We move the $i$-th marked point from the vertex $\al_i$ to the vertex $\ga$ and introduce an additional marked point on $\ga$. More precisely, we define
\[\BAR T^\new:=\left\{
  \begin{array}{ll}\BAR T\disj\{\ga\},&\textrm{if }z_{ki}=\infty,\\
\BAR T,&\textrm{otherwise},
  \end{array}
\right.\]
\begin{eqnarray}\nn &T^\new:=T\disj\{\ga\},\quad V^\new:=T^\new\wo\BAR T^\new,&\\
\nn &\al_i^\new:= \al_k:=\ga,\quad z_{\ga\al_i}^\new:=\infty,\quad z_{\al_i\ga}^\new:=z_i&\\
\nn&z_i^\new:=0,\quad z_k^\new:=\left\{
\begin{array}{ll}z_{ki},&\textrm{if }z_{ki}\neq\infty,\\
1,&\textrm{otherwise.}
\end{array}\right.&
\end{eqnarray}
Assume first that $\ga\in V^\new$. We choose a point $x_0$ in the orbit $\bar u_{\al_i}(z_i)\sub\mu^{-1}(0)$, and define $A_\ga:=0\in\Om^1(\R^2,\g)$ and $u_\ga:\R^2\to M$ to be the map which is constantly equal to $x_0$. We identify $A_\ga$ with a connection on $\R^2\x G$ and $u_\ga$ with a $G$-equivariant map $\R^2\x G\to M$, and set 
\begin{equation}\label{eq:W ga}W_\ga:=\big[\R^2\x G,A_\ga,u_\ga\big].
\end{equation}
If $\ga\in \BAR T^\new$ then we define $\bar u_{\ga}:S^2\to\BAR M$ by
\[\bar u_{\ga}\const \bar u_{\al_i}(z_i).\] 
Note that the new component $\ga$ is a ``ghost'', i.e., the map $W_{\ga}$ (or $\bar u_{\ga}$) has 0 energy.  The tuple $(\W^\new,\z^\new)$ obtained from $(\W,\z)$ in this way is again a stable map.

We define the sequence of M\"obius transformations $\phi_{\ga}^\nu:S^2\to S^2$ by 
\begin{equation}
  \label{eq:phi ga nu}\phi_\ga^\nu(z):=\left\{
    \begin{array}{ll}z+z_i^\nu,&\textrm{if }\ga\in V^\new,\\
(z_k^\nu-z_i^\nu)z+z_i^\nu,&\textrm{if }\ga\in \BAR T^\new,\,i\geq1,\\
\frac{z_k^\nu-\phi_{\al_0}^\nu(w)}z+\phi_{\al_0}^\nu(w),&\textrm{if }\ga\in \BAR T^\new,\,i=0,
    \end{array}
\right.
\end{equation}
where $w\in S^2\wo \{z_0=\infty\}$ is chosen such that $\phi_{\al_0}^\nu(w)\neq z_k^\nu$ for all $\nu$. Note that the last line makes sense, since $\phi_{\al_0}^\nu(w)\neq\phi_{\al_0}^\nu(z_0)=\infty$. Here we use the convention that $1/\infty:=0$. 
\begin{claim}\label{claim:w new II} There exists a subsequence of $\big(W_\nu,z_0^\nu,\ldots,z_k^\nu\big)$ that converges to $(\W^\new,\z^\new)$ via the M\"obius transformations $(\phi_\al^\nu)_{\al\in T^\new,\,\nu\in\N}$. 
\end{claim}

\begin{proof}[Proof of Claim \ref{claim:w new II}] Condition (\ref{eq:E V bar T}) (energy conservation) holds for every subsequence, since the new component $\ga$ carries no energy. 

We denote now by $i\in\{0,\ldots,k-1\}$ the unique index such that $z_{\al_ik}=z_i$. 

{\bf Condition (\ref{defi:conv phi z}) of Definition \ref{defi:conv}} holds (for the new collection of M\"obius transformations), by an elementary argument. (To show the third part of this condition in the case $\ga\in\BAR T^\new$, we set $\psi_\ga:=\id$ if $i\geq1$, and $\psi_\ga(z):=1/z$ if $i=0$.)

{\bf We check condition \ref{defi:conv}(\ref{defi:conv al be})}. Let $\al E^\new\be$ be an edge. Consider the case $(\al,\be):=(\al_i,\ga)$. (It suffices to look at this case, the case $(\al,\be):=(\ga,\al_i)$ can be treated analogously.) We define
\[x:=z_{\ga\al_i}^\new=\infty,\quad y:=z_{\al_i\ga}^\new=z_i,\quad x_1^\nu:=z_i^\new=0,\]
\[x_2^\nu:=\left\{
  \begin{array}{ll}z_k^\nu-z_i^\nu,&\textrm{if }\ga\in V^\new,\\
z_k^\new=1,&\textrm{if }\ga\in\BAR T^\new,
  \end{array}
\right.\]
\[y_\nu:=\left\{
  \begin{array}{ll}z_{\al_i,0},&\textrm{if }\ga\in V^\new\textrm{ or }(\ga\in \BAR T^\new \textrm{ and }i\geq1),\\
w,&\textrm{if }\ga\in\BAR T^\new,\,i=0,
  \end{array}
\right.\]
\[\phi_\nu:=\phi_{\al_i\ga}^\nu:=(\phi_{\al_i}^\nu)^{-1}\circ\phi_\ga^\nu.\]
Then the hypotheses of Lemma \ref{le:phi nu} below are satisfied, and therefore by that lemma $\phi_{\al_i\ga}^\nu$ converges to $y=z_{\al_i\ga}^\new$, uniformly with all derivatives on every compact subset of $S^2\wo\{x\}=S^2\wo\{z_{\ga\al_i}^\new\}$. By Remark \ref{rmk:phi nu -1} below it follows that $\phi_{\ga\al_i}^\nu=(\phi_{\al_i\ga}^\nu)^{-1}$ converges to $z_{\ga\al_i}^\new$, uniformly on every compact subset of $S^2\wo\{z_{\al_i\ga}^\new\}$. This proves condition \ref{defi:conv}(\ref{defi:conv al be}).

{\bf We check condition \ref{defi:conv}(\ref{defi:conv w})} up to some subsequence. For every $\al\in T$ we write  
\[W_\al^\nu:=(\phi_\al^\nu)^*W_\nu,\quad\bar u_\al^\nu:=\bar u_{W_\al^\nu}:\R^2\to M/G,\]  
where $\bar u_{W_\al^\nu}$ is defined as in (\ref{eq:BAR u W}). 

{\bf Assume that $\ga\in V^\new$.} This means that $z_{ki}\neq\infty$. Since by definition $z_{k0}=\infty$, it follows that $i\neq0$. It follows from Proposition \ref{prop:cpt mod} (Compactness modulo bubbling) with $R_\nu:=1$, $r_\nu:=\nu$ and $W_\nu$ replaced by $W_\ga^\nu$, that there exists a vortex $\wt W_\ga$ on $\R^2$ such that passing to some subsequence, the sequence $W_\ga^\nu$ converges to $\wt W_\ga$, with respect to $\tau_{\R^2}$ (defined as in (\ref{eq:tau Si})), and the sequence $\bar u_\ga^\nu$ converges to $\bar u_{\wt W_\ga}$, uniformly on every compact subset of $\R^2$. 
\begin{claim}\label{claim:wt W ga W ga} We have $\wt W_\ga=W_\ga$ (defined as in (\ref{eq:W ga})).
\end{claim} 
\begin{proof}[Proof of Claim \ref{claim:wt W ga W ga}] To see this, we use condition \ref{defi:conv}(\ref{defi:conv w}) for the sequence $\big(W_\nu,z_0^\nu,\ldots,z_{k-1}^\nu\big)$, recalling that $\al_i\in \BAR T^\new$ and $i\neq0$. It follows that the maps $\bar u_{\al_i}^\nu$ converge to $\bar u_{\al_i}$, in $C^1$ on every compact subset of $S^2\wo Z_{\al_i}$, and hence on every small enough neighbourhood of $z_i$. (To make sense of this convergence, we implicitely mean that the image of a given compact subset of $S^2\wo Z_{\al_i}$ under $\bar u_{W_{\al_i}^\nu}$ is contained in $M^*/G$, for $\nu$ large enough.) 

Since $\phi_{\al_i\ga}^\nu:=(\phi_{\al_i}^\nu)^{-1}\circ\phi_\ga^\nu$ converges to $z_{\al_i\ga}=z_i$, uniformly on every compact subset of $S^2\wo\{z_{\ga\al_i}^\new\}=\R^2$, it follows that 
\[\bar u_\ga^\nu=\bar u_{\al_i}^\nu\circ\phi_{\al_i\ga}^\nu\to\BAR x_0:=\bar u_{\al_i}(z_i),\]
uniformly on every compact subset of $\R^2$. It follows that $\bar u_{\wt W_\ga}\const\BAR x_0$. We choose representatives $(\wt A_\ga,\wt u_\ga)\in\Om^1(\R^2,G)\x C^\infty(\R^2,G)$ of $\wt W_\ga$ and $x_0\in\mu^{-1}(0)$ of $\BAR x_0$. By hypothesis (H) the action of $G$ on $\mu^{-1}(0)$ is free. Hence, after regauging, we may assume that $\wt u_\ga\const x_0$. (Here we use Lemma \ref{le:g smooth} below, which ensures that the gauge transformation is smooth.) It follows from the first vortex equation that $\wt A_\ga=0$. This shows that $\wt W_\ga=W_\ga$ and hence proves Claim \ref{claim:wt W ga W ga}. 
\end{proof}
This proves condition \ref{defi:conv}(\ref{defi:conv w}) in the case $\ga\in V^\new$. 

{\bf Assume now that $\ga\in \BAR T^\new$. Suppose also that $i\geq1$.} By condition \ref{defi:conv}(\ref{defi:conv w}) for the sequence $\big(W_\nu,z_0^\nu,\ldots,z_{k-1}^\nu\big)$, the map $\bar u_{\al_i}^\nu$ converges to $\bar u_{\al_i}$, in $C^1$ on every compact subset of $S^2\wo Z_{\al_i}$ and hence on every small enough neighbourhood of $z_i$. Let $Q\sub\R^2=S^2\wo Z_\ga$ be a compact subset. Since $\phi_{\al_i\ga}^\nu$ converges to $z_i$, in $C^\infty$ on $Q$, it follows that $\bar u_\ga^\nu=\bar u_{\al_i}^\nu\circ\phi_{\al_i\ga}^\nu$ converges to $\bar u_\ga\const\bar u_{\al_i}(z_i)$ in $C^1$ on $Q$, as required.

{\bf Suppose now that $i=0$.} An elementary argument using condition (\ref{defi:conv w}) (with $\al:=\al_0$) in the definition of convergence, our assumption (\ref{eq:lim de limsup nu}), and Proposition \ref{prop:en conc}, shows that for every $\eps>0$ there exist numbers $R\geq r_0$ and $\nu_0\in\N$ such that
\begin{equation}\nn\bar d(\bar u_{\al_0}(\infty),\bar u_{\al_0}^\nu(z))<\eps,\quad\forall\nu\geq\nu_0,\,z\in\R^2\wo B_R.
\end{equation}
Since $\phi_{\al_0\ga}^\nu$ converges to $z_{\al_0\ga}=z_0=\infty$, uniformly on every compact subset of $\R^2=S^2\wo\{z_{\ga\al_0}\}$, it follows that $\bar u_\ga^\nu$ converges to the constant map $\bar u_\ga\const \bar u_{\al_0}(\infty)$, uniformly on every compact subset of $\R^2\wo\{0\}=S^2\wo(Z_\ga\disj\{z_0^\new\})$. We show that passing to some subsequence the convergence is in $C^1$ on every compact subset of $\R^2\wo\{0\}$. To see this, we define $R_\nu>0$, $\phi_\nu\in[0,2\pi),\wt\phi_\ga^\nu,\wt w_\ga^\nu$ by
\begin{eqnarray*}&R_\nu e^{i\phi_\nu}:=z_k^\nu-\phi_{\al_0}^\nu(w),\quad \wt \phi_\ga^\nu(\wt z):=\phi_\ga^\nu(e^{i\phi_\nu}/\wt z)=R_\nu z+\phi_{\al_0}^\nu(w),&\\
&\wt w_\ga^\nu:=(\wt A_\ga^\nu,\wt u_\ga^\nu):=(\wt\phi_\ga^\nu)^*w_\nu.&
\end{eqnarray*}
(Recall here that we have chosen $w\in S^2\wo \{z_0=\infty\}$ such that $\phi_{\al_0}^\nu(w)\neq z_k^\nu$ for all $\nu$.) An elementary argument shows that $R_\nu$ converges to $R_0:=\infty$. Hence by Proposition \ref{prop:cpt mod} with $r_\nu:=\nu$ and $w_\nu$ replaced by the $R_\nu$-vortex $\wt w_\ga^\nu$ there exist a finite subset $Z\sub\R^2$ and an $\infty$-vortex $\wt w_\ga:=(\wt A_\ga,\wt u_\ga)$, and passing to some subsequence, there exist gauge transformations $g_\ga^\nu\in W^{2,p}_\loc(\R^2\wo Z,G)$, such that the assertions \ref{prop:cpt mod}(\ref{prop:cpt mod:=},\ref{prop:cpt mod lim nu E eps}) hold. By \ref{prop:cpt mod}(\ref{prop:cpt mod:=}) on every compact subset of $\R^2\wo Z$ the sequence $(g_\ga^\nu)^*\wt A_\ga^\nu$ converges to $\wt A_\ga$ in $C^0$, and the sequence $(g_\ga^\nu)^{-1}\wt u_\ga^\nu$ converges to $\wt u_\ga$ in $C^1$. 
\begin{claim}\label{claim:Z sub 0}We have $Z\sub\{0\}$. 
\end{claim}
\begin{proof}[Proof of Claim \ref{claim:Z sub 0}] It follows from (\ref{eq:lim de limsup nu}) that there exists $R>0$ such that 
\[\limsup_{\nu\to\infty}E\big(w_\nu,\R^2\wo\phi_{\al_0}^\nu(B_R)\big)<\Emin.\]
Hence by an elementary argument, we have $E^{R_\nu}(\wt w_\ga^\nu,B_\eps(z))<\Emin$, for every $z\in\R^2\wo\{0\}$, for $\nu$ large enough. Combining this with condition \ref{prop:cpt mod}(\ref{prop:cpt mod lim nu E eps}), the statement of Claim \ref{claim:Z sub 0} follows. 
\end{proof}
Using Claim \ref{claim:Z sub 0}, it follows that $G\wt u_\ga^\nu$ converges to $G\wt u_\ga$ in $C^1$ on every compact subset of $\R^2\wo\{0\}$. We pass to some subsequence, such that $\phi_\nu$ converges to some number $\phi_0\in[0,2\pi]$. It follows that $\bar u_\ga^\nu=G u_\ga^\nu$ converges to the map $\C\wo\{0\}\ni z\mapsto G \wt u_\ga(e^{i\phi_0}/z)\in\BAR M$, in $C^1$ on every compact subset of $\C\wo\{0\}$. Since $\bar u_\ga^\nu$ also converges to $\bar u_\ga$, it follows that condition \ref{defi:conv}(\ref{defi:conv w}) holds in the case $\ga\in \BAR T^\new$, $i=0$. Hence this condition is satisfied in all cases. 

Condition \ref{defi:conv}(\ref{defi:conv z}) follows from the definition (\ref{eq:phi ga nu}) of $\phi_\ga^\nu$. This proves Claim \ref{claim:w new II}. 
\end{proof}

{\bf Assume now that Case (III) holds.} In this case we introduce a new vertex $\ga$ between $\al$ and $\be$, corresponding to a bubble in $\BAR M.$ Hence $\al$ and $\be$ are no longer adjacent, but are separated by $\ga$. We define
\begin{eqnarray*}&\BAR T^\new:=\BAR T\disj\{\ga\},\quad V^\new:=V,\quad T^\new:=T\disj\{\ga\},&\\
&z_{\al\ga}^\new:=z_{\al\be},\,z_{\be\ga}^\new:=z_{\be\al},\,z_{\ga\al}^\new:=0,\,z_{\ga\be}^\new:=\infty,\,\al_k^\new:=\ga,\,z_k^\new:=1.&
\end{eqnarray*}
We define $\bar u_\ga:S^2\to \BAR M$ to be the constant map equal to $\barev_{z_{\al\be}}(W_\al)$, where $\barev$ is defined as in (\ref{eq:barev},\ref{eq:barev w}). (In the case $\al\in\BAR T$ we denote $W_\al:=\bar u_\al$.) (The new component is a ``ghost'', i.e., carries no energy.) The tuple $(\W^\new,\z^\new)$ obtained from $(\W,\z)$ in this way is again a stable map. For every $\al\in T^\new$ we define  $z_{\al,0}^\new$ as in (\ref{eq:z al i}) and (\ref{eq:z al i }), with $i:=0$ and \wrt to the new tree $T^\new$. By interchanging $\al$ and $\be$ if necessary, we may assume w.l.o.g.~that $\be$ is contained in the chain of edges from $\al$ to $\al_0$. It follows that for every $\al\neq\ga$, $z_{\al,0}^\new=z_{\al,0}$, where $z_{\al,0}$ is defined as in (\ref{eq:z al i}) and (\ref{eq:z al i }), with $i:=0$ and \wrt to the old tree $T$, and $z_{\ga,0}^\new=z_{\ga,\be}^\new=\infty$. Furthermore, the hypotheses of Lemma \ref{le:middle} (Middle rescaling) are satisfied with 
\begin{eqnarray}\nn&x:=z_{\be\ga}^\new,\quad x':=z_{\be,0},\quad x_\nu:=z_{\be k}^\nu:=(\phi_\be^\nu)^{-1}(z_k^\nu),\quad y:=z_{\al\ga}^\new,\\
\nn& \phi_\nu:=\phi_{\al\be}^\nu=(\phi_\al^\nu)^{-1}\circ\phi_\be^\nu.&
\end{eqnarray}
We choose a sequence $\psi_\nu$ as in this lemma (satisfying $\psi_\nu(\infty)=z_{\be,0}$). We define 
\[\phi_\ga^\nu:=\phi_\be^\nu\circ\psi_\nu.\]
The sequence $(W_\nu,z_0^\nu,\ldots,z_k^\nu)$ converges to $(\W^\new,\z^\new)$ along the sequence of collections of M\"obius transformations $(\phi_\al^\nu)_{\al\in T^\new,\nu\in\N}$. This follows from elementary arguments, except for the proof of condition \ref{defi:conv}(\ref{defi:conv w}), which uses (\ref{eq:lim limsup be al}) and an argument as in Case (II).

This proves the induction step and hence terminates the proof of Theorem \ref{thm:bubb} in the general case.
\end{proof}
\noindent{\bf Remark.} In the above proof the stable map $(\W,\z)$ is constructed by ``terminating induction''. Intuitively, this is induction over the integer $N$ occuring in Claim \ref{claim:tree}. The ``auxiliary index'' $\ell$ in Claim \ref{claim:tree} is needed to make this idea precise. Condition (\ref{claim:tree compl}) and the inequality (\ref{eq:N E}) ensure that the ``induction stops''. $\Box$
\appendix
\section{Vortices}\label{sec:vort} 

Let $M,\om,G,\g,\lan\cdot,\cdot\ran_\g,\mu,J$ be as in Section \ref{sec:main}. (As always, we assume that hypothesis (H) is satisfied.) The next result is used in the proofs of Propositions \ref{prop:soft} and \ref{prop:en conc}. For $r>0$ and $z_0\in \R^2$ we denote by $B_r(z_0)$ the open ball in $\R^2$, and we abbreviate $B_r:=B_r(z_0)$. 
\begin{lemma}[Bound on energy density]\label{le:a priori} Let $K\sub M$ be a compact subset. Then there exists a constant $E_0>0$ such that the following holds. Let $z_0\in\R^2$, $r>0$, $P$ be a smooth principal over $B_r(z_0)$, $p>2$, and $(A,u)$ a vortex on $P$ of class $W^{1,p}_\loc$, such that
\begin{eqnarray}
\nn &u(P)\sub K,&\\
\nn&E(w,B_r(z_0))\leq E_0.&
\end{eqnarray}
Then we have 
\begin{equation}
\nn e_w(z_0)\leq \frac8{\pi r^2}E(w,B_r(z_0)).
\end{equation}
\end{lemma}
For the proof of Lemma \ref{le:a priori} we need the following lemma.
\begin{lemma}[Heinz]\label{le:Heinz} Let $r>0$ and $c\geq0$. Then for every function $f\in C^2(B_r,\R)$ satisfying the inequalities
\begin{equation}
\nn f\geq0,\qquad \La f\geq-cf^2,\qquad \int_{B_r}f<\frac\pi{8c}
\end{equation}
we have
\begin{equation}
\nn f(0)\leq\frac8{\pi r^2}\int_{B_r}f.
\end{equation}
\end{lemma}
\begin{proof}[Proof of Lemma \ref{le:Heinz}] This is \cite[Lemma 4.3.2]{MS}.
\end{proof}
\begin{proof}[Proof of Lemma \ref{le:a priori}] Since $G$ is compact, we may assume w.l.o.g.~that $K$ is $G$-invariant. The result then follows from Theorem \ref{thm:reg gauge bdd} below, the calculation in Step 1 of the proof of \cite[Proposition 11.1.]{GS}, and Lemma \ref{le:Heinz}.
\end{proof}
Lemma \ref{le:a priori} has the following consequence.
\begin{cor}[Quantization of energy]\label{cor:quant} If $M$ is equivariantly convex at $\infty$, then we have
\[\inf_wE(w)>0,\] 
where $w=(P,A,u)$ ranges over all vortices on $\R^2$ with $P$ smooth and $(A,u)$ of class $W^{1,p}_\loc$ for some $p>2$, such that $E(w)>0$ and $\bar u(P)$ is compact.
\end{cor}
\begin{proof}[Proof of Corollary \ref{cor:quant}] This is an immediate consequence of Proposition \ref{prop:bounded} below and Lemma \ref{le:a priori}.
\end{proof}
This corollary implies that the minimal energy $E_V$ of a vortex on $\R^2$ (defined as in (\ref{eq:E V})) is positive, and therefore $\Emin>0$ (defined as in (\ref{eq:Emin})). 

The next lemma is used in the proofs of Proposition \ref{prop:cpt bdd} and Lemma \ref{le:conv e}. It is a consequence of \cite[Lemma 9.1]{GS}. Let $(\Si,\om_\Si,j)$ be a surface with an area form and a compatible complex structure. For $\xi\in\g$ and $x\in M$ we denote by $L_x\xi=X_\xi(x)\in T_xM$ the (infinitesimal) action of $\xi$ at $x$. 
\begin{lemma}[Bounds on the moment map component]\label{le:mu u} Let $c>0$, $Q\sub\Si\wo\dd\Si$ and $K\sub M$ be compact subsets, and $p>2$. Then there exist positive constants $R_0$ and $C_p$ such that the following holds. Let $R\geq R_0$, $P$ a smooth principal bundle over $\Si$, and $(A,u)$ an $R$-vortex on $P$ of class $W^{1,p}_\loc$, such that 
\begin{eqnarray}\nn &u(P)\sub K,&\\
\nn &\Vert d_Au\Vert_{L^\infty(\Si)}\leq c,&\\
\nn &|\xi|\leq c|L_{u(p)}\xi|,\forall p\in P,\forall \xi\in\g.&
\end{eqnarray}
Then 
\[\int_Q|\mu\circ u|^p\om_\Si\leq C_pR^{-2p},\qquad \sup_Q|\mu\circ u|\leq C_pR^{2/p-2},\]
where we view $|\mu\circ u|$ as a function from $\Si$ to $\R$. 
\end{lemma}

\begin{proof}[Proof of Lemma \ref{le:mu u}] This follows from the proof of \cite[Lemma 9.1]{GS}, using Theorem \ref{thm:reg gauge bdd}. 
\end{proof}
The next result is used in the proofs of Propositions \ref{prop:en conc} and \ref{prop:reg gauge}, and Lemma \ref{le:a priori}.
\begin{thm}[Regularity modulo gauge over compact surface]\label{thm:reg gauge bdd} Let $k\in\N\cup\{\infty\}$, $P$ a smooth principal $G$-bundle over $\Si$, $p>2$, and $(A,u)$ a vortex on $P$ of class $W^{1,p}$. Then there exists a gauge transformation $g\in W^{2,p}(\Si,G)$ such that $g^*w$ is smooth over $\Si\wo\dd\Si$. 
\end{thm}
\begin{proof}[Proof of Theorem \ref{thm:reg gauge bdd}] This follows from the proof of \cite[Theorem 3.1]{CGMS}, using a version of the local slice theorem allowing for boundary (see \cite[Theorem 8.1]{Wehrheim}).
\end{proof}
The next result is used in the proof of Proposition \ref{prop:bounded} below.
\begin{prop}[Regularity modulo gauge over $\R^2$]\label{prop:reg gauge} Let $R\geq0$ be a number, $P$ a smooth principal $G$-bundle over $\R^2$, $p>2$, and $w:=(A,u)$ an $R$-vortex on $P$ of class $W^{1,p}_\loc$. Then there exists a gauge transformation $g$ on $P$ of class $W^{2,p}_\loc$ such that $g^*w$ is smooth.
\end{prop}
The proof of Proposition \ref{prop:reg gauge} follows the lines of the proofs of \cite[Theorems 3.6 and Theorem A.3]{FrPhD}. 
\begin{proof}[Proof of Proposition \ref{prop:reg gauge}]\setcounter{claim}{0} 
\begin{claim}\label{claim:g j} There exists a collection $(g_j)_{j\in\N}$, where $g_j$ is a gauge transformation over $B_{j+1}$ of class $W^{2,p}$, such that for every $j\in\N$, we have
\begin{eqnarray}\label{eq:g j w}&g_j^*w\textrm{ smooth over }B_{j+1},\\
\label{eq:g j g j-1}&g_{j+1}=g_j\textrm{ over }B_j.&
\end{eqnarray}
\end{claim}
\begin{proof}[Proof of Claim \ref{claim:g j}] By Theorem \ref{thm:reg gauge bdd} there exists a gauge transformation $g_1\in W^{2,p}(B_2,G)$ such that $g_1^*w$ is smooth. Let $\ell\in\N$ be an integer and assume by induction that there exist gauge transformations $g_j\in W^{2,p}(B_{j+1},G)$, for $j=1,\ldots,\ell$, such that (\ref{eq:g j w}) holds for $j=1,\ldots,\ell$, and (\ref{eq:g j g j-1}) holds for $j=1,\ldots,\ell-1$. We show that there exists a gauge transformation $g_{\ell+1}\in W^{2,p}(B_{\ell+2},G)$ such that 
\begin{eqnarray}\label{eq:g ell+1}&g_{\ell+1}^*w\textrm{ smooth over }B_{\ell+2},&\\
\label{eq:g ell+1 g ell}&g_{\ell+1}=g_\ell\textrm{ over }B_\ell.& 
\end{eqnarray}
We choose a smooth function $\rho:\bar B_{\ell+2}\to B_{\ell+1}$ such that $\rho(z)=z$ for $z\in B_\ell$. By Theorem \ref{thm:reg gauge bdd} there exists a gauge transformation $h\in W^{2,p}(B_{\ell+2},G)$ such that 
\begin{equation}\nn h^*w\textrm{ smooth over }\bar B_{\ell+2}.
\end{equation}
We define
\[g_{\ell+1}:=h\big((h^{-1}g_\ell)\circ\rho\big).\]
Then $g_{\ell+1}$ is of class $W^{2,p}$ over $B_{\ell+2}$, and (\ref{eq:g ell+1 g ell}) is satisfied. Furthermore, $h^*w$ is of class $W^{k+1,p}$ over $B_{\ell+2}$, and 
\[g_\ell^*w=(h^{-1}g_\ell)^*h^*w\textrm{ smooth over }B_{\ell+1}.\] 
Therefore, Lemma \ref{le:g smooth}(\ref{le:g k+1 p}) below implies that $h^{-1}g_\ell$ is of class $W^{k+2,p}$ over $B_{\ell+1}$. It follows that $g_{\ell+1}^*w=\big((h^{-1}g_\ell)\circ\rho\big)^*h^*w$ is smooth over $B_{\ell+2}$. Hence (\ref{eq:g ell+1}) is satisfied. This terminates the induction and concludes the proof of Claim \ref{claim:g j}.  
\end{proof}
We choose a collection $(g_j)$ as in Claim \ref{claim:g j}, and define $g$ to be the unique gauge transformation on $P$ that restricts to $g_j$ over $B_j$. This makes sense by (\ref{eq:g j g j-1}). Furthermore, (\ref{eq:g j w}) implies that $g^*w$ is smooth. This proves Proposition \ref{prop:reg gauge}. 
\end{proof}
The next result is used in the proof of Proposition \ref{prop:cpt bdd}.
\begin{thm}[Compactness for vortices over compact surface]\label{thm:cpt cpt} Let $\Si$ be a compact surface (possibly with boundary), $\om_\Si$ an area form, $j$ a compatible complex structure on $\Si$, $P$ a principal $G$-bundle over $\Si$, $K\sub M$ a compact subset, $R_\nu\in[0,\infty)$, $p>2$, and $(A_\nu,u_\nu)$ an $R_\nu$-vortex on $P$ of class $W^{1,p}$, for every $\nu\in\N$. Assume that $R_\nu$ converges to some $R_0\in[0,\infty)$, and
\[u_\nu(P)\sub K,\quad\sup_\nu\Vert d_{A_\nu}u_\nu\Vert_{L^p(\Si)}<\infty.\]
Then there exist a smooth $R_0$-vortex $(A_0,u_0)$ on $P|\Si\wo\dd\Si$ and gauge transformations $g_\nu$ on $P$ of class $W^{2,p}$, such that $g_\nu^*(A_\nu,u_\nu)$ converges to $(A_0,u_0)$, in $C^\infty$ on every compact subset of $\Si\wo\dd\Si$. 
\end{thm}
\begin{proof}[Proof of Theorem \ref{thm:cpt cpt}] This follows from a modified version of the proof of \cite[Theorem 3.2]{CGMS}: We use a version of Uhlenbeck compactness for a compact base with boundary, see Theorem \ref{thm:Uhlenbeck compact} below, and a version of the local slice theorem allowing for boundary, see \cite[Theorem 8.1]{Wehrheim}. Note that the proof carries over to case in which $R_\nu=0$ for some $\nu\in\N$, or $R_0=0$. 
\end{proof}
The following result was used in the proofs of Theorem \ref{thm:bubb} and Corollary \ref{cor:quant}. 
\begin{prop}[Boundedness of image]\label{prop:bounded} Assume that $M$ is equivariantly convex at $\infty$. Then there exists a $G$-invariant compact subset $K_0\sub M$ such that the following holds. Let $p>2$, $P$ a principal $G$-bundle over $\R^2$, and $(A,u)$ a vortex on $P$ of class $W^{1,p}_\loc$, such that $E(w)<\infty$ and $\BAR{u(P)}$ is compact. Then we have $u(P)\sub K_0$.
\end{prop}
\begin{proof}[Proof of Proposition \ref{prop:bounded}] Let $P$ be a principal $G$-bundle over $\R^2$. By an elementary argument every smooth vortex on $P$ is smoothly gauge equivalent to a smooth vortex that is in radial gauge outside $B_1$. Using Proposition \ref{prop:reg gauge}, it follows that every vortex on $P$ of class $W^{1,p}_\loc$ is gauge equivalent to a smooth vortex that is in radial gauge outside $B_1$. Hence the statement of Proposition \ref{prop:bounded} follows from \cite[Proposition 11.1]{GS}.
\end{proof}
The following lemma was used in the proof of Proposition \ref{prop:simple}. Consider the action of the group of translations of $\R^2$ on the set of equivalence classes of smooth vortices over $\R^2$ given by (\ref{eq:phi []}).
\begin{lemma}\label{le:trans} The restriction of this action to the set of vortices of finite positive energy is free.
\end{lemma}
\begin{proof}[Proof of Lemma \ref{le:trans}] \setcounter{claim}{0} Assume that $W$ is a smooth vortex over $\R^2$ and $v\in\R^2$ is a vector, such that defining $T:\R^2\to\R^2$ by $Tz:=z+v$, we have $T^*W=W$. Then $e_W(z+nv)=e_W(z)$ for every $z\in\R^2$ and $n\in\Z$. It follows that $E(W)=\infty$, $e_W\const0$, or $v=0$. Lemma \ref{le:trans} follows from this.
\end{proof}
\section{Further auxiliary results}\label{sec:add}
The proof of Proposition \ref{prop:en conc} (Energy concentration at ends) is based on an isoperimetric inequality for the invariant action functional (Theorem \ref{thm:isoperi} below). Building on work by D.~A.~Salamon and R.~Gaio \cite{GS}, we define this functional as follows. (This is the definition from \cite{ZiA}, written in a more intrinsic way.)

We first review the usual symplectic action functional: Let $(M,\om)$ be a symplectic manifold without boundary. We fix a Riemannian metric $\lan\cdot,\cdot\ran_M$ on $M$, and denote by $d,\exp,|v|,\iota_x>0$, and $\iota_X:=\inf_{x\in X}\iota_x\geq0$ the distance function, the exponential map, the norm of a vector $v\in TM$, and the injectivity radii of a point $x\in M$ and a subset $X\sub M$, respectively. We define the symplectic action of a loop $x:S^1\to M$ of length $\ell(x)<2\iota_{x(S^1)}$ to be
\begin{equation}
\nn \A(x):=-\int_\D u^*\om.
\end{equation}
Here $\D\sub\R^2$ denotes the (closed) unit disk, and $u:\D\to M$ is any smooth map such that 
\begin{equation}\nn u(e^{it})=x(t),\,\forall t\in \R/(2\pi\Z)\iso S^1,\quad d\big(u(z),u(z')\big)<\iota_{x(S^1)},\,\forall z,z'\in\D.
\end{equation}
\begin{lemma}\label{le:A} The action $\A(x)$ is well-defined, i.e., a map $u$ as above exists, and $\A(x)$ does not depend on the choice of $u$. 
\end{lemma}
\begin{proof} The lemma follows from an elementary argument, using the exponential map $\exp_{x(0+\Z)}:T_{x(0+\Z)}M\to M$. \end{proof}
Let now $G$ be a compact connected Lie group with Lie algebra $\g$. Suppose that $G$ acts on $M$ in a Hamiltonian way, with (equivariant) moment map $\mu:M\to\g^*$, and that $\lan\cdot,\cdot\ran_M$ is $G$-invariant. We denote by $\lan\cdot,\cdot\ran:\g^*\x\g\to\R$ the natural contraction. Let $P$ be a smooth principal $G$-bundle over $S^1$ and $x\in C^\infty_G(P,M)$. We call $(P,x)$ \emph{admissible} iff there exists a section $s:S^1\to P$ such that $\ell(x\circ s)<2\iota_{x(P)}$, and 
\begin{equation}\nn\A(g(x\circ s))-\A(x\circ s)=\int_{S^1}\big\lan\mu\circ x\circ s,g^{-1}dg\big\ran,
\end{equation}
for every $g\in C^\infty(S^1,G)$ satisfying $\ell(g(x\circ s))\leq \ell(x\circ s)$. 
\begin{defi}\label{defi:action} Let $(P,x)$ be an admissible pair, and $a$ be a connection on $P$. We define the \emph{invariant symplectic action of $(P,a,x)$} to be 
\begin{equation}\nn\A(P,a,x):=\A(x\circ s)+\int_{S^1}\big\lan\mu\circ x\circ s,a\,ds\big\ran,
\end{equation}
where $s:S^1\to P$ is a section as above.
\end{defi}
(This is a modified version of the ``local equivariant symplectic action functional'' introduced by A.~R.~Gaio and D.~A.~Salamon in \cite{GS}.) To formulate the isoperimetric inequality, we need the following. If $X$ is a manifold, $P$ a principal $G$-bundle over $X$ and $u\in C^\infty_G(P,M)$, then we define $\bar u:X\to M$ by $\bar u(y):=Gu(p)$, where $p\in P$ is any point in the fiber over $y$. We define $M^*$ as in (\ref{eq:M *}). For a loop $\bar x:S^1\to M^*/G$ we denote by $\bar\ell(\bar x)$ its length \wrt the Riemannian metric on $M^*/G$ induced by $\lan\cdot,\cdot\ran_M$. Furthermore, for each subset $X\sub M$ we define 
\[m_X:=\inf\big\{|L_x\xi|\,\big|\,x\in X,\,\xi\in\g:\,|\xi|=1\big\}.\]
The next result is Theorem 1.2 in \cite{ZiA}. 
\begin{thm}[Sharp isoperimetric inequality]\label{thm:isoperi} Assume that there exists a $G$-invariant $\om$-compatible almost complex structure $J$ such that $\lan\cdot,\cdot\ran_M=\om(\cdot,J\cdot)$. Then for every compact subset $K\sub M^*$ and every constant $c>\frac12$ there exists a constant $\de>0$ with the following property. Let $P$ be a principal $G$-bundle over $S^1$ and $x\in C^\infty_G(P,M)$, such that $x(P)\sub K$ and $\bar\ell(\bar x)\leq\de$. Then $(P,x)$ is admissible, and for every connection $a$ on $P$ we have
\begin{equation}\nn|\A(P,a,x)|\leq c\Vert d_ax\Vert_2^2+\frac1{2m_K^2}\Vert\mu\circ x\Vert_2^2.
\end{equation}
\end{thm}
Here we view $d_ax$ as a one-form on $S^1$ with values in the bundle $(x^*TM)/G\to S^1$, and $\mu\circ x$ as a section of the co-adjoint bundle $(P\x\g^*)/G\to S^1$. Furthermore, $S^1$ is identified with $\R/(2\pi\Z)$, and the norms are taken with respect to the standard metric on $\R/(2\pi\Z)$, the metric $\lan\cdot,\cdot\ran_M$ on $M$, and the operator norm on $\g^*$ induced by $\lan\cdot,\cdot\ran_\g$. (Note that in \cite[Theorem 1.2]{ZiA} $S^1$ was identified with $\R/\Z$ instead. Note also that hypothesis (H) is not needed for Theorem \ref{thm:isoperi}.) 

In the proof of Proposition \ref{prop:en conc} we also used the following result. For $s\in\R$ we denote by $\iota_s:S^1\to\R\x S^1$ the map given by $\iota_s(t):=(s,t)$. Let $X,X'$ be manifolds, $f\in C^\infty(X',X)$, $P$ a principal $G$-bundle over $X$, $A$ a connection on $P$, and $u\in C^\infty_G(P,M)$. Then the pullback triple $f^*(P,A,u)$ consists of a principal $G$-bundle $P'$ over $X'$, a connection on $P'$, and an equivariant map from $P'$ to $M$. 
\begin{prop}[Energy action identity]\label{prop:en act} For every compact subset $K\sub M^*$ there exists a constant $\de>0$ with the following property. Let $s_-\leq s_+$ be numbers, $\Si:=[s_-,s_+]\x S^1$, $\om_\Si$ an area form on $\Si$, $j$ a compatible complex structure, and $w:=(A,u)$ a smooth vortex over $\Si$ (with respect to $(\om_\Si,j)$, such that $u(P)\sub K$ and $\bar\ell(\bar u\circ\iota_s)<\de$, for every $s\in [s_-,s_+]$. Then the pairs $\iota_{s_\pm}^*(P,u)$ are admissible, and 
\begin{equation}
\nn E(w,\Si)=-\A\big(\iota_{s_+}^*(P,A,u)\big)+\A\big(\iota_{s_-}^*(P,A,u)\big).
\end{equation}
\end{prop}
\begin{proof}[Proof of Proposition \ref{prop:en act}] This follows from \cite[Proposition 3.1]{ZiA}.
\end{proof}
The next result is used in the proof of Proposition \ref{prop:cpt bdd}. It is \cite[Theorem A]{Wehrheim}. See also \cite[Theorem 1.5]{Uhlenbeck}. 
\begin{thm}[Uhlenbeck compactness]\label{thm:Uhlenbeck compact} Let $n\in\N$, $G$ be a compact Lie group, $X$ a compact smooth Riemannian $n$-manifold (possibly with boundary), $P$ a principal $G$-bundle over $X$, $p>n/2$ a number, and $A_\nu$ a sequence of connections on $P$ of class $W^{1,p}$. Assume that 
\begin{equation}
\nn\sup_{\nu\in\N}\Vert F_{A_\nu}\Vert_{L^p(X)}<\infty.
\end{equation}
Then passing to some subsequence there exist gauge transformations $g_\nu$ of class $W^{2,p}$, such that $g_\nu^*A_\nu$ converges weakly in $W^{1,p}$.
\end{thm}
The next result was used in the proof of Proposition \ref{prop:cpt bdd}. Its proof goes along the lines of the proof of \cite[Proposition B.4.2]{MS}. 
\begin{prop}[Compactness for $\bar \del_J$]\label{prop:compactness delbar}\label{prop:comp delbar} Let $M$ be a manifold without boundary, $k\in\N$, $p>2$, $J$ an almost complex structure on $M$ of class $C^k$, $\Om_1\sub\Om_2\sub\ldots\sub\C$ open subsets, and $u_\nu:\Om_\nu\to M$ a sequence of functions of class $W^{1,p}_\loc$. Assume that $\bar\del_Ju_\nu$ is of class $W^{k,p}_\loc$, for every $\nu$, and that for every open subset $\Om\sub\bigcup_\nu\Om_\nu$ with compact closure the following holds. If $\nu_0\in\N$ is so large that $\Om\sub \Om_{\nu_0}$ then 
  \begin{eqnarray}
    \label{eq:u nu K}&\exists K\sub M\textrm{ compact: }u_\nu(\Om)\sub K,\quad\forall \nu\geq\nu_0,&\\
  \label{eq:du nu}&\sup_{\nu\geq\nu_0}\Vert du_\nu\Vert_{L^p(\Om)}<\infty,&  \\
\label{eq:sup k p}&\sup_{\nu\geq\nu_0}\Vert\bar\del_Ju_\nu\Vert_{W^{k,p}(\Om)}<\infty.&
\end{eqnarray}
Then there exists a subsequence of $u_\nu$ that converges weakly in $W^{k+1,p}$ and in $C^k$ on every compact subset of $\bigcup_\nu\Om_\nu$. 
\end{prop}
The next lemma is used in the proofs of Propositions \ref{prop:cpt bdd} and \ref{prop:reg gauge}, and Theorem \ref{thm:bubb}. 
\begin{lemma}[Regularity of the gauge transformation]\label{le:g smooth} Let $X$ be a smooth manifold, $G$ a compact Lie group, $P$ a principal $G$-bundle over $X$, $k\in\N\cup\{0\}$, and $p>\dim X$. Then the following assertions hold.
\begin{enui}
\item \label{le:g C k+1}Let $g$ be a gauge transformation of class $W^{1,p}_\loc$ and $A$ a connection on $P$ of class $C^k$, such that $g^*A$ is of class $C^k$. Then $g$ is of class $C^{k+1}$. 
\item\label{le:g k+1 p} Assume that $X$ is compact (possibly with boundary). Let $\U$ be a subset of the space of $W^{k,p}$-connections on $P$ that is bounded in $W^{k,p}$. Then there exists a $W^{k+1,p}$-bounded subset $\V$ of the set of $W^{k+1,p}$-gauge transformations on $P$, such that the following holds. Let $A\in\U$ and $g$ be a $W^{1,p}$-gauge transformation, such that $g^*A\in\U$. Then $g\in\V$. 
\end{enui}
\end{lemma}
\begin{proof}[Proof of Lemma \ref{le:g smooth}] This follows from induction in $k$, using the equality $dg=g\big(g^*A)-Ag$ and Morrey's inequality (for (\ref{le:g k+1 p})). (See \cite[Lemma A.8]{Wehrheim}.)
\end{proof}
The next proposition is used in the proof of Proposition \ref{prop:quant en loss} (Quantization of energy loss). 
\begin{prop}\label{prop:ka 0} Let $n\in\N$, $G$ a compact Lie group, $P$ be a principal $G$-bundle over $\R^n$, and $A,A'$ smooth flat connections on $P$. Then there exists a smooth gauge transformation $g$ such that $A'=g^*A$. 
\end{prop}
\begin{proof}[Proof of Proposition \ref{prop:ka 0}] (In the case $n=2$, see also \cite[Corollary 3.7]{FrPhD}.) In the case $n=1$ such a $g$ exists, since then the condition $A'=g^*A$ can be viewed as an ordinary differential equation for $g$. Let $n\in\N$ and assume by induction that we have already proved the statement for $n$. Let $P$ be a principal $G$-bundle over $\R^{n+1}$, and $A,A'$ smooth flat connections on $P$. We define $\iota:\R^n\to\R^{n+1}$ by $\iota(x):=(x,0)$. By the induction hypothesis there exists a smooth gauge transformation $g_0$ on $\iota^*P\to\R^n$, such that 
\begin{equation}\label{eq:g 0 iota A}g_0^*\iota^*A=\iota^*A'.
\end{equation} 
Since $P$ is trivializable, there exists a smooth gauge transformation $\wt g_0$ on $P$ such that $\iota^*\wt g_0=g_0$. 

Let $x\in\R^n$. We define $\iota_x:\R\to\R^{n+1}$ by $\iota_x(t):=(x,t)$. There exists a unique smooth gauge transformation $h_x$ on $\iota_x^*P\to\R$, such that
\begin{equation}\label{eq:iota x wt g 0}h_x^*\iota_x^*\wt g_0^*A=\iota_x^*A',\quad h_x(p)=\one,\,\forall p\in\textrm{ fiber of }\iota_x^*P\textrm{ over }0\in\R.
\end{equation}
To see this, note that these conditions can be viewed as an ordinary differential equation for $h_x$ with prescribed initial value. Since this solution depends smoothly on $x$, there exists a unique smooth gauge transformation $h$ on $P$ such that $\iota_x^*h=h_x$, for every $x\in\R^n$. The gauge transformation $g:=\wt g_0h$ on $P$ satisfies the equation $A'=g^*A$. This follows from (\ref{eq:g 0 iota A},\ref{eq:iota x wt g 0}) and flatness of $A$ and $A'$.
\end{proof}
The next result was used in the proofs of Proposition \ref{prop:cpt bdd}, Remark \ref{rmk:bar J}, and Theorem \ref{thm:bubb}. Let $M,\om,G,\g,\lan\cdot,\cdot\ran_\g,\mu,J,\Si,\om_\Si,j$ be as in Section \ref{sec:main}. We define the almost complex structure $\bar J$ on $\BAR M$ as in (\ref{eq:bar J}). For every open subset $\Om\sub S^2\iso\C\cup\{\infty\}$ the energy density of a map $f\in W^{1,p}(\Om,\BAR M)$ is given by
\[e_f(z):=\frac12|df|^2,\]
where the norm is with respect to the metrics $\om_\Si(\cdot,j\cdot)$ on $\Si$ and $\BAR\om(\cdot,\bar J\cdot)$ on $\BAR M$. Let $P$ be a smooth principal bundle over $\Si$, $A$ a connection on $P$, and $u:P\to M$ an equivariant map. We define
\[e^\infty_{A,u}=\frac12|d_Au|^2,\]
where the norm is taken with respect to the metrics $\om_\Si(\cdot,j\cdot)$ on $\Si$ and $\om(\cdot,J\cdot)$ on $M$. Furthermore, we define 
\[\bar u:\Si\to\BAR M,\quad\bar u(z):=Gu(p),\]
where $p\in P$ is an arbitrary point in the fiber over $z$. 
\begin{prop}[Pseudo-holomorphic curves in symplectic quotient]\label{prop:bar del J} Let $P$ be a smooth principal $G$-bundle over $\Si$, $p>2$, $A$ a $W^{1,p}_\loc$-connection on $P$, and $u:P\to M$ a $G$-equivariant map of class $W^{1,p}_\loc$, such that $\mu\circ u=0$. Then we have
\[e_{\bar u}=e^\infty_{A,u}.\]
If $(A,u)$ also solves the equation $\bar\del_{J,A}(u)=0$ then
\begin{eqnarray}\nn\bar\del_{\bar J}\bar u=0.
\end{eqnarray}
\end{prop}
\begin{proof}[Proof of Proposition \ref{prop:bar del J}] This follows from an elementary argument. For the second part see also \cite[Section 1.5]{Gaio}. 
\end{proof}
In the proof of Theorem \ref{thm:bubb} we used the following lemma.
\begin{lemma}[Bound for tree]\label{le:weight} Let $k\in\N\cup\{0\}$ be a number, $(T,E)$ a finite tree, $\al_1,\ldots,\al_k\in T$ vertices, $f:T\to[0,\infty)$ a function, and $E_0>0$ a number. Assume that for every vertex $\al\in T$ we have 
\begin{equation}
\label{eq:f al E}f(\al)\geq E_0\quad\textrm{or}\quad\#\big\{\be\in T\,|\,\al E\be\big\}+\#\big\{i\in\{1,\ldots,k\}\,|\,\al_i=\al\big\}\geq3.
\end{equation}
Then 
\begin{equation}
\nn\#T\leq\frac{2\sum_{\al\in T}f(\al)}{E_0}+k.
\end{equation}
\end{lemma}
\begin{proof}[Proof of Lemma \ref{le:weight}] This follows from an elementary argument. (It is Exercise 5.1.2.~in the book \cite{MS}.)
\end{proof}
We used the following facts about sequences of M\"obius transformations in the proof of the Bubbling Theorem in the case $k\geq1$.
\begin{rmk}\label{rmk:phi nu -1}\rm Let $x,y\in S^2$ be points and $\phi_\nu$ a sequence of M\"obius transformations that converges to $y$, uniformly on every compact subset of $S^2\wo \{x\}$. Then $\phi_\nu^{-1}$ converges to $x$, uniformly on every compact subset of $S^2\wo \{y\}$. This follows from an elementary argument. (It is Exercise D.1.3 in the book \cite{MS}.) $\Box$
\end{rmk}
\begin{lemma}[Convergence for M\"obius transformations]\label{le:phi nu} Let $\phi_\nu$ be a sequence of M\"obius transformations and $x,y\in S^2$ be points. Suppose there exist convergent sequences $x_1^\nu,x_2^\nu,y^\nu\in S^2$ such that
\begin{eqnarray}
\nn &x\neq\lim_{\nu\to\infty}x_1^\nu\neq\lim_{\nu\to\infty}x_2^\nu\neq x,\qquad y\neq\lim_{\nu\to\infty}y^\nu,&\\
\nn&\lim_{\nu\to\infty}\phi_\nu(x_1^\nu)=\lim_{\nu\to\infty}\phi_\nu(x_2^\nu)=y,\qquad\lim_{\nu\to\infty}\phi_\nu^{-1}(y^\nu)=x.
\end{eqnarray}
Then $\phi_\nu$ converges to $y$, uniformly with all derivatives on every compact subset of $S^2\wo \{x\}$.
\end{lemma}
\begin{proof} This follows from \cite[Lemmata D.1.4 and 4.6.6]{MS}.
\end{proof}
\begin{lemma}[Middle rescaling]\label{le:middle} Let $x,x_\nu,y\in S^2$ be points and $\phi_\nu$ be a sequence of M\"obius transformations that converges to $y$, uniformly on compact subsets of $S^2\wo \{x\}$, such that $x_\nu$ converges to $x$ and $\phi_\nu(x_\nu)$ converges to $y$. Then there exists a sequence of M\"obius transformations $\psi_\nu$ such that $\psi_\nu(1)=x_\nu$, $\psi_\nu$ converges to $x$, uniformly with all derivatives on compact subsets of $S^2\wo\{\infty\}$, and $\phi_\nu\circ\psi_\nu$ converges to $y$, uniformly with all derivatives on compact subsets of $S^2\wo\{0\}$. Moreover, if $x'\neq x$ is any point in $S^2$ then we may choose $\psi_\nu$ such that $\psi_\nu(\infty)=x'$. 
\end{lemma}
\begin{proof}[Proof of Lemma \ref{le:middle}] Let $x'\neq x$ and $y''\neq y$ be any two points in $S^2$. It follow from Remark \ref{rmk:phi nu -1} that for $\nu$ large enough the three points $x''_\nu:=\phi_\nu^{-1}(y'')$, $x_\nu$, $x'$ are all distinct. W.l.o.g.~we may assume that this holds for every $\nu$. For $\nu\in\N$ we define $\psi_\nu$ to be the unique M\"obius transformation such that
\[\psi_\nu(0)=x''_\nu,\quad \psi_\nu(1)=x_\nu,\quad \psi_\nu(\infty)=x'.\]
Then the hypotheses of Lemma \ref{le:phi nu} with $\phi_\nu,x,y$ replaced by $\psi_\nu,\infty,x$ and $x_1^\nu:=0$, $x_2^\nu:=1$ and $y_\nu:=x'$ are satisfied. Hence by that Lemma the maps $\psi_\nu$ converge to $x$, uniformly with all derivatives on compact subsets of $S^2\wo\{\infty\}$. Moreover, the hypotheses of the same lemma with $\phi_\nu,x$ replaced by $\phi_\nu\circ\psi_\nu,0$ and $x_1^\nu:=1$, $x_2^\nu:=\infty$ and $y_\nu:=y''$ are satisfied. It follows that $\phi_\nu\circ\psi_\nu$ converges to $y$, uniformly with all derivatives on compact subsets of $S^2\wo\{0\}$. This proves Lemma \ref{le:middle}.
\end{proof}
The next result was used in the proof of Proposition \ref{prop:soft}. Let $(X,d)$ be a metric space. ($d$ is allowed to attain the value $\infty$.) Let $G$ be a topological group and $\rho:G\x X\to X$ a continuous action by isometries. By $\pi:X\to X/G$ we denote the canonical projection. The topology on $X$, determined by $d$, induces a topology on the quotient $X/G$. 

\begin{lemma}[Induced metric on the quotient]\label{le:metr} Assume that $G$ is compact. Then the map $\bar d:X/G\x X/G\to [0,\infty]$ defined by
\[\bar d(\bar x,\bar y):=\min_{x\in\bar x,\,y\in\bar y}d(x,y)\]
is a metric on $X/G$ that induces the quotient topology on $X/G$.
\end{lemma}

\begin{proof}[Proof of Lemma \ref{le:metr}] This follows from an elementary argument. 
\end{proof}


\begin{thebibliography}{99}

\bibitem[CGS]{CGS} K.~Cieliebak, A.R.~Gaio and D.A.~Salamon, \emph{$J$-holomorphic curves, moment maps, and invariants of Hamiltonian group actions}, Internat.~Math.~Res.~Notices  2000, {\bf no. 16}, 831-882.

\bibitem[CGMS]{CGMS} K.~Cieliebak, A.~R.~Gaio and I.~Mundet i Riera, D.~A.~Salamon, \emph{The symplectic vortex equations and invariants of Hamiltonian group actions}, Journal of Symplectic Geometry {\bf 1} (2002), 543-645.

\bibitem[Fr1]{FrPhD} U.~Frauenfelder, \emph{Floer Homology of Symplectic Quotients and the Arnold-Givental Conjecture}, dissertation, ETH Z\"urich, January 2003.

\bibitem[Fr2]{FrArnold} U. Frauenfelder, \emph{The Arnold-Givental Conjecture and Moment Floer Homology}, Int.~Math.~Res.~Not.~{\bf 42} (2004), 2179--2269. 

\bibitem[Ga]{Gaio} A.~R.~Gaio, \emph{J-Holomorphic Curves and Moment Maps}, Ph.D.-thesis, University of Warwick, April 2000.

\bibitem[GS]{GS} A.R.~Gaio and D.A.~Salamon, \emph{Gromov-Witten invariants of symplectic quotients and adiabatic limits}, J.~Symplectic Geom.~{\bf 3}  (2005), {\bf no. 1}, 55--159. 

\bibitem[IS]{IS} S.~Ivashkovich, V.~Shevchishin, \emph{Gromov compactness theorem for $J$-complex curves with boundary}, Internat.~Math.~Res.~Notices 2000, no.~{\bf 22}, 1167--1206.

\bibitem[Ko]{Kontsevich} M.~Kontsevich, \emph{Enumeration of rational curves via torus actions}, 335-368 in R.~Dijkgraaf, C.~Faber, G.~van der Geer, eds., \emph{The moduli space of curves}, Progress in Mathematics {\bf 129}, Birkh\"auser, 1995. 

\bibitem[MS]{MS} D.~McDuff and D.~A.~Salamon, \emph{J-Holomorphic Curves and Symplectic Topology}, AMS Colloquium Publications, {Vol.~\bf 52}, 2004.

\bibitem[Mu1]{MuPhD} I.~Mundet i Riera, \emph{Yang-Mills-Higgs theory for symplectic fibrations}, Ph.D.-thesis, Madrid, April 1999.

\bibitem[Mu2]{MuHam} I.~Mundet i Riera, \emph{Hamiltonian Gromov-Witten invariants}, Topology {\bf 42} (2003), 525-553.

\bibitem[MT]{MT} I.~Mundet i Riera and G.~Tian, \emph{A Compactification of the Moduli Space of Twisted Holomorphic Maps}, Adv.~Math.~{\bf 222} (2009), no.~{\bf 4}, 1117--1196. 

\bibitem[Ott]{Ott} A.~Ott, \emph{Removal of singularities and Gromov compactness for symplectic vortices}, arXiv:0912.2500

\bibitem[Uh]{Uhlenbeck} K. Uhlenbeck, \emph{Connections with $L^p$ Bounds on Curvature}, Comm. Math. Phys. {\bf 83} (1982), 31-42.

\bibitem[Weh]{Wehrheim} K. Wehrheim, \emph{Uhlenbeck Compactness}, EMS, 2004.

\bibitem[WZ]{WZ} C.~Woodward and F.~Ziltener, \emph{Functoriality for Gromov-Witten invariants under symplectic quotients}, preprint, http://www.math.rutgers.edu/~ctw/papers.html 

\bibitem[Zi1]{ZiPhD} F.~Ziltener, \emph{Symplectic Vortices on the Complex Plane and Quantum Cohomology}, Ph.D.-thesis, ETH Z\"urich, May 2006.

\bibitem[Zi2]{ZiA} F.~Ziltener, \emph{The Invariant Symplectic Action and Decay for Vortices}, J.~Symplectic Geom.~{\bf 7} (2009), no.~{\bf 3}, 357--376.

\bibitem[Zi3]{ZiFredholm} F. Ziltener, \emph{A Quantum Kirwan Map, I: Fredholm Theory}, arXiv:0905.4047

\bibitem[Zi3]{ZiConsEv} F. Ziltener, \emph{A Quantum Kirwan Map, III: Conservation of Homology and an Evaluation Map}, in preparation.

\bibitem[Zi4]{ZiTrans} F.~Ziltener, \emph{A Quantum Kirwan Map, IV: Transversality}, in preparation.

\bibitem[Zi5]{ZiExStable}F.~Ziltener, \emph{Stable Maps of Vortices on the Plane and Pseudo-holomorphic Spheres, an Example}, in preparation.

\end{thebibliography}
\end{document}